\documentclass[11pt,a4paper,reqno]{amsart}%article}%
\usepackage{amsfonts}
\usepackage{amsmath,amssymb,amsthm,amsxtra}
\usepackage{float}
\usepackage[colorlinks, linkcolor=blue!72,anchorcolor=orange,
citecolor=red,urlcolor=Emerald, bookmarksopen,bookmarksdepth=2]{hyperref}
\usepackage[usenames,dvipsnames]{xcolor}
\usepackage{enumitem}
\usepackage{geometry,array}
\usepackage{graphicx}
\usepackage{subfigure}
\usepackage{bookmark}
\usepackage{tikz}
\usepackage{quiver}
\usetikzlibrary{positioning, shapes.geometric}
\usepackage{hyperref}
\usepackage{cleveref}

\usepackage{url}

\usetikzlibrary{matrix,positioning,decorations.markings,arrows,decorations.pathmorphing, 
backgrounds,fit,positioning,shapes.symbols,chains,shadings,fadings,calc}
\tikzset{->-/.style={decoration={ markings, mark=at position #1 with
{\arrow{>}}},postaction={decorate}}}
\tikzset{-<-/.style={decoration={ markings, mark=at position #1 with
{\arrow{<}}},postaction={decorate}}}
\usepackage[all]{xy}

\usepackage{pdflscape}

\usepackage[markup=underlined]{changes}
\definechangesauthor[color=NavyBlue]{Zy}
\definechangesauthor[color=Emerald]{Qy}

\theoremstyle{plain}
\newtheorem{theorem}{Theorem}[section]
\newtheorem{lemma}[theorem]{Lemma}
\newtheorem{corollary}[theorem]{Corollary}
\newtheorem{proposition}[theorem]{Proposition}
\theoremstyle{definition}
\newtheorem{definition}[theorem]{Definition}
\newtheorem{construction}[theorem]{Construction}
\newtheorem{example}[theorem]{Example}
\newtheorem{remark}[theorem]{Remark}

\newtheorem{convention}[theorem]{Convention}
\numberwithin{equation}{section}

\numberwithin{equation}{section}

\newtheorem{problem}[theorem]{Problem}
\newtheorem{conjecture}[theorem]{Conjecture}
\usetikzlibrary{decorations.markings}
\tikzset{->-/.style={decoration={ markings, mark=at position #1 with
{\arrow{>}}},postaction={decorate}}}
\tikzset{-<-/.style={decoration={ markings, mark=at position #1 with
{\arrow{<}}},postaction={decorate}}}

\begin{document}
%=========================================================
\title{Riemann-Hilbert problems from rank 3 WKB spectral networks}
\author{Dongjian Wu}
\address{Department of Mathematical Sciences, Tsinghua University, 100084 Beijing, China}
\email{wdj20@mails.tsinghua.edu.cn}
%=========================================================

%=========================================================

\begin{abstract}
We extract cluster structures and establish spectral coordinates from rank 3 WKB spectral networks $\mathcal W(\varphi,\vartheta)$ when zeros of $\varphi(z)$ are almost on a line in the complex plane. Then, we provide solutions to the Riemann-Hilbert problems \cite{B2} defined by these WKB spectral networks, using the spectral coordinates. As an application, we embed spaces of framed polynomial cubic differentials, associated with these WKB spectral networks, into spaces of stability conditions, adopting the approach of \cite{BS}.
\end{abstract}

\maketitle
\tableofcontents
\setlength\parindent{0pt}
\setlength{\parskip}{5pt}

%=========================================================
\section{Introduction}
%=========================================================
In this paper, our primary focus is on the WKB spectral netowrks $\mathcal W(\varphi,\vartheta)$ of rank 3 \cite{HN, GMN2} when zeros of $\varphi(z)$ are almost on a line in the complex plane. The first result involves utilizing compatible abelianization trees to derive cluster structures and constructing spectral coordinates from these WKB spectral networks. Our second result provides solutions to Riemann-Hilbert problems \cite{B2} from these WKB spectral networks. As an application, we intend to establish a holomorphic embedding between the spaces of framed cubic differentials, associated with these WKB spectral networks, and stability conditions, by following the approach outlined in \cite{BS}.

\subsection{WKB spectral networks and clusters}
Spectral networks, introduced in \cite{GMN1,GMN2}, are networks of trajectories on Riemann surfaces subject to specific rules. They serve as a valuable tool for computing BPS degeneracies in four-dimensional $\mathcal N=2$ theories coupled to surface defects, enabling the determination of soliton and particle degeneracies and facilitating the study of moduli spaces of flat $\mathrm{GL}(K, \mathbb C)$ connections. Spectral networks have been shown to have a number of interesting applications, both physical and mathematical, especially to moduli spaces arising in geometry. An intriguing issue proposed in \cite{GMN2, N1} conjectured that there is a potential correspondence between clusters and spectral networks. As one of the primary objectives of this paper, we manage to establish connection for a certian class of rank 3 WKB spectral networks. 

Let $\mathcal W(\varphi_m,\vartheta)$ be a (rank 3) WKB spectral network, where $\varphi_m(z)=z^m+a_1z^{m-1}+\cdots+a_{m-1}z+a_m$ is a complex polynomial of degree $m\ge2$ and $\vartheta$ is a real number (cf. \Cref{WKB spectral networks} or \cite{N1}). We implicitly assume that zeros of $\varphi_m(z)$ are all simple without separate mention. The associated spectral curve $\Sigma(\varphi_m)$ is defined by 
\[
\Sigma(\varphi_m):=\{(x,z): x^3-\varphi_m(z)=0\}\subset\mathbb C^2.
\]
Denote by $\Gamma(\varphi_m)=H_1(\Sigma(\varphi_m),\mathbb Z)$ the first homology group of $\Sigma(\varphi_m)$ and consider the period map
\[
Z:\Gamma(\varphi_m)\to\mathbb C,\quad \gamma\mapsto Z_{\gamma}:=\oint_{\gamma}xdz.
\]
A compatible abelianization tree $\mathrm{T}$ with respect to $\mathcal W(\varphi_m,\vartheta)$ is a collection of oriented arcs in $\overline{\mathbb C}$, subject to certain constriants (see \Cref{compatible trees}). We denote by $\mathcal{T}(\varphi_m,\vartheta)$ the collection of compatible abelianization trees associated with $\mathcal W(\varphi_m,\vartheta)$ and further define
\[
\mathcal{T}^{\circ}(\varphi_m,\vartheta):=\mathcal{T}(\varphi_m,\vartheta)\backslash\{\mathrm{T}_{1,2,3},\mathrm{T}_{2,3,4}\dots,\mathrm{T}_{m+2,m+3,1},\mathrm{T}_{m+3,1,2}\}.
\]
The BPS count $\Omega(\varphi_m,\gamma)\in\mathbb Z$ for a nontrivial class $\gamma\in\Gamma(\varphi_m)$ represents the count of the finite trajectories occurring within $\mathcal W(\varphi_m,\vartheta=\mathrm{arg}\,Z_{\gamma})$. We are particularly interested in the case when zeros of $\varphi_m(z)$ are almost on a line (cf. \Cref{almost on a line}), meaning that there exists another polynomial $\widetilde{\varphi_m}(z)$ obtained by perturbing $\mathrm{Zero}(\varphi_m)$ such that zeros of $\widetilde{\varphi_m}(z)$ are on a line in $\mathbb C$, and $\mathcal W(\varphi_m,\vartheta)$ and $\mathcal W(\widetilde{\varphi_m},\vartheta_1)$ are in the same chamber, i.e., their BPS counts are equal. We say that $\mathcal W(\varphi_m,\vartheta)$ is BPS-free if it has no finite trajectories, and BPS-ful otherwise. 

Some fundamental concepts in cluster theory have been reviewed in \Cref{clusters} (cf. \cite{FZ1, FZ2, FST, FP, S}). We can state our main findings regarding the correspondence between clusters and spectral networks. We consider $\mathcal {T}(\varphi_m,\vartheta)$ as an extended cluster, where $\mathcal{T}^{\circ}(\varphi_m,\vartheta)\subset\mathcal{T}(\varphi_m,\vartheta)$ is the set of cluster variables and its complement is the set of coefficient variables. Based on the discussions in \cite{FP}, we can define a quiver $Q(\varphi_m,\vartheta)$ associated with $\mathcal W(\varphi_m,\vartheta)$ such that the associated cluster algebra $\mathcal A(Q(\varphi_m,\vartheta),\mathcal{T}(\varphi_m,\vartheta))$ is isomorphic to $\mathbb C[Gr(3,m+3)]$. We summarize the result as follows.

\begin{sloppypar}
\begin{theorem}
Let $\mathcal W(\varphi_m,\vartheta)$ be a BPS-free WKB spectral network, where $\varphi_m$ is a polynomial of degree $m\ge2$ whose zeros are almost on a line in $\mathbb C$. Then, there exists a quiver $Q(\varphi_m,\vartheta)$ associated with $\mathcal W(\varphi_m,\vartheta)$ such that the associated cluster algebra $\mathcal A(Q(\varphi_m,\vartheta)$, $\mathcal{T}(\varphi_m,\vartheta))$ is isomorphic to $\mathbb C[Gr(3,m+3)]$. In particular, $\mathcal A(Q(\varphi_m,\vartheta),\mathcal{T}(\varphi_m,\vartheta))$ is of type $A_2, D_4, E_6, E_8$ when $m = 2,3,4,5$ respectively. Additionally, while changing $\vartheta$ along $\mathbb R$, the variations on $\mathcal{T}(\varphi_m,\vartheta)$ will correspond to a sequence of seed mutations on $\mathcal A(Q(\varphi_m,\vartheta),\mathcal{T}(\varphi_m,\vartheta))$.
\end{theorem}
\end{sloppypar}
Note that $Q(\varphi,\vartheta)$ is uniquely defined up to the simultaneous reversal of directions of all edges incident to any subset of connected components of the mutable part of the quiver. We determine the directions of arrows of the quiver $Q(\varphi,\vartheta)$ according to \Cref{convention}. 

Furthermore, based on \Cref{main_thm1} and the proof process of \cite[Theorem 3]{S}, we find that $\mathcal T(\varphi_m,\vartheta)$ can be embedded in a Postnikov diagram $P(\varphi_m,\vartheta)$ of type $(3,m+3)$, i.e., each compatible abelianization tree $\mathrm{T}_{p,q,r}\in\mathcal{T}(\varphi_m,\vartheta)$ corresponds to a 3-subset $(p,q,r)$ in $P(\varphi_m,\vartheta)$, and while changing $\vartheta$ along $\mathbb R$, the variations on $\mathcal{T}(\varphi_m,\vartheta)$ will correspond to a sequence of geometric exchanges on the Postnikov diagram $P(\varphi_m,\vartheta)$, as illustrated in \Cref{Postnikov diagram}.

Let $V=\mathbb C^3$ as a three dimensional complex vector space. We can attach a nonzero vector $x_i = (x_{i,1}, x_{i,2}, x_{i,3})^t \in V$ to each asymptotic direction $l_i$. This allows us to define an $\mathrm{SL}(V)$ invariant in $\mathbb C[Gr(3,m+3)]$ associated with a compatible abelianization tree $\mathrm{T}$, denoted by $A_{\mathrm{T}}$, as described in \Cref{SL(V) invariants}. This invariant follows the notion of $\mathrm{SL}(V)$ invariants for tensor diagrams in \cite{FP}. The spectral coordinate $X_{\mathrm{T}}$ with respect to a compatible abelianization tree $\mathrm{T}\in\mathcal{T}^{\circ}(\varphi_m,\vartheta)$, as a rational function on $(\mathbb{CP}^2)^{m+3}/\mathrm{SL(V)}$, is defined as 
\[
X_{\mathrm{T}}:=\frac{\displaystyle\prod_{a\in Q_1(\varphi_m,\vartheta),t(a)=\mathrm{T}} A_{s(a)}}{\displaystyle\prod_{b\in Q_1(\varphi_m,\vartheta),s(b)=\mathrm{T}} A_{t(b)}},
\]
where $Q_1(\varphi_m,\vartheta)$ denotes the set of arrows in $Q(\varphi_m,\vartheta)$, and $s(a)$ and $t(a)$ denote the source and target of an arrow $a$, respectively (cf. \Cref{spectral coordinate1}). The homology class corresponding to the $\mathrm{SL}(V)$ invariant $X_{\gamma}$, denoted by $[X_{\gamma}]$, is defined as
\[
[X_{\mathrm{T}}]:=\left[\sum_{a\in Q_1(\varphi_m,\vartheta),t(a)=\mathrm{T}}s(a)-\sum_{b\in Q_1(\varphi_m,\vartheta),s(b)=\mathrm{T}}t(b)\right] \in\Gamma(\varphi_m),
\]
which is expressed as a formal linear combination of compatible abelianization trees (cf. Equation \eqref{spectral coordinate0}). For any BPS-free WKB spectral network $\mathcal W(\varphi_m,\vartheta)$, where zeros of $\varphi_m$ are almost on a line, it is shown in \Cref{generator} that $\{[X_{\mathrm{T}}]: \mathrm{T}\in\mathcal{T}^{\circ}(\varphi_m,\vartheta)\}$ generates a basis in $\Gamma(\varphi_m)$. Due to this observation, we can now define the spectral coordinate with respect to any class $\gamma=\sum k_{\mathrm{T}}[X_{\mathrm T}]\in\Gamma(\varphi_m)$ from $\mathcal W(\varphi_m,\vartheta)$ as 
\begin{equation}
\label{spectral coordinate}
X_{\gamma}:=\prod_{\mathrm{T}\in\mathcal{T}^{\circ}(\varphi_m,\vartheta)}X_{\mathrm{T}}^{k_{\mathrm T}},
\end{equation}
which satisfies that $[X_{\gamma}]=\gamma$.

\subsection{Riemann-Hilbert problems}
In \cite{B2}, a class of Riemann-Hilbert problems is studied, which naturally arises in Donaldson-Thomas theory. These problems involve maps from the complex plane to an algebraic torus, with prescribed discontinuites along a collection of rays. They can be understood as the conformal limit of the ones considered by Gaiotto, Moore, and Neitzke in \cite{GMN3} from a physical perspective. In this paper, our focus is on the Riemann-Hilbert problems mentioned in \cite{B2}. The solutions to the Riemann-Hilbert problems we consider, arising from quadratic differentials, have been extensively studied in various references, including \cite{ A1, A2, A3,A4,BM}. Building on their ideas, we will construct solutions to the Riemann-Hilbert problems from WKB spectral networks $\mathcal W(\varphi,\vartheta)$ when zeros of $\varphi(z)$ are almost on a line. This construction heavily relies on the spectral coordinates and the properties of solutions to certain differential equations. 

Some basic definitions in Riemann-Hilbert problems has been reviewed in \Cref{Riemann-Hilbert pro} (cf. \cite{B1, B2}), which include BPS structures $(\Gamma,Z,\Omega)$ and the twisted torus $\mathbb{T}_-$. Let $\mathcal W(\varphi_a,\vartheta)$ be a generic WKB spectral network, where $a=(a_1,\dots,a_m)\in\mathbb C^m$ represents the coefficients of $\varphi_a(z)$, the associated BPS structure $(\Gamma_a,Z_a,\Omega_a)$ is defined as (cf. \Cref{BPS}):
\begin{enumerate}
\item the charge lattice is $\Gamma_a=H_1(\Sigma(\varphi_a),\mathbb Z)$ with its intersection form $\langle-,-\rangle$;
\item the central charge $Z_a:\Gamma_a\to\mathbb C$ is the period map \eqref{central charge};
\item the BPS invariants $\Omega_a(\gamma):=\Omega(\varphi_a,\gamma)$ are either $0$ or $1$, as discussed in \Cref{BPS_counts}.
\end{enumerate}

It is worth noting that when zeros of $\varphi_a(z)$ are almost on a line in $\mathbb C$, the BPS invariants remain constant, and the variation of $\Gamma_a$ is determined by the Gauss-Manin connection as the stability parameter varies. However, in more general cases, the behavior of BPS structures can differ significantly depending on the geometry of the associated moduli spaces and BPS structures. 

Given a ray $r\subset\mathbb C^{\ast}$, we define the half-plane
\[
\mathbb H_r=\{\hbar\in\mathbb C^{\ast}:\hbar=z\cdot v\ \text{with}\ z\in r,\,\mathrm{Re}(v)>0\}\subset\mathbb C^{\ast}.
\]
The Riemann-Hilbert problem from $\mathcal W(\varphi_a,\vartheta)$ depends on the additional choice of an element $\xi\in\mathbb T_-$ as a constant term. The problem is stated as follows:
\begin{problem}\cite[Problem 4.3]{B2}
We aim to find a meromorphic funtion $Y_r:\mathbb H_r\to \mathbb T_-$ for each non-active ray $r\in\mathbb C^{\ast}$, subject to the following conditions:
\begin{enumerate}
\item[(RH1)] given two non-active rays $r_{-},r_{+}$ forming the boundary rays of a convex sector $\triangle\subset\mathbb C^{\ast}$, taken in clockwise order, then we require that
\[
Y_{r_2}(\hbar)=\mathbb S(\triangle)(Y_{r_1}(\hbar)),
\]
as meromorphic functions of $\hbar\in\mathbb H_{r_-}\cap\mathbb H_{r_+}$, where $\mathbb S(\triangle)$ is defined as in \eqref{ts};
\item[(RH2)] for each non-active ray $r\subset\mathbb C^{\ast}$ and each class $\gamma\in\Gamma$, we have 
\[
\mathrm{exp}(Z(\gamma)/\hbar)\cdot Y_{r,\gamma}(\hbar)\to\xi(\gamma),
\]
as $\hbar\to0$ in the half-plane $\mathbb H_r$.
\item[(RH3)] for each class $\gamma$ and any non-active ray $r\subset\mathbb C^{\ast}$, there exists $k>0$ such that 
\[
|\hbar|^{-k}<|Y_{r,y}(\hbar)|<|\hbar|^k,
\]
for all $\hbar\in\mathbb H_r$ with $|\hbar|\gg 0$.
\end{enumerate}
\end{problem}

In order to solve the Riemann-Hilbert problem, it is crucial to analyze the asymptotic properties of the solutions to certain differential equations. In the context of quadratic differentials, this involves analyzing the Schr\"{o}dinger equation, which has been well studied in \cite{W, S1, S2}. For our purpose, we need to consider the following differential equation:
\begin{equation}
\label{ODE0}
\hbar^3y'''(z)+\varphi_a(z)y(z)=0,
\end{equation}
where $a\in\mathbb C^m$ and $\hbar$ is a complex parameter. The asymptotic properties of the solutions to \eqref{ODE0} are derived in \Cref{App:differential equation} using a similar approach as in \cite{S1}.

Let $S^m\subset\mathbb C^m$ be an open submanifold that contains points $a$ where zeros of $\varphi_a(z)$ are almost on a line in $\mathbb C$ and $\mathcal W(\varphi_a,\vartheta)$ is generic for $\vartheta\in\mathbb R$. For any point $a\in S^m$, we can use the unique quadratic refinement $g$, which satisfies the propety that $g(\gamma)=-1$ for every active class $\gamma\in\Gamma_a$, to identify the twisted torus $\mathbb T_{a,-}$ with the standard torus $\mathbb T_a=\mathbb T_{a,+}$. Let $\mathcal W(\varphi_a,\vartheta)$ be a generic spectral network for some $a\in S^m$. Consider the differential equation \eqref{ODE0} with $\hbar=1$. Based on the discussion in \Cref{sec4.1}, we can obtain $m+3$ distinct solutions $y_1,\dots, y_{m+3}$ with the property that for each $k\in\{1,\dots,m+3\}$, $y_k$ is subdominant along the asymptotic direction $l_{k}$. Define
\begin{equation}
\mathcal V(3,n):=\{(v_1,v_2,\dots,v_{n})\in Gr(3,n)\ |\
\mathrm{det}(v_i,v_{i+1},v_{i+2})\ne0, i\in\mathbb Z/n\mathbb Z\}, 
\end{equation}
and set
\[
V(3,n):=\mathcal V(3,n)/\mathrm{PGL(3,\mathbb C)}.
\]

We introduce an element $M(\varphi_a,\vartheta)\in V(3,m+3)$ as follows:
\[
M(\varphi_a,\vartheta):=
\begin{bmatrix}
y_1&y_2&\cdots&y_{m+3}\\
y_1'&y'_2&\cdots&y_{m+3}'\\
y_1''&y''_2&\cdots&y''_{m+3}\\
\end{bmatrix}.
\]
Let $r=e^{i\vartheta_0}\cdot\mathbb R$ be a non-active ray with respect to a WKB spectral network $\mathcal W(\varphi_a,\vartheta)$ for some $a\in S^m$. Under the identification of $\mathbb T_{a,-}$ with $\mathbb T_a$, the required solution map $Y_r:\mathbb H_r\to\mathbb T_a$ is constructed as a composition of the following two maps 
\begin{equation}
\begin{split}
\mathcal X_{r}: V(3,m+3)&\to\mathbb T_a\\
v&\mapsto (X_{a,\vartheta_0}(v):\gamma\mapsto X_{a,\vartheta_0,\gamma}(v)),
\end{split}
\end{equation}
where $X_{a,\vartheta,\gamma}(v)$ is the spectral coordinate with respect to $\gamma\in\Gamma_a$ for the spectral network $\mathcal W(\varphi_a,\vartheta)$ as defined in \eqref{spectral coordinate},
and 
\begin{equation}
F_r: \mathbb H_r\to V(3,m+3), \quad \hbar\mapsto \overline{M(\hbar^{-3}\varphi_a,0)}.
\end{equation}
Note that using the Gauss-Manin connection, the map $X_{a,\vartheta,\gamma}$ makes senses for $\gamma\in\Gamma_a$. 

We examine the reqiurements (RH1)-(RH3) of Problem \ref{main problem} for the solution map $Y_r$ in \Cref{sec.4.3}, \Cref{sec.4.4} and \Cref{sec.4.5} respectively. We conclude with
the following result.
\begin{theorem}
Let $\mathcal W(\varphi_a,\vartheta)$ be a WKB spectral network associated with the BPS structure $(\Gamma_a,Z_a,\Omega_a)$ for some $a\in S^m$.
Then the map $Y_r:\mathbb H_r\to\mathbb T_a$ provides a meromorphic solution to the Riemann-Hilbert problem with constant term $\xi=1\in\mathbb T_{a,-}$.
\end{theorem}

\subsection{Polynomial cubic differentials and stability conditions}
The concept of a stability condition on a triangulated category was introduced in \cite{B1}, arising from the exploration of slope stability of vector bundles over curves and $\Pi$-stability of D-branes in string theory. After the developments of numerous mathematicians, the theory of stability conditions has become an important role in various branches of mathematics, such as mirror symmetry, Donaldson-Thomas invariants and cluster theory, etc. Kontsevich and Seidel proposed that moduli spaces of differentials might have a relationship with spaces of stability conditions on Fukaya-type categories of surfaces.
Subsequently, many results have been achieved in this field, including \cite{BMQS,BS,CHQ, HKK,I,IQ,KQ2}. The majority of these studies focus on developing connections between quadratic differentials and stability conditions. 
In contrast, our objective is to create links between cubic differentials and stability conditions, based on the preceding discussions on clusters and spectral coordinates from spectral networks.

Let $\mathcal C(m)\simeq\mathbb C^{\ast}\times\mathbb C^m$ be the vector space of polynomial cubic differentials $\varphi(z)dz^3$ of degree $m$ with nonzero leading coefficient. Denote by $\mathcal {MC}(m)$ the space of equivalence classes of polynomial cubic differentials in $\mathcal C(m)$. Consider a free abelian group $\Gamma$ of rank $2m-2$. A $\Gamma$-framing of $\varphi(z)dz^3\in\mathcal{MC}(m)$ is an isomorphism of abelian groups
\[
\theta:\Gamma\to\Gamma(\varphi).
\] 
We denote by $\mathcal{MC}(m)^{\Gamma}$ the space of pairs $(\varphi(z)dz^3,\theta)$ consisting of a cubic differential $\varphi(z)dz^3\in\mathcal{MC}(m)$ and a $\Gamma$-framing $\theta$ (cf. \Cref{def.5.1}). The open subspace $\mathcal{MC}(m)_{1}$ of $\mathcal{MC}(m)$ is defined by
\[
\mathcal{MC}(m)_{1}:=\{\varphi(z)dz^3\in\mathcal{MC}(m): \text{zeros of }\, \varphi(z) \text{ are simple and almost on a line}\}.
\] 

Let us fix a base-point $\phi_0=(\varphi_0(z)dz^3,\theta_0)\in\mathcal{MC}(m)^{\Gamma}_1$, assuming that $\mathcal W(\varphi_0,0)$ is BPS-free and generic. Denote by $\mathcal{MC}(m)^{\Gamma}_{1,\ast}$ the connected component in $\mathcal{MC}(m)^{\Gamma}_{1}$ containing $\phi_0$. For simplicity, we refer to the $CY_3$ category associated with $\phi_0$ as $\mathcal D$, which possesses a canonical bounded t-structure with heart denoted as $\mathcal H$ (See \Cref{CY3fcubic} for details). We define $\mathrm{Stab}_0(\mathcal D)$ as the connected component in $\mathrm{Stab}(\mathcal D)$ containing stability conditions with heart $\mathcal H$. Let $\mathcal Aut_0(\mathcal D):=\mathrm{Aut}_{0}(\mathcal D)/\mathrm{Nil}_0(\mathcal D)$, where $\mathrm{Aut}_{0}(\mathcal D)$ is the subgroup of $\mathrm{Aut}(\mathcal D)$ that preserves $\mathrm{Stab}_0(\mathcal D)$ and $\mathrm{Nil}_0(\mathcal D)$ consists of automorphisms in $\mathrm{Aut}_{0}(\mathcal D)$ that act trivially on $\mathrm{Stab}_0(\mathcal D)$. By adapting the approach from \cite{BS} to our cubic case, we can derive the following result.

\begin{theorem}
There is a holomorphic embedding $K$ between complex manifolds that fits into a commutative diagram
\begin{equation}
\label{diagram6.1}
\xymatrix @C=5mm{
\mathcal{MC}(m)^{\Gamma}_{1,\ast}\ar[rr]^{K\qquad\,\,} \ar[rd]_{\pi_1}&&\mathrm{Stab}_{0}(\mathcal D)/\mathcal Aut_0(\mathcal D) \ar[ld]^{\pi}\\
&\mathrm{Hom}_{\mathbb Z}(\Gamma,\mathbb C)&
}
\end{equation}
and which commutes with the $\mathbb C$-actions on both sides.
\end{theorem}

Let $\mathcal{MC}(m)^{\Gamma}_{\ast,s}$ be an open subspace of $\mathcal{MC}(m)^{\Gamma}_{\ast}$ consisting of framed differentials $(\varphi(z)dz^3$, $\theta)$ where zeros of $\varphi(z)$ are distinct. We now propose the following conjecture as a potential avenue for future study: 
\begin{conjecture}
The map $K$ can be holomorphically extended to $\mathcal{MC}(m)^{\Gamma}_{\ast,s}$, denoted as $\widetilde{K}$:
\[
\widetilde{K}: \mathcal{MC}(m)^{\Gamma}_{\ast,s}\to\mathrm{Stab}_0(\mathcal D)/\mathcal Aut_0(\mathcal D).
\]
\end{conjecture}
%=========================================================
\subsection{A further example}
Lastly, we consider a specific WKB spectral network $\mathcal W(\varphi=\frac{1}{2}(-z^3+3z^2+2),\vartheta)$ in \Cref{further example}. Compared to cases where the zeros of polynomials are almost on a line in $\mathbb C$, $\mathcal W(\varphi,\vartheta)$ is in a different chamber, meaning that the BPS counts are different. We provide the cluster structures and spectral coordinates associated to $\mathcal W(\varphi=\frac{1}{2}(-z^3+3z^2+2),\vartheta)$. Note that solutions to the Riemann-Hilbert problem arising from $\mathcal{W}(\varphi,\vartheta)$ can also be constructed using a similar approach as demonstrated in \Cref{solution}. Moreover, $\mathcal{W}(\varphi,\vartheta)$ can also be linked to a stability condition as formulated in \Cref{cubicstab}.

%=========================================================
\subsection*{Acknowledgement}
%=========================================================
I would like to extend my sincere gratitude to my supervisor Yu Qiu for his exceptional suggestions and continuous support throughout this research journey. I am grateful to Dylan Allegretti for his stimulating discussions and constructive comments on the initial manuscript. I also wish to express my appreciation to Andrew Neitzke for his insightful suggestions. Furthermore, I  acknowledge the use of the Mathematica code \cite{N4} for drawing WKB spectral networks in this paper.

%=========================================================
%=========================================================

%=========================================================
%=========================================================

\section{WKB spectral networks}
In this section, we will review some basic concepts related to WKB spectral networks of rank 3 refferring to \cite{GMN1,GMN2, N1}, and discuss the case when zeros of the corresponding polynomials are almost on a line. Throughout this paper, whenever we mention WKB spectral networks, it is implicit that we are referring to the case of rank 3.

\subsection{The spectral curve, homology and periods}
\label{1.1}
Let $\varphi(z)=z^m+a_1z^{m-1}+\dots+a_{m-1}z+a_m$ be a polynomial of degree $m\ge2$ on $\mathbb C$. The \emph{spectral curve} associated to $\varphi(z)$ is defined by
\[
\Sigma(\varphi):=\{(x,z): x^3-\varphi(z)=0\}\subset\mathbb C^2.
\]
The projection map $\pi:(x,z)\mapsto z$ makes $\Sigma(\varphi)$ a branched $3$-fold cover of $\mathbb C$. Unless otherwise specified, we always assume that the polynomials have only simple zeros. Denote by $\mathrm{Zero}(\varphi)$ the set of zeros of $\varphi(z)$. Then the ramification points of $\Sigma(\varphi)$ are all of index $3$ over the $m$ zeros in $\mathrm{Zero}(\varphi)$. Using the Riemann-Hurwitz formula we obtain
\begin{equation}
(\text{genus}(\Sigma(\varphi)),\text{holes}(\Sigma(\varphi)))=
\begin{cases}
(m-2,3) \quad \text{for}\quad 3|m,\\
(m-1,1) \quad \text{otherwise}.
\end{cases}
\end{equation}
The following lattice will be utilized frequently 
\[
\Gamma(\varphi)=H_1(\Sigma(\varphi),\mathbb Z),
\]
and it is easy to see
\[
\mathrm{rank}(\Gamma(\varphi))=2m-2.
\]
Note that $\Gamma(\varphi)$ is naturally equipped with the skew-symmetric intersection pairing $\langle-, -\rangle$, and the \emph{period map}
\begin{equation}
\label{central charge}
Z:\Gamma(\varphi)\to\mathbb C,\quad \gamma\mapsto Z_{\gamma}:=\oint_{\gamma}xdz.
\end{equation}

\subsection{WKB spectral networks}
\label{WKB spectral networks}
Given $\vartheta\in\mathbb R$, a \emph{WKB $\vartheta$-trajectory} $\mathcal W(\varphi,\vartheta)$ is a path $z(t)$ on $\mathbb C$ which satisfies a first-order differential equation:
\begin{equation}
\label{ODE}
(x_i(t)-x_j(t))\frac{\mathrm{d}z(t)}{\mathrm{d}t}=e^{i\vartheta}.
\end{equation}
where $\{(x_i(t),x_j(t))\}$ represents an ordered pair of distinct sheets from $\Sigma(\varphi)$ lying over $z(t)$. Note that the WKB $\vartheta$-trajectories will start from one of the zero point of $\varphi(z)$ and end on the other zero point or extend to the infinity. The construction of the network $\mathcal W(\varphi,\vartheta)$ proceeds in several steps. We start with the $8n$ critical trajectories emanating from the $n$ zeros of $\varphi(z)$, and extend them to $t\to+\infty$ by integrating the differential equation \eqref{ODE}. These trajectories are included in $\mathcal W(\varphi,\vartheta)$. Then we iteratively add more trajectories to $\mathcal W(\varphi,\vartheta)$ as follows.

We examine the intersections between trajectories already included in $\mathcal W(\varphi,\vartheta)$. Each intersection gives rise to three possibilities: the trajectories meet head-on, they intersect in an angle $\frac{\pi}{3}$, or they intersect in an angle $\frac{2\pi}{3}$. We will add a new trajectory only when the intersection angle is $\frac{2\pi}{3}$ as follows. The fact that the intersection angle is $\frac{2\pi}{3}$ implies that the labels of the intersecting trajectories are of the form $(i_1,j_1)=(i,j)$ and $(i_2,j_2)=(j,k)$. We add a new WKB $\vartheta$-trajectory starting from the intersection point, with the label $(i,k)$ (cf. Figure \ref{new trajectory}).
\begin{figure}
\centering
\includegraphics[width=0.2\textwidth]{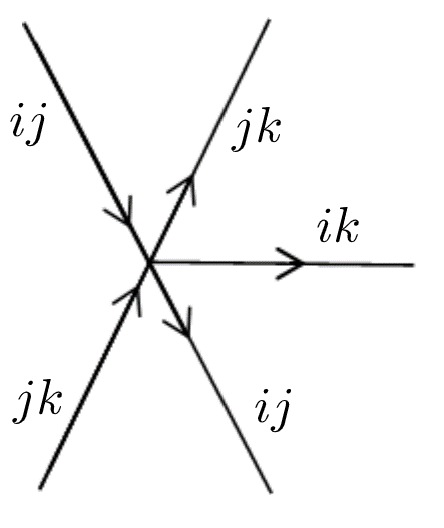}
\caption{A new trajectory carrying label $(i,k)$.}
\label{new trajectory}
\end{figure}

As before, we extend these new trajectories to $t\to+\infty$, which may create new intersections between trajectories. We then repeat the process, allowing these new intersections give rise to new trajectories, and so on, until we have a complete collection of trajectories, which we denote by $\mathcal W(\varphi,\vartheta)$.

We say that the $\mathcal W(\varphi,\vartheta)$ is \emph{BPS-free} if no trajectory runs into another zero of $\varphi(z)$ at some finite $t$. Otherwise, if such a trajectory exists, we say that $\mathcal W(\varphi,\vartheta)$ is \emph{BPS-full}.

\subsection{Asymptotices of WKB spectral networks}
\label{asymptotices}
The behavior of $\mathcal W(\varphi,\vartheta)$ near $z\to\infty$ is characterized by the fact that the WKB $\vartheta$-trajectories approach $(2m+6)$ asymptotic directions. To facilitate the analysis, we now compactify $\mathbb C$ to $\overline{\mathbb C}=\mathbb C\bigsqcup S^1$ (i.e., real blow-up of $\mathbb{CP}^1$ at $z=\infty$). The $2m+6$ asymptotic directions correspond to $2m+6$ marked points on this circle. A marked point is labeled by an ordered pair of sheets $(i,j)$ if all the WKB $\vartheta$-trajectories asymptotic to it satisfy \eqref{ODE}. The WKB $\vartheta$-trajectories asymptotic to a $ij$ marked point will now be labeled by an orderd pair $ij$. When we move counterclockwise from one marked point to the next, one label remains the same while the other changes. This pattern of alternating labels continues in the order $ij,ik,jk,ji,ki,kj,\dots$ for consecutive rays. 

These marked points divide the circle at infinity into $(2m+6)$ arcs. We use the term \emph{initial (final)} to describe an arc at infinity whose two boundary points have the same initial (final) label. There are exactly $(m+3)$ arcs of each type. Each asymptotic direction $l_r$ lies at the midpoint of one of the final arcs. Thus we can label each $l_r$ by the final label for the two nearest boundary rays, which we refer to as the \emph{fading sheet} at $l_r$.

\subsection{Abelianization trees}
\label{compatible trees}
An \emph{abelianization tree compatible with} $\mathcal W(\varphi,\vartheta)$ is defined to be a collection of oriented arcs in $\overline{\mathbb C}$, where each arc is labeled by a sheet of $\Sigma(\varphi)$ and a representation of $\mathrm{SL}(3,\mathbb C)$ (either fundamental $V$ or its dual $V^{\ast}$), with the following constraints:
\begin{figure}
\centering
\includegraphics[width=0.6\textwidth]{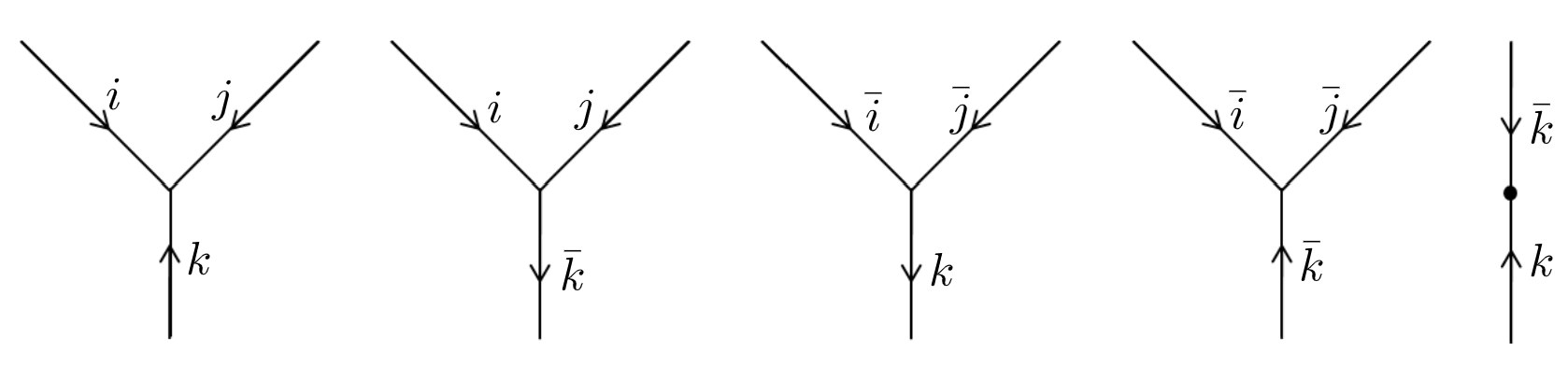}
\caption{Junctions for abelianization trees}
\label{Junctions}
\end{figure}
\begin{itemize}
\item Each arc has two endpoints, with the initial point of each arc lying at one of the $l_r$ on the circle at infinity, or at a junction as shown in Figure \ref{Junctions}. The endpoint of each arc must lie at a junction.
\item If an arc ends at $l_r$, then it carries the representation $V$, and its sheet label matches the fading sheet at $l_r$.
\item The arcs of the abelianization tree do not intersect one another.
\item No arc carrying the label $i$ and representation $V$ should intersect with a trajectory of $\mathcal W(\varphi,\vartheta)$ carrying a label $(i,j)$.
\item No arc carrying the label $i$ and representation $V^{\ast}$ should intersect with a trajectory of $\mathcal W(\varphi,\vartheta)$ carrying a label $(j,i)$.
\end{itemize}

The \emph{nodes} of compatible abelianization trees are defined as the endpoints of the arcs in the first four pictures of Figure \ref{Junctions}. The nodes that lie on the boundary of $\overline{\mathbb C}$ are referred to as \emph{internal nodes}, while the remaining nodes are called \emph{boundary nodes}. 

The \emph{faces} of a compatible abelianization tree are defined as the polygons on $\overline{\mathbb C}$ with the nodes serving as vertices and arcs from abelianization trees as edges. Note that each edge can consist of one or two arcs by definition.

\begin{definition}
A \emph{bipartification} on a compatible abelianization $\mathrm{T}$ is a coloring of its nodes into two colors, black and white, such that no two nodes with the same color are adjacent.
\end{definition}

\begin{lemma}
\label{Bipartification}
Any compatible abelianization tree can be bipartified.
\end{lemma}
\begin{proof}
Let $\mathrm{T}$ be a compatible abelianization tree on a WKB spectral network $\mathcal W(\varphi,\vartheta)$. The statement is clear if $\mathrm{T}$ does not contain any faces with an odd number of edges. Assume that $\mathrm{T}$ has a face $F$ with $(2k+1)$ nodes for some positive integer $k$. We start with the case when $F$ dose not contain two edges arising from a boundary node. Fix an arc $b$ from some edge of $F$ with two adjacent arcs $a,c$. There are only two representations of $b$ that we need to take into consideration, as shown in Figure \ref{choices of arc b}.
\begin{figure}
\centering
\includegraphics[width=0.4\textwidth]{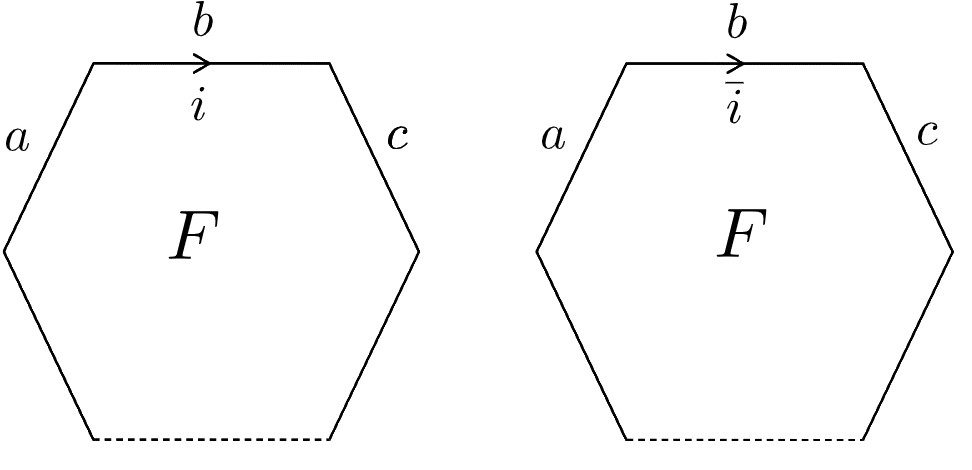}
\caption{Choices of arc $b$}
\label{choices of arc b}
\end{figure}
According to the definition of compatible abelianization trees, there are only two choices of representations of arc $a$ presented in Figure $\ref{choices of arc a}$.
\begin{figure}
\centering
\includegraphics[width=0.4\textwidth]{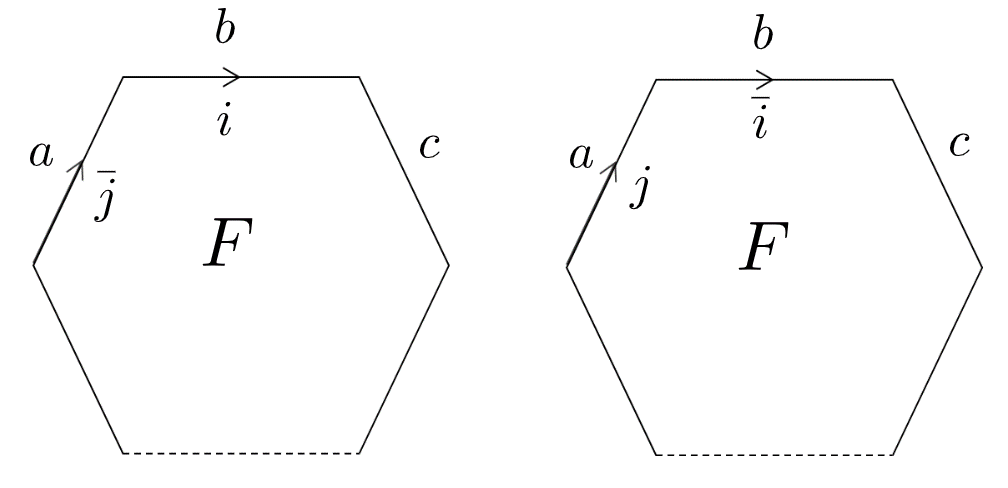}
\caption{Choices of arc $a$}
\label{choices of arc a}
\end{figure}
The depiction of the other adjacent edge of $a$ can also be determined in a similar manner. We continue this process and finally, there are only two possible representations of arc $c$ that need to be considered, as shown in Figure \ref{choices of arc c}. However, both of these choices do not comply with the notion of compatible abelianization trees.
\begin{figure}
\centering
\includegraphics[width=0.4\textwidth]{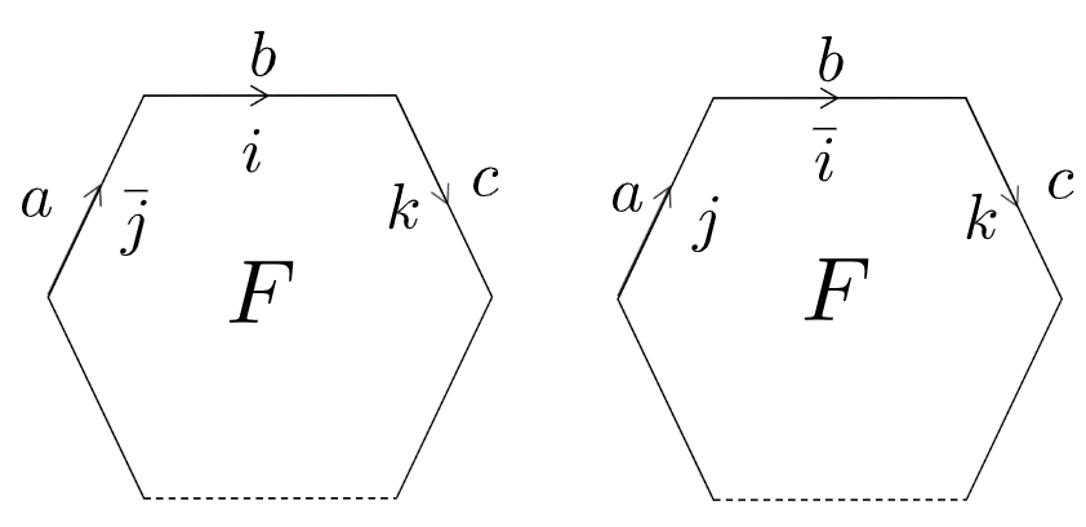}
\caption{Choices of arc $c$}
\label{choices of arc c}
\end{figure}

Now we consider the case when $F$ contains two arcs, denoted by $a,b$ emanating from the same boundary node. Note that at least one arc in $F$ has an orientation that differs from $a$ and denote by $c$ an arc in $F$ that is nearest and opposite in direction to $a$. There are two choices of labelings of $c$ as shown in Figure \ref{choices of c}.
\begin{figure}
\centering
\includegraphics[width=0.4\textwidth]{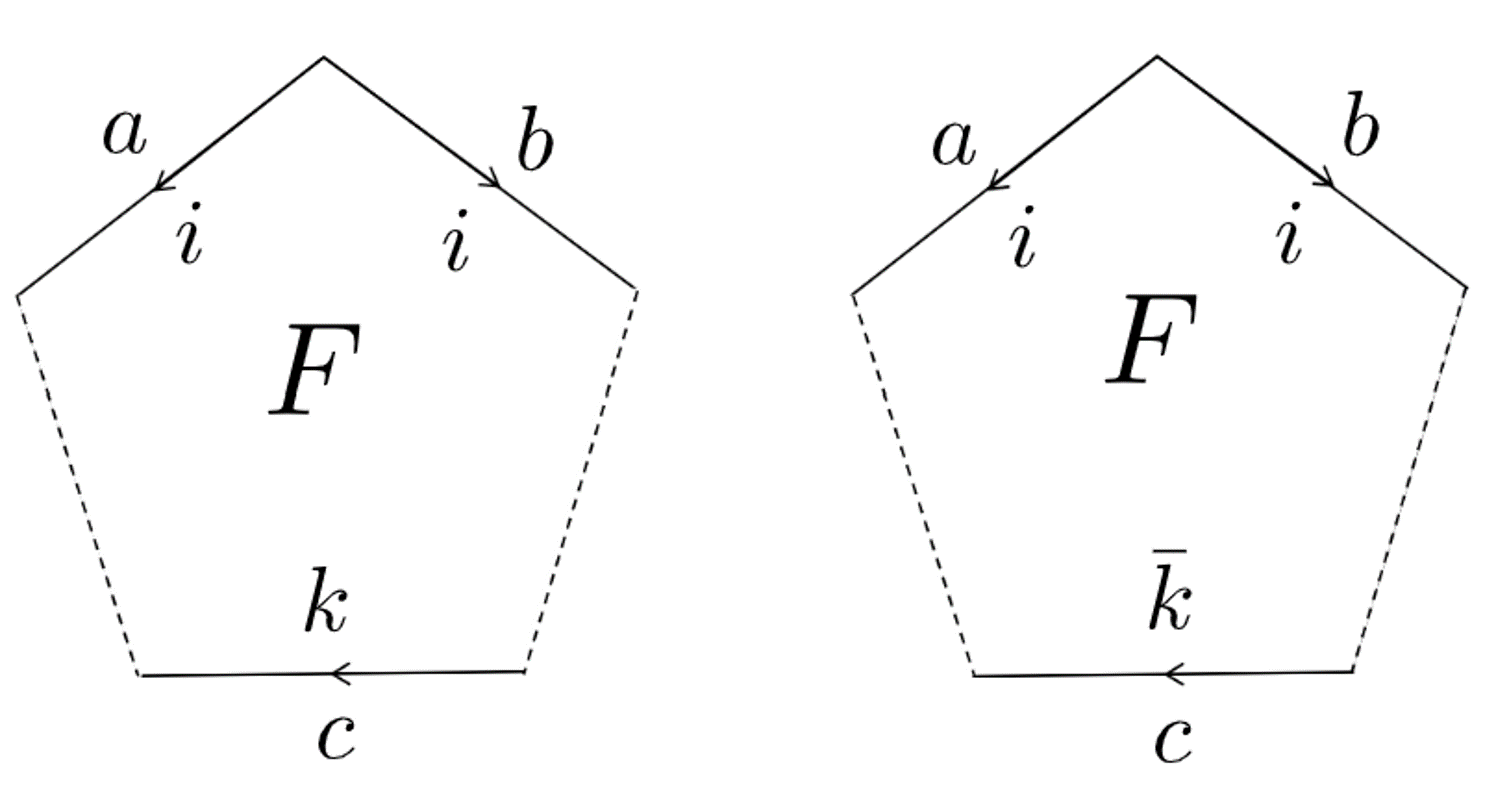}
\caption{Choices of arc $c$}
\label{choices of c}
\end{figure}
By the discussion above we conclude that the number of edges from $b$ to $c$ is odd in the left case, whereas in the right case, the number of edges from $b$ to $c$ is even. As illustrated in Figure \ref{choices of d}, $d$, the adjacent arc in front of $c$, has four options. 
\begin{figure}
\centering
\includegraphics[width=0.6\textwidth]{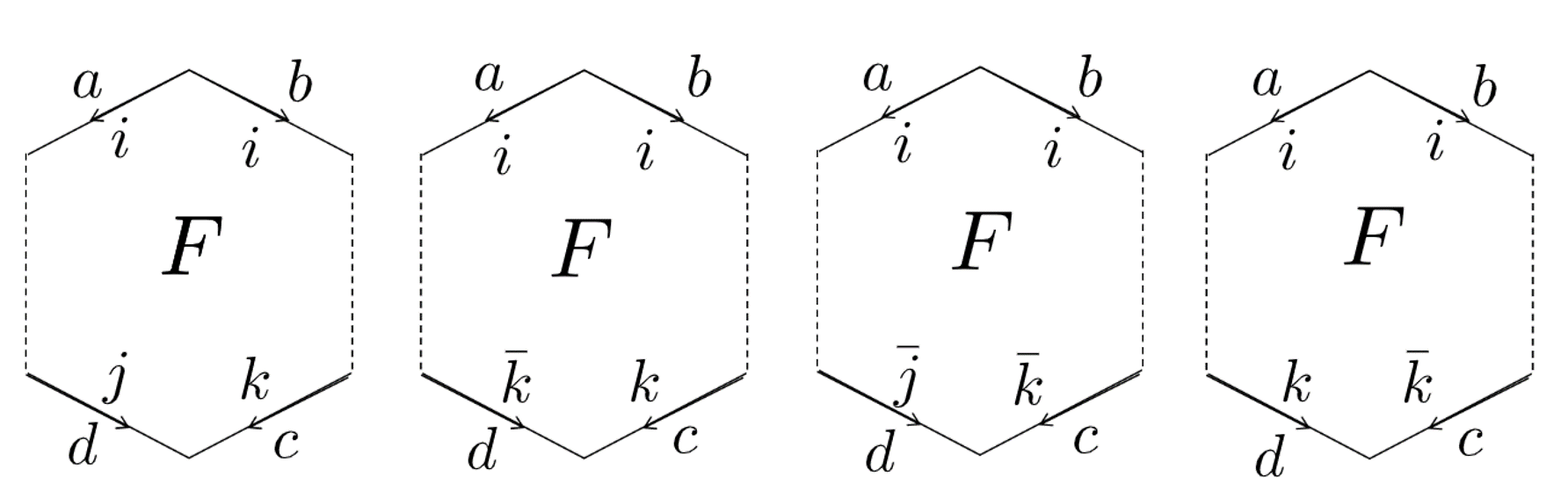}
\caption{Choices of arc $d$}
\label{choices of d}
\end{figure}
In the first situation, it can be concluded that $F$ must have an even number of edges since there are odd edges connecting $a$ to $d$ and $b$ to $c$. This contradicts our initial assumption that $F$ has an odd number of edges. Similarly, in the other situations, we observe that $F$ also has an even number of edges, which contradicts our assumption.

Noting that our previous discussion is unaffected, the branch cuts just impact how those arcs are labeled. We conclude that any compatible abelianization tree can be bipartified.
\end{proof}

Some examples of compatible abelianization trees on various WKB spectral networks are shown in Figures \ref{ex:compatible tree m=3}, \ref{ex:compatible tree m=4}, and \ref{ex:compatible tree m=5}. Note that the compatible abelianization tree in Figure \ref{ex:compatible tree m=5} has two edges coming from a boundary node. 

For simplicity, we denote an abelianization tree with asymptotic directions $l_{i_1},\dots,l_{i_s}$, ordered counterclockwise as $\mathrm T_{i_1,\dots,i_s}$. It should be noted that the asymptotic directions from different compatible abelianization trees may coincide, and in such cases, we will use additional notation. Let $\mathcal T(\varphi,\vartheta)$ denote the collection of all compatible abelianization trees associated with a spectral network $\mathcal W(\varphi,\vartheta)$.

\begin{figure}
\centering
\includegraphics[width=0.39\textwidth]{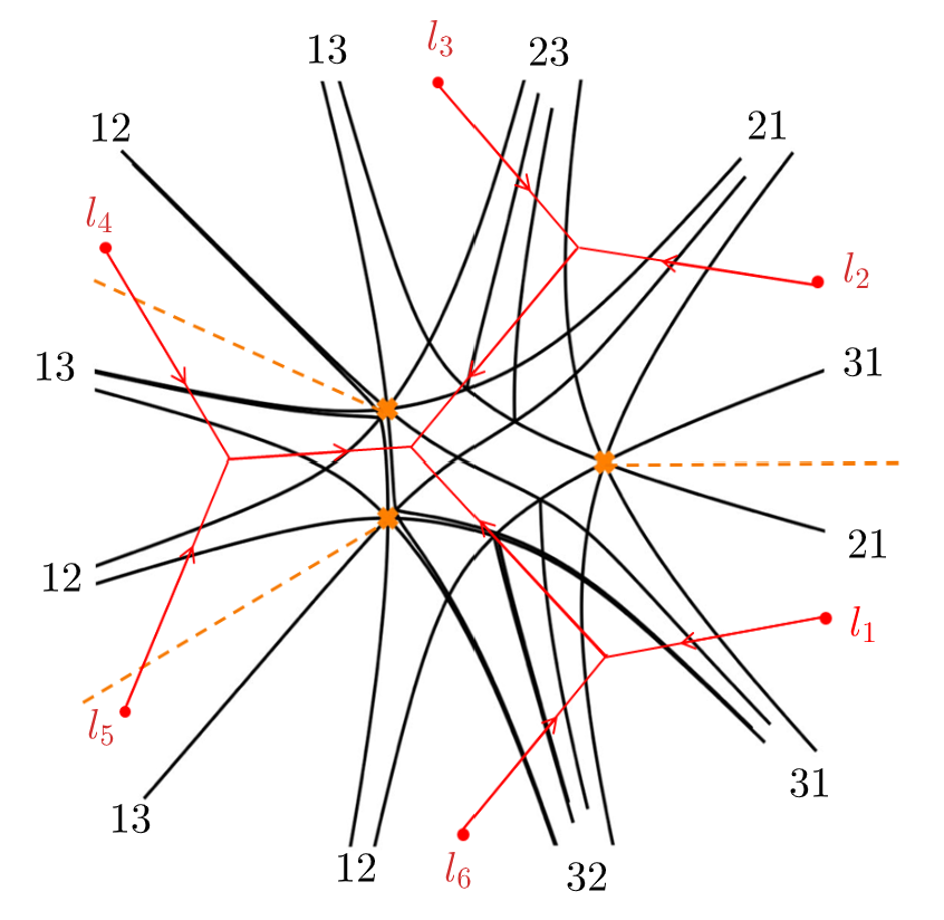}
\caption{A compatible tree when $m=3$.}
\label{ex:compatible tree m=3}
\end{figure}

\begin{figure}
\centering
\includegraphics[width=0.39\textwidth]{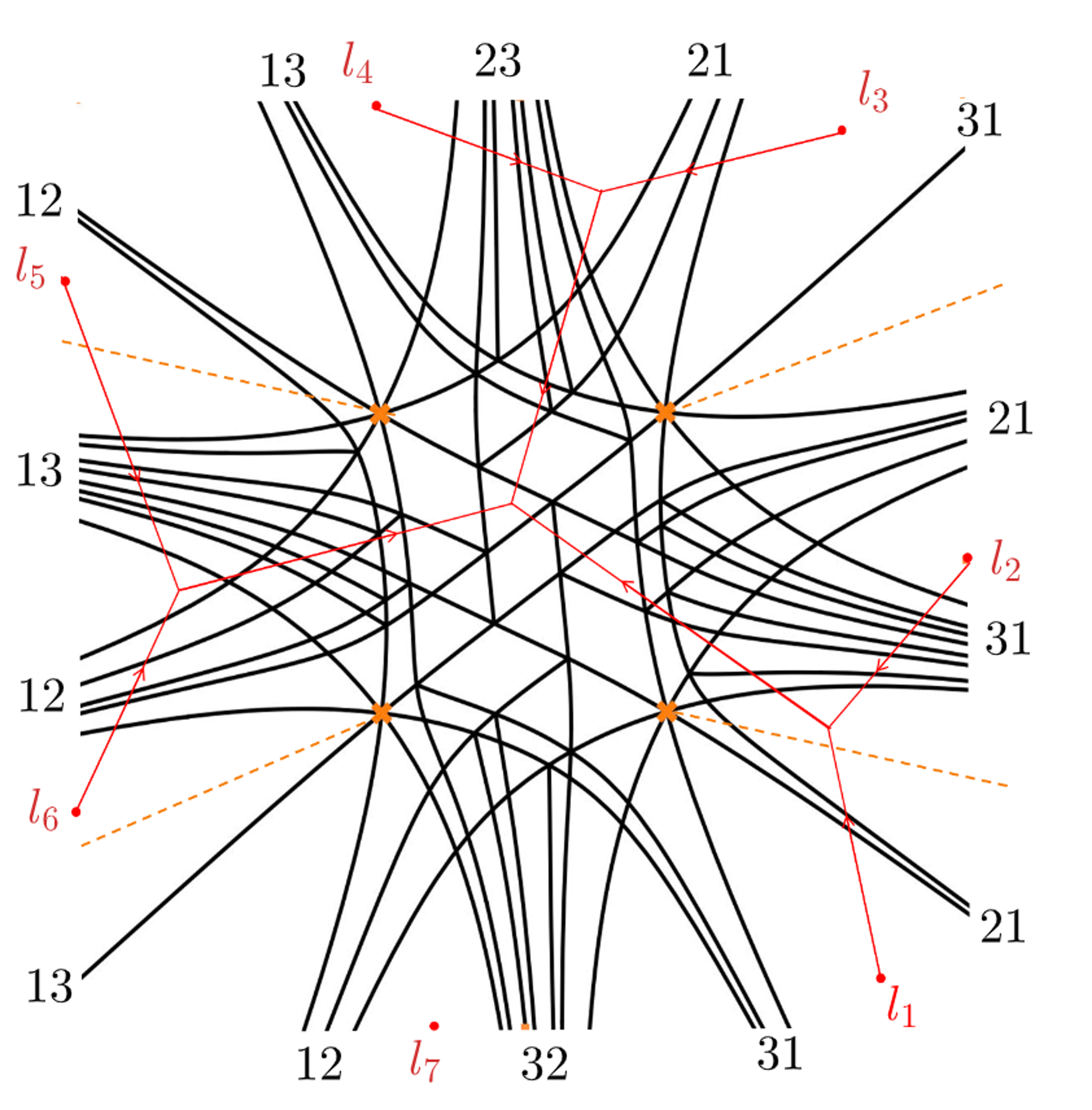}
\caption{A compatible tree when $m=4$.}
\label{ex:compatible tree m=4}
\end{figure}

\begin{figure}
\centering
\includegraphics[width=0.39\textwidth]{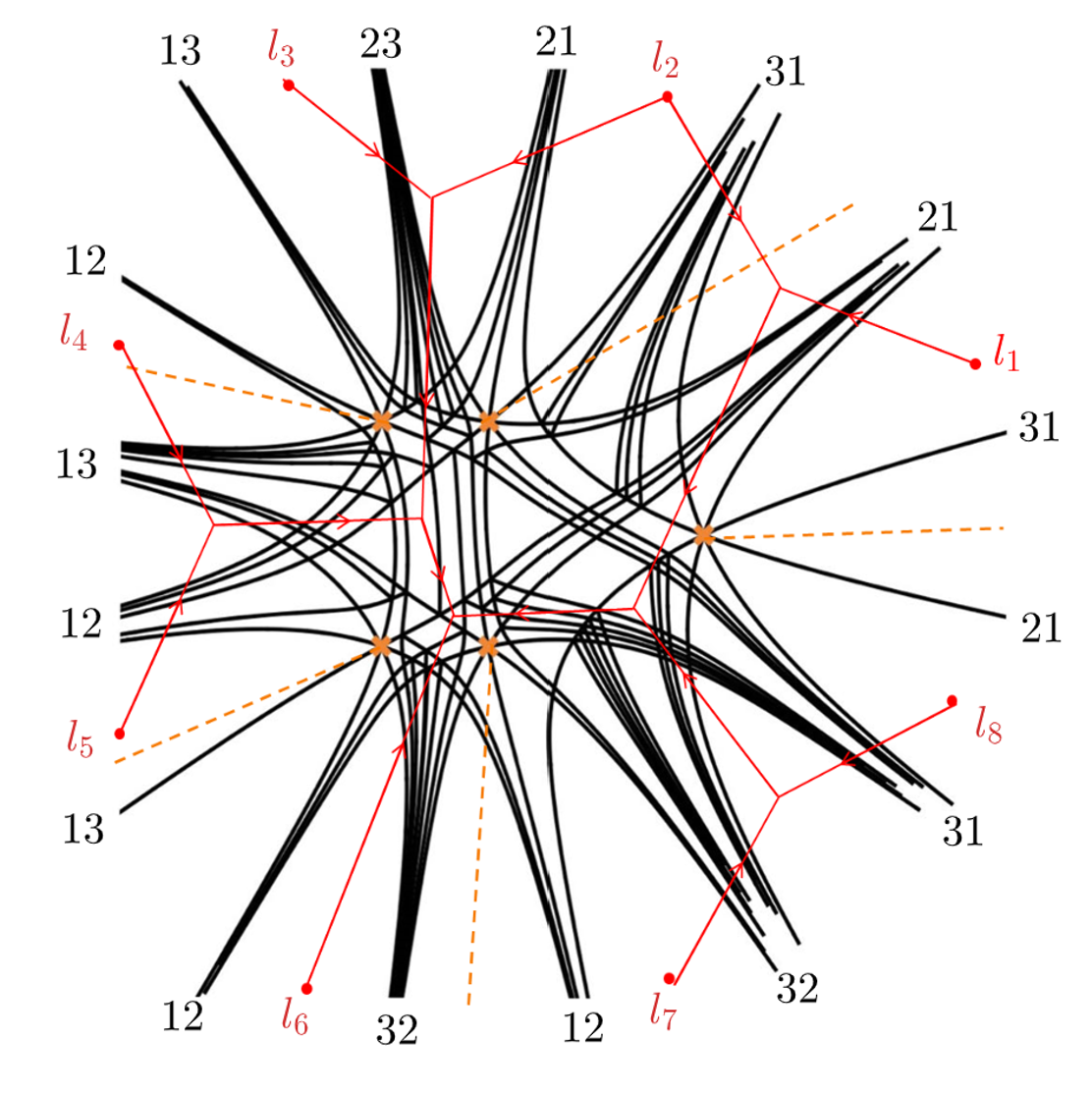}
\caption{A compatible tree when $m=5$.}
\label{ex:compatible tree m=5}
\end{figure}

\subsection{BPS counts}
\label{BPS_counts}
For any WKB $\vartheta$-trajectory $p$, there is a way for us to obtain a canonical lift to a 1-chain $\widetilde p$ on $\Sigma(\varphi)$. Assuming that $p$ is labeled by the pair of sheets $(i,j)$, $\widetilde p$ is obtained by taking the closure of the lift of $p$ to sheet $i$, and adding the orientation-reversal of the closure of the lift of $p$ to sheet $j$. A \emph{finite web of charge} $\gamma\in\Gamma(\varphi)$ is a collection of trajectories $p_l$ such that $\sum_l\widetilde p_l$ is a closed 1-cycle in the homology class $\gamma$. Equation \eqref{ODE} implies that if there exists a finite web of charge $\gamma$ within a WKB network $\mathcal W(\varphi,\vartheta)$, then we must have $\vartheta=\mathrm{arg}\,Z_{\gamma}$. Furthermore, in this case the network $\mathcal W(\varphi,\gamma)$ is BPS-ful. The \emph{BPS count} $\Omega(\varphi,\gamma)\in\mathbb Z$ is a count of the finite webs of charge $\gamma$ which occur within the WKB network $\mathcal W(\varphi,\vartheta=\mathrm{arg} Z_{\gamma})$. The term ``count'' here is defined by an algorithm that is explained in \cite{GMN2}. We will not describe this algorithm here, as it can be quite complex in the general case. Figure \ref{BPS counts} depicts two types of webs: one is a single critical trajectory connecting two zeros of $\varphi_a$ and the other is a three-string junctions; both of these contribute a value of $+1$ to $\Omega(\varphi,\gamma)$.

\begin{figure}
\centering
\includegraphics[width=0.3\textwidth]{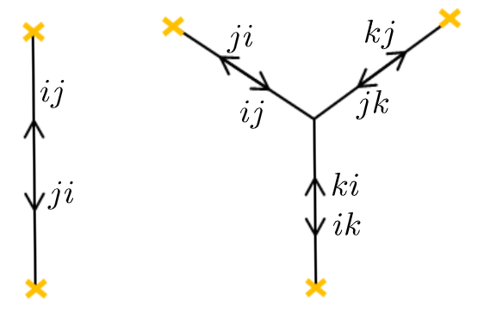}
\caption{Two examples of webs.}
\label{BPS counts}
\end{figure}

It should be mentioned that the approach we used to define the BPS invariants in this context is primarily a geometric one, and it is practical for lower dimensions. However, for more complex examples, it may be difficult to use in practice.

\subsection{WKB spectral networks when zeros are alomst on a line}
We start with a WKB spectral network $\mathcal W(\varphi,\vartheta)$ where zeros of $\varphi(z)$ are on a line in $\mathbb C$. Along this line, we can naturally order the zeros of $\varphi(z)$, which we denote as $\{z_1,z_2,\dots,z_m\}$. Now we choose a basis for $\Gamma(\varphi)$ as follows. For each pair of adjacent zeros $\{z_i,z_{i+1}\}$, we assign a pair of associated classes $\{\gamma_{i,i+1}^{12},\gamma_{i,i+1}^{23}\}$ as shown in Figure \ref{generators}. Then, $\{\gamma_{i,i+1}^{12},\gamma_{i,i+1}^{23}\}_{i=1}^{m-1}$ forms a basis for $\Gamma(\varphi)$. Denote by $\gamma_{i,i+1}^{jk}=-\gamma_{i,i+1}^{kj}$ and $\gamma_{i,i+1}^{13}=\gamma_{i,i+1}^{12}+\gamma_{i,i+1}^{23}$. 

\begin{figure}
\centering
\includegraphics[width=0.3\textwidth]{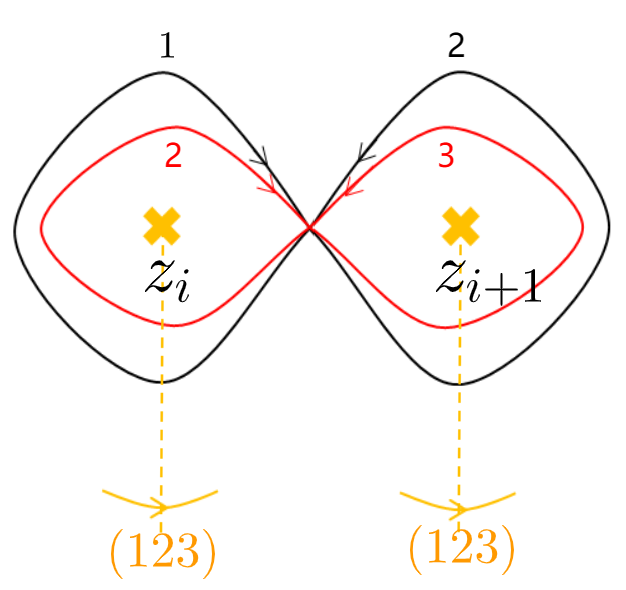}
\caption{Two classes $\gamma^{12}_{i,i+1}$(in black), $\gamma^{23}_{i,i+1}$ (in red) associated with adjacent zeros $\{z_i,z_{i+1}\}$. The dotted lines represent branch cuts, and crossing a branch cut in the direction indicated by the arrow induces the permutation $(123)$ of sheet labels. The numerical labels next to the paths indicate which sheet the paths lie on.}
\label{generators}
\end{figure}
The following result can be obtained by definitions:
\begin{proposition}
\label{dependence}
\begin{enumerate}
\item For any pair of generators $\gamma^{12}_{i,i+1},\gamma^{23}_{i,i+1}$, where $ 1\le i\le m-1$, 
\[
Z_{\gamma^{12}_{i,i+1}}/Z_{\gamma^{23}_{i,i+1}}=e^{2\pi i/3}.
\]
\item For any pair of generators $\gamma^{kl}_{i,i+1},\gamma^{kl}_{j,j+1}$, where $k\ne l\in\{1,2,3\}$ and $i\ne j\in\{1,2,\dots,m-1\}$.
\[
Z_{\gamma^{kl}_{i,i+1}}/Z_{\gamma^{kl}_{j,j+1}}\in\mathbb R.
\]
\end{enumerate}
\end{proposition}

\begin{definition}
\label{generic}
A WKB spectral network $\mathcal W(\varphi,\vartheta)$ is \emph{generic} if, for any two classes $\gamma_1,\gamma_2\in\Gamma(\varphi)$, we have 
\[
\mathbb R\cdot Z_{\gamma_1}=\mathbb R\cdot Z_{\gamma_2}\Longleftrightarrow \mathbb Z\cdot\gamma_1=\mathbb Z\cdot\gamma_2.
\]
\end{definition}

As a result, we obtain that $\mathcal W(\varphi,\vartheta)$ is non-generic if zeros of $\varphi(z)$ are on a line in $\mathbb C$. In addition, we can create a generic WKB spectral network from $\mathcal W(\varphi,\vartheta)$ by perturbing the zeros of $\varphi(z)$ slightly. In this case, the BPS counts $\Omega(\varphi,-)$ associated with $\mathcal W(\varphi,\vartheta)$ is straightforward. Specifically, $\Omega(\varphi,\gamma)=1$ iff 
\[
\gamma\in\left\{\pm\gamma_{i,i+1}^{jk} \,\big|\, i=1,\dots,m-1,\, jk=12,23,13\right\}.
\]

\begin{definition}
\label{almost on a line}
Let $\mathcal W(\varphi,\vartheta)$ be a WKB spectral network. We say that zeros of $\varphi(z)$ are \emph{almost on a line} if there exists another polynomial $\widetilde{\varphi}(z)$ obtained by perturbing $\mathrm{Zero}(\varphi)$ such that zeros of $\widetilde{\varphi}(z)$ are on a line in $\mathbb C$, and $\mathcal W(\varphi,\vartheta)$ and $\mathcal W(\widetilde{\varphi},\vartheta_1)$ are in the same chamber, meaning that their BPS counts are equal. 
\end{definition}

We now provide a geometric model for $\mathcal W(\varphi,\vartheta)$ when zeros of $\varphi(z)$ are almost on a line, which can be taken as a geometric definition of WKB spectral networks for this type. The geometric model provides us a better underding of these WKB spectral networks and offers an advantage in obtaining the collection of compatible abelianization trees.

\begin{construction}[Geometric model]
\label{geometric model}
Let $\mathcal W(\varphi,\vartheta)$ be a WKB spectral network where zeros of $\varphi(z)$ are almost on a line. The geometric model of $\mathcal W(\varphi,\vartheta)$ is constructed as follows:
\begin{enumerate}
\item For each $\vartheta\in\mathbb R$, we label the following $2m+6$ directions as marked points on the boundary of $\overline{\mathbb C}$:
\begin{equation}
\left\{\frac{3}{4}\left(\frac{(2n+1)\pi}{2}+\vartheta-\frac{k\pi}{3}-\frac{\mathrm{arg}\,\varphi'(z_{\ast})}{3}\right): \quad n\in\mathbb Z,\, k=3,4,5,\, z_{\ast}\in \mathrm{Zero}(\varphi)\right\}.
\end{equation}
\item For each point $z_{\ast}\in\mathrm{Zero}(\varphi)$, there will be eight smooth curves starting from $z_{\ast}$ and ending at eight marked points corresponding to the following directions:
\begin{equation}
\label{curves from zeros}
\left\{\frac{3}{4}\left(\frac{(2n+1)\pi}{2}+\vartheta-\frac{k\pi}{3}-\frac{\mathrm{arg}\,\varphi'(z_{\ast})}{3}\right):\quad n\in\mathbb Z,\, k=3,4,5\right\}.
\end{equation}
\item These curves are labeled by pairs of distinct sheets of $\Sigma(\varphi)$ in the following way. For each $k\in\{3,4,5\}$, we can uniquely write it as the sum of a pair of distinct integers $i,j\in\{1,2,3\}$, assuming that $i<j$. For odd $n$ in \eqref{curves from zeros}, we assign the label $(i,j)$ to the curve, and for even $n$ we assign the label $(j,i)$. 
\item When two curves with labelings $(i,j)$ and $(j,i)$ intersect, we add a new curve starting from that intersection point and ending at the middle marked point of the two original curves as shown in \Cref{new trajectory}. 
\item We require that these curves are in a minimum position, meaning that if we perturb these curves relative to their endpoints, the number of intersections between these curves can only increase or remain the same.
\item As $\vartheta\in\mathbb R$ varies, these curves will rotate continuously with respect to their ends. One of the curves from $z_i$ and $z_{i+1}$ will collide and form a new finite curve connecting zeros $z_i$ and $z_{i+1}$ only when 
\[
\vartheta\ \mathrm{mod}\ 2\pi\in\{\mathrm{arg}\pm Z_{\gamma_{i,i+1}^{12}},\mathrm{arg}\pm Z_{\gamma_{i,i+1}^{23}},\mathrm{arg}\pm Z_{\gamma_{i,i+1}^{13}}\}.
\]
\end{enumerate}
It is observed that the curves ending at the same marked point have the same labelings. We denote the corresponding marked point by this labeling. There is a pattern of alternating labels for the marked points, which continues in the order $ij,ik,jk,ki,kj,\dots$. Similarly, we can choose $(m+3)$ asymptotic directions as described in \Cref{asymptotices}.
\end{construction}

\begin{remark}
The concept of abelianization trees compatible with the geometric model of WKB spectral networks for this type can be naturally defined, as described in \Cref{compatible trees}. According to \cite{IKS}, the ends and labelings of curves from the geometric construction of $\mathcal W(\varphi,\vartheta)$ coincide with the corresponding WKB $\vartheta$-trajectoires. Since the intersections between any pair of WKB $\vartheta$-trajectories are at most one, we can conclude that the collection of WKB $\vartheta$-trajectories from $\mathcal W(\varphi,\vartheta)$ are also in a minimal position. Therefore, the collection of compatible abelianization trees from the geometric construction of $\mathcal W(\varphi,\vartheta)$ is $\mathcal{T}(\varphi,\vartheta)$. 
\end{remark}

\begin{example}[Gr(3,5)]
\label{ex:trees on Gr(3,5)}
The geometric model of spectral network $\mathcal W(\frac{1}{2}(-z^2+1),\vartheta)$ with fixed asymptotic directions $\{l_i\}$ are shown in Figure \ref{m=2}.
\begin{figure}
\centering
\includegraphics[width=0.3\textwidth]{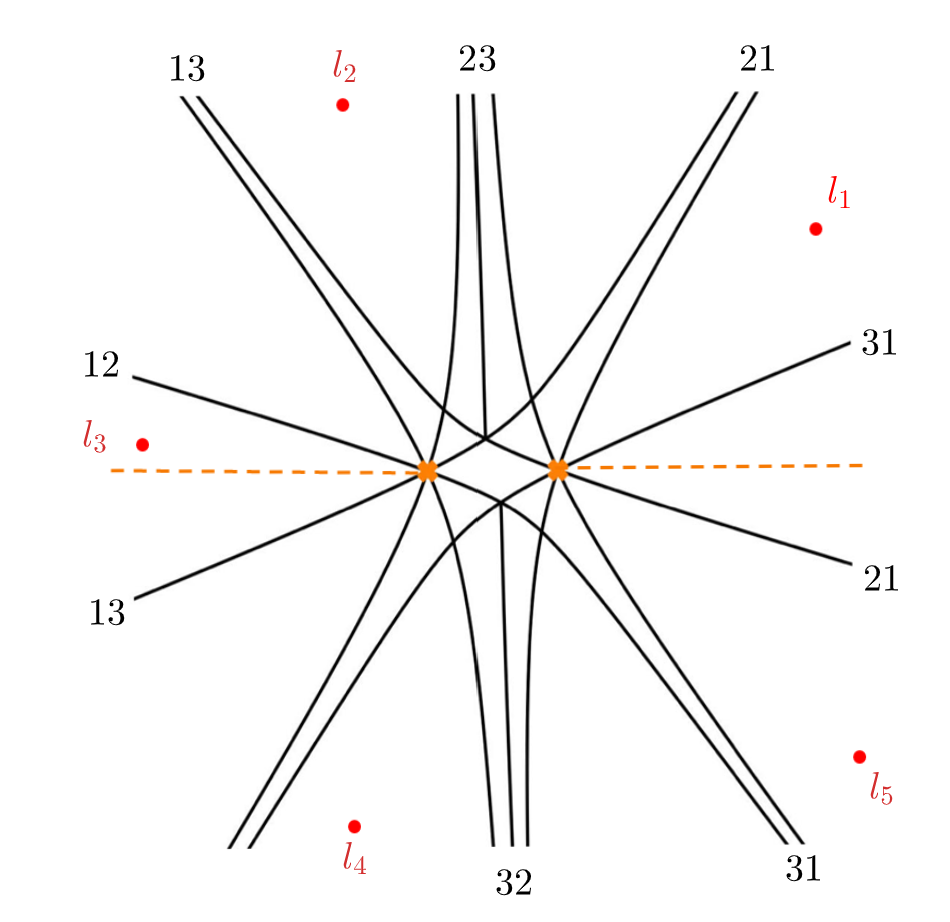}
\caption{The geometric model of the spectral network $\mathcal W(\frac{1}{2}(-z^2+1),\vartheta)$.}
\label{m=2}
\end{figure}
\end{example}

It is worth noting that this geometric model is not applicable in more general cases since the Stokes geometry associated with higher order differential equations is quite different (See \Cref{further example} for an example). We refer to \cite{Ho,GMN1,GMN2} for more discussions.

We consider a WKB spectral network $\mathcal W(\varphi,\vartheta)$ where zeros of $\varphi(z)$ are almost on a line in $\mathbb C$. Assume that the order of zeros is given as $z_1,\dots,z_m$. Let $\varphi_+(z)=(z-z_1)(z-z_2)\cdots(z-z_{m-1})$ and $\varphi_-(z)=(z-z_2)(z-z_3)\cdots(z-z_m)$.

\begin{definition}
\label{degeneration}
The degenerate geometric model of $\mathcal W(\varphi,\vartheta)$ with respect to $\varphi_+$, denoted as $D^+\mathcal W(\varphi,\vartheta)$, is constructed as follows. First, we remove $z_m$ and the curves associated with it from the geometric model of $\mathcal W(\varphi,\vartheta)$. This involves excluding the eight curves starting from $z_m$ and the additional new curves whose starting points lie on these eight curves. Afterward, we modify the behavior of the remaining curves in a way that the resulting degenerate model has the same type of geometric structure as $\mathcal W(\varphi_{+},\vartheta)$, but with different labelings inherited from $\mathcal W(\varphi,\vartheta)$. In this process, we may need to either identify two asymptotic directions or omit one asymptotic direction, depending on the specific situation. Here, we illustrate it in the following case. For instance, if zeros of $\varphi(z)$ are almost on the real line and the labels of the marked points are given as $31,21,23,13,\dots$ in the range of $(0,2\pi)$, with the asymptotic direction $l_1$ being between the first $31$ and $21$, then for even $m$, we need to identify $l_1$ and $l_{m+3}$. However, for odd $m$, there is no $l_{m+3}$ anymore. 

We denote the collection of compatible abelianization trees from $D^{+}\mathcal W(\varphi,\vartheta)$ as $D^{+}\mathcal{T}(\varphi,\vartheta)$. Similarly, we can obtain another degenerate geometric model $D^-\mathcal W(\varphi,\vartheta)$ with the associated collection of compatible abelianization trees $D^{-}\mathcal{T}(\varphi,\vartheta)$.
\end{definition}

\begin{proposition}
\label{compatible trees on finite Gr(3,n)}
Let $\mathcal W(\varphi_m,\vartheta)$ be any BPS-free WKB spectral network where $\varphi_m=(z-z_1)\cdots(z-z_m)$ is a polynomial of degree $m\ge2$ whose zeros are almost on a line in $\mathbb C$. Then there are $3m+1$ distinct compatible abelianization trees in $\mathcal T(\varphi_m,\vartheta)$, which are all in the type of $\mathrm{T}_{p,q,r}$.
\end{proposition}
\begin{proof}
We start with the case where zeros of $\varphi_m(z)$ are on a line in $\mathbb C$. Without loss of generality, assume that the zeros are real numbers $z_1<z_2<\cdots<z_m$. There is a compatible abelianization tree $\mathrm{T}_{p,p+1,p+2}$ as shown in Figure \ref{T(p,p+1,p+2)} for each of the six neighboring asymptotic directions labeled by $ij,kj,ki,ji,jk,ik$. As a result, we obtain $m+3$ abelianization trees $\mathrm{T}_{1,2,3},\mathrm{T}_{2,3,4},\cdots,\mathrm{T}_{m+2,m+3,1},\mathrm{T}_{m+3,1,2}$.
\begin{figure}
\centering
\includegraphics[width=0.3\textwidth]{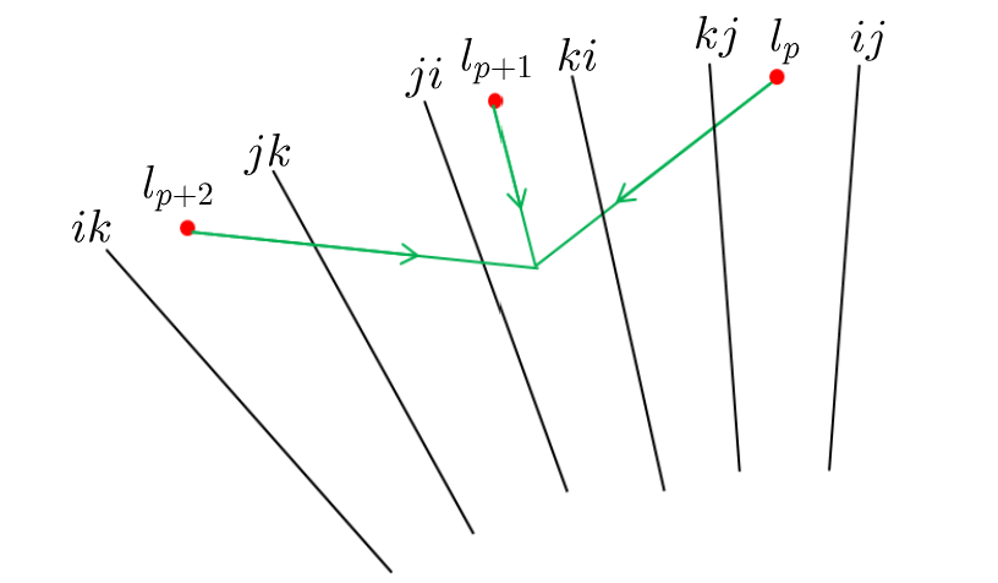}
\caption{The compatible abelianization tree $\mathrm{T}_{p,p+1,p+2}.$}
\label{T(p,p+1,p+2)}
\end{figure}
We only need to show that the set
\[
\mathcal{T}^{\circ}(\varphi_m,\vartheta):=\mathcal{T}(\varphi_m,\vartheta)\backslash\{\mathrm{T}_{1,2,3},\mathrm{T}_{2,3,4}\dots,\mathrm{T}_{m+2,m+3,1},\mathrm{T}_{m+3,1,2}\}
\]
is made up of $2m-2$ compatible abelianization trees in $\mathcal{T}(\varphi,\vartheta)$. Assume that the labels of marked points are given as $31,21,23,13, \dots$ in the range of $(0,2\pi)$, with the asymptotic direction $l_1$ being between the first $31$ and $21$. The cases for the other possible labelings of the marked points can be shown in a similar way.

Under this assumption, the geometric models of $\mathcal W(\varphi_2,\vartheta)$ and $\mathcal W(\varphi_3,\vartheta)$ are as shown in Figures \ref{ex:compatible tree of Gr(3,5)} and \ref{ex:compatible tree of Gr(3,6)}. In the case of $m=2$, we have
\[
\mathcal{T}^{\circ}(\varphi_2,\vartheta)=\{\mathrm{T}_{2,4,5},\mathrm{T}_{1,2,4}\}.
\]
For $m=3$, we have
\[
\mathcal{T}^{\circ}(\varphi_3,\vartheta)=\{\mathrm{T}_{1,2,4},\mathrm{T}_{2,4,5},\mathrm{T}_{1,2,5},\mathrm{T}_{1,4,5}\}.
\]
Note that for ${\mathrm{T}_{2,4,5},\mathrm{T}_{1,2,4}}$ in $\mathcal{T}^{\circ}(\varphi_3,\vartheta)$, one of their edges passes between ${z_1,z_2}$. In this case $\{\mathrm{T}_{2,4,5},\mathrm{T}_{1,2,4}\}$ in $\mathcal{T}^{\circ}(\varphi_3,\vartheta)$ are determined by $\{z_1,z_2\}$.
\begin{figure}
\centering
\includegraphics[width=0.3\textwidth]{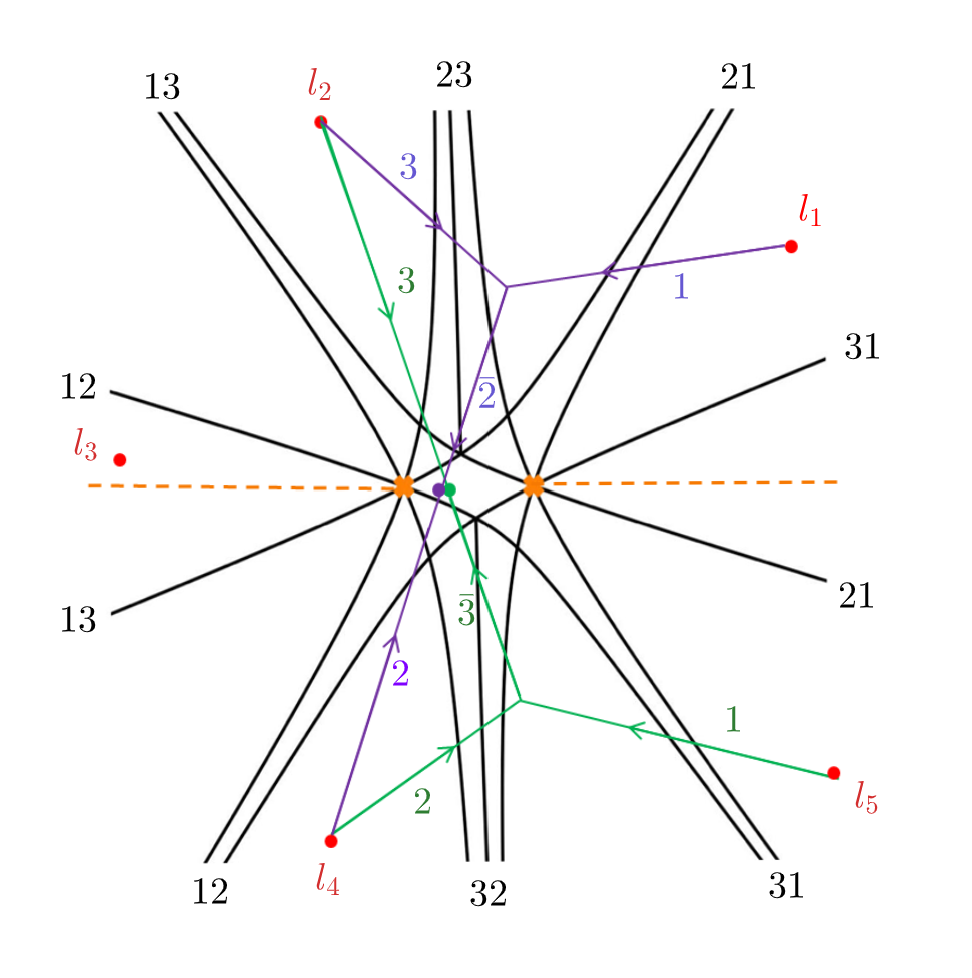}
\caption{The compatible abelianization trees in $\mathcal{T}^{\circ}(\varphi_2,\vartheta).$}
\label{ex:compatible tree of Gr(3,5)}
\end{figure}
\begin{figure}
\centering
\includegraphics[width=0.3\textwidth]{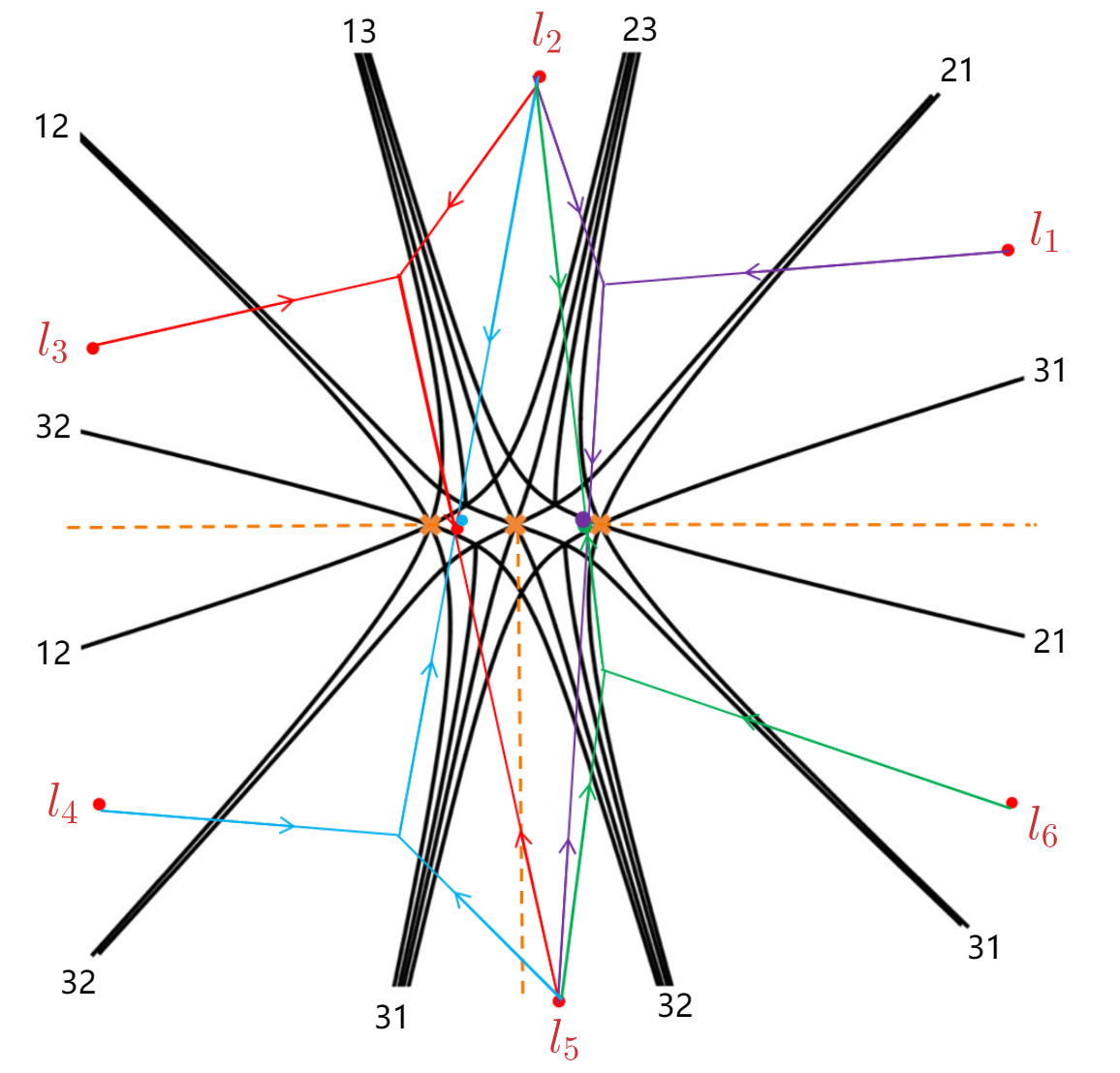}
\caption{The compatible abelianization trees in $\mathcal{T}^{\circ}(\varphi_3,\vartheta_1)$.}
\label{ex:compatible tree of Gr(3,6)}
\end{figure}
In general, we claim that pairs of zeros $\{z_i,z_{i+1}\}$ determine uniquely pairs of compatible abelianization trees in the type of $\mathrm{T}_{p,q,r}$, which form all the compatible abelianization trees in $\mathcal T^{\circ}(\varphi_m,\vartheta)$. We assume that the claim holds for $2\le m<n$. Denote by $\mathcal D^{\pm}\mathcal{T}^{\circ}(\varphi_m,\vartheta)$ the subset of $\mathcal D^{\pm}\mathcal{T}^{\circ}(\varphi_m,\vartheta)$ that excludes the specific type of compatible abelianization trees described in Figure \ref{T(p,p+1,p+2)}. Since the degenerate geometric models $D^{\pm}\mathcal W(\varphi_n,\vartheta)$ have the same type of geometric models as $\mathcal W(\varphi_{\pm},\vartheta)$, where $\varphi_+(z)=(z-z_1)\cdots(z-z_{n-1}),\varphi_-(z)=(z-z_2)\cdots(z-z_n)$, the claim holds for $D^{\pm}\mathcal{T}^{\circ}(\varphi_n,\vartheta)$. Therefore, we obtain
\[
|D^{-}\mathcal T^{\circ}(\varphi_n,\vartheta)|=|D^{+}\mathcal T^{\circ}(\varphi_n,\vartheta)|=2n-4.
\]
By its construction, we have 
\[
D^{-}\mathcal T^{\circ}(\varphi_n,\vartheta)\cup D^+\mathcal T^{\circ}(\varphi_n,\vartheta)\subset\mathcal T^{\circ}(\varphi_n,\vartheta),\quad |D^{-}\mathcal T^{\circ}(\varphi_n,\vartheta)\cap D^+\mathcal T^{\circ}(\varphi_n,\vartheta)|=2n-6.
\]
Thus,
\[
|\mathcal T^{\circ}(\varphi_n,\vartheta)|\ge|D^{-}\mathcal T^{\circ}(\varphi_n,\vartheta)\cup D^+\mathcal T^{\circ}(\varphi_n,\vartheta)|=2n-2.
\]
We now need to show that there are no other compatible abelianization trees in $\mathcal{T}^{\circ}(\varphi_n,\vartheta)$. Assume $\mathrm{T}\in\mathcal{T}^{\circ}(\varphi_n,\vartheta)\backslash (D^{-}\mathcal T^{\circ}(\varphi_n,\vartheta)\cup D^+\mathcal T^{\circ}(\varphi_n,\vartheta))$. In the case of even $n$, $\mathrm{T}$ must include the ends $l_1, l_{n/2+2}, l_{n+3}$. If $\mathrm{T}$ does not contain any of these ends, it can be naturally degenerated into $D^+\mathcal T^{\circ}(\varphi_n,\vartheta)$ or $D^+\mathcal T^{\circ}(\varphi_n,\vartheta)$, which contradicts our assumption. We claim that there are two arcs, denoted as $a_1,a_2$, in $\mathrm T$, where one end of each arc coincides at a junction $p$, and the other ends are $l_1$ and $l_{n+3}$ respectively. If this is not the case, $\mathrm{T}$ can also be degenerated into another compatible abelianization tree $\mathrm T'$ in $D^+\mathcal T^{\circ}(\varphi_n,\vartheta)\}$, which is a contradiction. Let $a_3$ be the other arc ending at $p$. Then $\mathrm{T}$ can be degenerated into another compatible abelianization tree $\mathrm{T}'$ in $D^+\mathcal T^{\circ}(\varphi_n,\vartheta)\}$ by replacing $a_1$ with an arc $a_1'$ with ends $p$ and $l_{2}$, or replacing $a_2$ with an arc $a_2'$ with ends $p$ and $l_{n+2}$. In this case, the ends of $\mathrm{T}'$ contain $l_{n/2+2}$ and $l_{1}$ $(l_{n+3})$, which is also a contradiction. For odd $m$, the ends of $\mathrm{T}$ must contain $l_{1}$, $l_{(n+3)/2}$, $l_{(n+5)/2}$, and $l_{n+3}$ for the same reason as in the even case. Similarly, $\mathrm{T}'$ can also be degenerated into another compatible abelianization tree in $D^{-}\mathcal{T}^\circ(\varphi_n,\vartheta)$ or $D^{+}\mathcal{T}^\circ(\varphi_n,\vartheta)$, which contradicts our assumption.

In fact, the proof procedure remains valid when zeros of $\varphi_m$ are almost on a line, i.e., the claim, that pairs of zeros ${z_i, z_{i+1}}$ uniquely determine pairs of compatible abelianization trees in the form of $\mathrm{T}_{p, q, r}$ which form all the compatible abelianization trees in $\mathcal{T}^{\circ}(\varphi_m, \vartheta)$, can still be demonstrated in a similar manner.

In conclusion, we obtain that $\mathcal{T}(\varphi_m,\vartheta)$ consists of $3m+1$ compatible abelianization trees when zeros of $\varphi_m$ are almost on a line in $\mathbb C$.
\end{proof}
\begin{example}
We have described the collections of compatible abelianization trees $\mathcal T^{\circ}(\varphi_3=(z-1)(z-2)(z-3),\vartheta_1),\mathcal T^{\circ}(\varphi_4=(z-1)(z-2)(z-3)(z-4),\vartheta_2)$ for the WKB spectral networks $\mathcal W(\varphi_3,\vartheta_1)$ and $\mathcal W(\varphi_4,\vartheta_2)$, as shown in Figures \ref{ex:compatible tree of Gr(3,6)}, \ref{ex:compatible tree of Gr(3,7)}.
\end{example}

\begin{figure}
\centering
\includegraphics[width=0.4\textwidth]{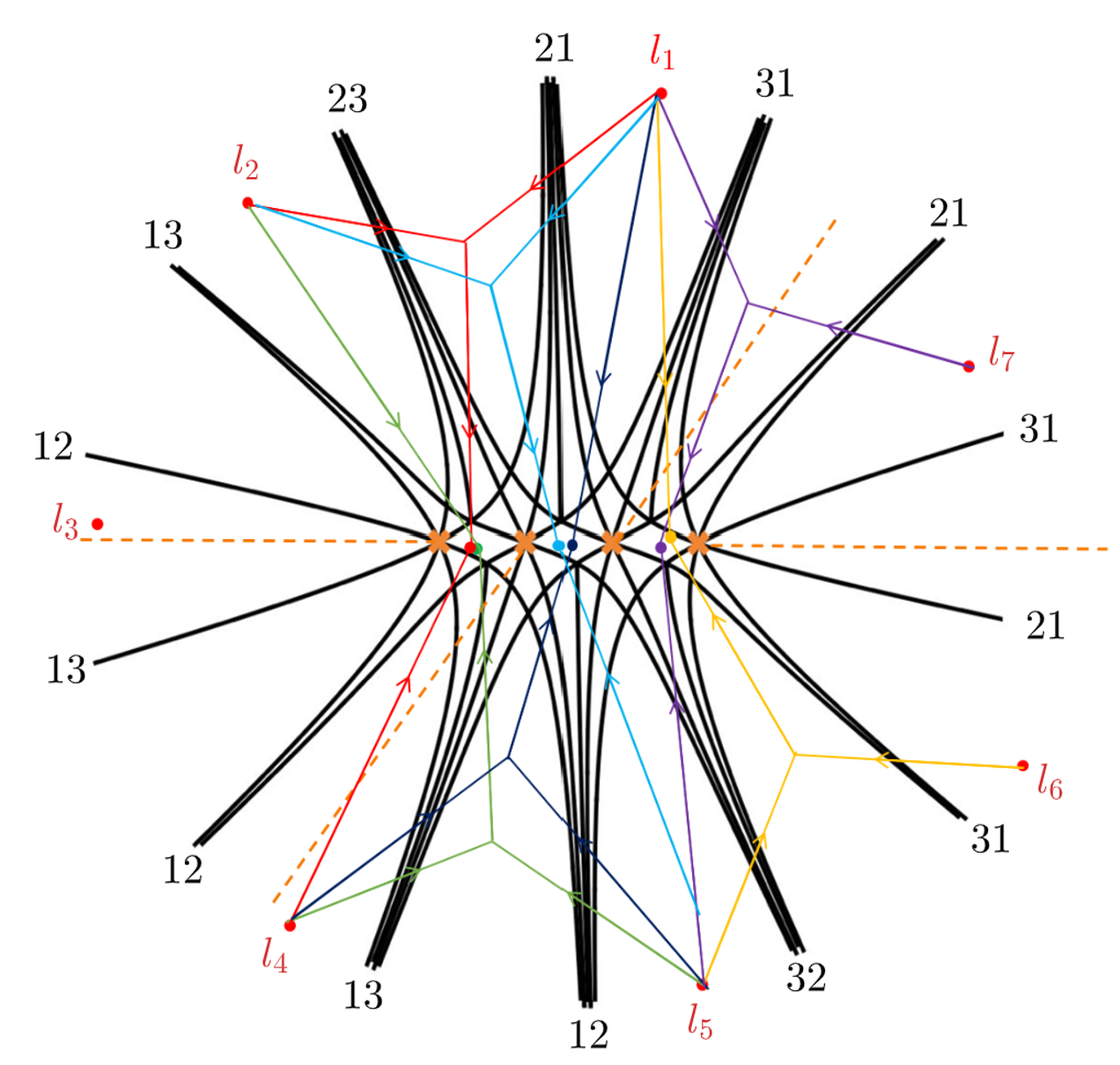}
\caption{The compatible abelianization trees in $\mathrm{T}^{\circ}(\varphi_4,\vartheta_2)$.}
\label{ex:compatible tree of Gr(3,7)}
\end{figure}

\section{Cluster structures and WKB spectral networks}
\label{sec3cluster}
In this section, we will go over some definitions of cluster theory \cite{FP, FST, FZ1, FZ2,S}. Additionally, we will introduce a technique to extract cluster structures from the WKB spectral networks when zeros are almost on a line in $\mathbb C$, and define spectral coordinates from these WKB spectral networks. 
\subsection{Cluster algebras}
\label{clusters}
Commutative rings that have a specific type of combinatorial structure are referred to as cluster algebras. We only require the notion of a cluster algebra of \emph{geometric type}, with the scalar field $\mathbb C$ and \emph{skew-symmetric exchange matrices}. To keep things simple, we refer to these gadgets as cluster algebras. A \emph{quiver} $Q$, a finite oriented loopless graph with no oriented $2$-cycles, determines a cluster algebra. Some of $Q$'s vertices are said to as \emph{mutable}, while the others are referred to as \emph{frozen}.
\begin{definition}
\label{mutation}
Let $z$ represent a mutable vertex in a quiver $Q$. The \emph{quiver mutation} $\mu_z$ transforms $Q$ into the new quiver $Q'=\mu_z(Q)$ through a sequence of three steps. At the beginning of the process, for each pair of directed edges $x\to z\to y$ traveling through $z$, introduce a new edge $x\to y$ (unless $x$ and $y$ are both frozen, in which case do nothing). Reverse the direction of all edges incident to $z$ in the second step. Remove oriented $2$-cycles till you can no longer at the third stage. 
\end{definition}

\begin{definition}
Assume that $\mathcal F$ is a field containing $\mathbb C$. A \emph{seed} in $\mathcal F$ is a pair $(Q,\mathbf z)$ consisting of a quiver $Q$ as described above and a collection $\mathbf z$, called an \emph{extended cluster}, comprising algebraically independent (over $\mathbb C$) elements of $\mathcal F$, one for each vertex of $Q$. The elements of $\mathbf z$ associated with the mutable vertices are referred to as \emph{cluster variables}, which form a \emph{cluster}. The elements associated with the frozen vertices are called \emph{frozen variables}, or \emph{coefficient variables}.

The seed $(Q,\mathbf z)$ is transformed into the new seed $(Q',\mathbf z')=\mu_z(Q,\mathbf z)$ by a \emph{seed mutation} at a mutable vertex linked to a cluster variable $z$ defined as follows. $Q'=\mu_z(Q)$ is the new quiver. The new extended cluster is defined as $\mathbf z'=\mathbf z\cup\{z'\}\setminus\{z\}$, where $z$ is replaced by the new cluster variable $z'$ based on the \emph{exchange relation}.
\begin{equation}
\label{ex1}
zz'=\prod_{z\leftarrow y}y+\prod_{z\to y}y
\end{equation}
\end{definition}

It should be noted that the original seed $(Q,\mathbf z)$ and the mutated seed $(Q',\mathbf z')$ both contain the identical coefficient varibles. By performing a seed mutation at $z'$.
It is simple to verify that $(Q,\mathbf z)$ can be recovered from $(Q',\mathbf z')$.

\begin{definition}[Cluster algebra]
\emph{Seed mutation equivalent} refers to two seeds that can be produced from one another by a sequence of seed mutations. The subring of $\mathcal F$ generated by all elements of all extended clusters of the seed mutation equivalent to $(Q,\mathbf z)$ is defined as the \emph{cluster algebra} $\mathcal A(Q,\mathbf z)$.
\end{definition}
The number of cluster variables in each of a cluster algebra's seeds, or alternatively the number of mutable vertices in each of its quivers, determines the algebra's \emph{rank}.

The \emph{rank} of a cluster algebra $\mathcal A(Q,\mathbf z)$ is the number of cluster variables in each of its seeds, or equivalently the number of mutable vertices in each of its quivers. We consider the cluster algebras with rank $n\ge2$.

When a quiver $Q$ has one of its vertices mutated, the induced subquiver on the set of mutable vertices also does so. The mutable components of the quivers at various seeds of a particular cluster algebra $\mathcal A$ are therefore all seed mutation equivalent. The (cluster) \emph{type} of $\mathcal A$ is determined by this vertex mutation equivalence class.

If a cluster algebra contains a finite number of different seeds, it is said to be of \emph{finite type}. The \emph{finite type classification} is one of the key structural findings in cluster theory:
\begin{theorem}\cite{FZ2}
A cluster algebra is of finite type if and only if it has a quiver whose mutable part is an orientation of disjoint union of Dynkin diagrams. 
\end{theorem}
Remember that the cluster type of a cluster algebra is the only factor that influences the finite type property. $\mathcal A$ is referred to as being of type ``$X_n$" if a cluster algebra of rank $n$ has a quiver whose mutable component is an orientation of a Dynkin diagram of type $X_n$.

The Grassmannian algebra $\mathbb C[{Gr_{3,n}}]$ is the homogeneous coordinate ring of $Gr(3,n)$, a projective subvariety of $\mathbb P^n$ and can also be equivalently defined by 
$\mathbb C[Gr(3,n)]:=\mathbb C[V^n]^{\mathrm{SL}(V)}$, where $V=\mathbb C^3$ as a three dimentional complex vector space. A well-known theorem relating to cluster algebras of finite type and Grassmann algebras is as follows:
\begin{theorem}\cite[Theorem 5]{S}
$Gr(3,6),Gr(3,7)$ and $Gr(3,8)$ are the only Grassmannians $Gr(k,n)$ within the range $2<k\le\frac{n}{2}$, whose homogenous coordinate rings are cluster algebras of finite type.
\end{theorem}

\subsection{Cluster structures from WKB spectral networks}
We first recall the results in \cite{FP}. Let $P_n$ be a convex $n$-gon with counterclockwise-labled vertices $1,2,\cdots,n$. The edge in $P_n$ connecting vertices $i$ and $j$ is denoted by $e_{i,j}$ and the triangle with counterclockwise-labeled vertices $i,j,k$ is denoted by $t_{i,j,k}$. Consider a triangulation $T_n$ on $P_n$ consisting of a set of interior edges $E(T_n)=\{e_{i,j}\}$ and all boundary edges $\{e_{1,2},\dots,e_{n,1}\}$. The concept of \emph{special invariants} $J^{?}_{?}$ on disks has been introduced in \cite[Definition 6.1]{FP}. In this section, the special invariants refer to elements in $\mathbb C[Gr_{3,n}]$ that are determined by the corresponding graphs on disks.

The collection of special invariants associated with $T_n$, denoted by $K(T_n)$, is constructed as follows \cite[Definition 7.1]{FP}:
\begin{enumerate}
\item For each edge $e_{i, j}\in E(T_n)$, include $J_{i,i+1,j}$ and $J_{j,j+1,i}$;
\item For each triangle $t_{i,j,k}$ on $T_n$, include $J_{i,j,k}$. 
\end{enumerate}
The extended cluster $\mathbf z(T_n)$ consists of all special invariants that appear in nonzero elements of $K(T_n)$. The cluster $\mathbf{x}(T_n)=\mathbf z(T_n)\backslash\{J_{i,i+1,i+2}\}_{i=1}^{m+2}$ consists of all non-coefficient invariants in $\mathbf z(T_n)$. 
\begin{definition}
\label{ex3}
The exchange relations for $\mathbf z(T_n)$ associated with $T_n$ include the following identities:
\begin{enumerate}
\item For each triangle $t_{i,j,k}$ of $T_n$, write
\[
J_{i,j,k}J^{i,j,k}=J_{i,i+1,k}J_{k,k+1,j}J_{j,j+1,i}+J_{j,j+1,k}J_{k,k+1,i}J_{i,i+1,j};
\]
\item For each edge $e_{i,j}$ in $T_n$ separating triangles $t_{i,j,k}$ and $t_{i,h,j}$, write
\[
J_{j,j+1,i}J_{k,j,h}=J_{j,j+1,k}J_{i,j,h}+J_{j,j+1,h}J_{i,k,j};
\]
Furthermore, if there is one side of $t_{i,j,k}$ that is exposed, write
\[
J_{j,j+1,i}J^{i,k}_{j,h}=J_{k,k+1,i}J_{i,i+1,j}J_{j,j+1,h}+J_{i,j,h}J^{i,k,j};
\]
\[
J_{i,i+1,j}J^{k,j}_{h,i}=J_{k,k+1,j}J_{i,i+1,h}J_{j,j+1,i}+J_{i,j,h}J^{i,k,j};
\]
\end{enumerate}
\end{definition}
The quiver $Q(T_n)$ associated with a triangulation $T_n$ is constructed so that the exchange relations \eqref{ex1} match the 3-term exchange relations associated with $T_n$. This requirement uniquely determines $Q(T_n)$ up to the simultaneous reversal of the directions of all edges incident to any subset of connected components of the mutable part of the quiver. 

\begin{theorem}[{\cite[Theorem 8.1]{FP}}]
$\mathcal A(Q(T_n),\mathbf z(T_n))\cong\mathbb C[Gr(3,n)]$.
\end{theorem}

Let $\mathcal W(\varphi_m,\vartheta)$ be a BPS-free WKB spectral network where zeros of $\varphi_m(z)$ are on the real line in $\mathbb C$. As before, we assume that the labels of marked points are given as $31,21,23,13, \dots$ in the range of $(0,2\pi)$, with the asymptotic direction $l_1$ lying between the first $31$ and $21$. In this case, every compatible abelianization tree in $\mathcal T(\varphi_m,\vartheta)$ can be written as $\mathrm{T}_{p,q,r}$ for three distinct integers $p,q,r\in[1,m+3]$ by Proposition \ref{compatible trees on finite Gr(3,n)}. Furthermore, based on the proof process of Proposition \ref{compatible trees on finite Gr(3,n)}, we can describe the collection of compatible abelianization trees $\mathcal{T}(\varphi_m,\vartheta)$ explicitly as follows. 
\begin{itemize}
\item $m=2n\ge2$: 
\begin{equation}
\begin{split}
\mathcal{T}(\varphi_m,\vartheta)=&\{\mathrm{T}_{1,2,3},\mathrm{T}_{2,3,4},\dots,\mathrm{T}_{m+3,1,2}\}\cup\{\mathrm{T}_{1,2,m+2},\mathrm{T}_{2,m+2,m+3}\}\\
&\cup\{\mathrm{T}_{k,k+1,m+3-k}, \mathrm{T}_{k,k+1,m+4-k}: k\in[2,n]\cup[n+3,m+1]\},
\end{split}
\end{equation}
\item $m=2n+1\ge3$:
\begin{equation}
\begin{split}
\mathcal{T}(\varphi_m,\vartheta)=&\{\mathrm{T}_{1,2,3},\mathrm{T}_{2,3,4},\dots,\mathrm{T}_{m+3,1,2}\}\cup\{\mathrm{T}_{1,2,m+2},\mathrm{T}_{2,m+2,m+3},\mathrm{T}_{n+1,n+2,n+4},\\&\mathrm{T}_{n+1,n+3,n+4}\}\cup\{\mathrm{T}_{k,k+1,m+3-k}, \mathrm{T}_{k,k+1,m+4-k}: k\in[2,n]\cup[n+4,m+1]\},
\end{split}
\end{equation}
\end{itemize}
Note that $[k,l]=\emptyset$ if $k>l$. Consider the following triangulations $T_{m+3}$ determined by subsets of edges $E(T_{m+3})$ on a convex $(m+3)$-gon $P_{m+3}$:
\begin{itemize}
\item $m=2n\ge2$: 
\[
E({T_{m+3}}):=\{e_{k,m+3-k}:k\in[1,n]\}\cup\{e_{k,m+4-k}:k\in[2,n+1]\};
\]
\item $m=2n+1\ge3$:
\[
E({T_{m+3}}):=\{e_{k,m+3-k}:k\in[1,n+1]\cup\{e_{k,m+4-k}:k\in[2,n+1]\}.
\]
\end{itemize}
The collection of compatible abelianization associated with $T_{m+3}$ is constructed similarly to $K(T_n)$, using the following process:
\begin{enumerate}
\item For each edge $e_{i, j}\in E(T_{m+3})$, include $\mathrm{T}_{i,i+1,j}$ and $\mathrm{T}_{j,j+1,i}$;
\item For each triangle $t_{i,j,k}$ on $T_{m+3}$, include $\mathrm{T}_{i,j,k}$. 
\end{enumerate}
It is straightforward to show that the collection of compatible abelianization trees associated with $T_{m+3}$ is $\mathcal{T}(\varphi_m,\vartheta)$ in each case and we omit the proof here. (See Example \ref{eg:Gr(3,5)} for the case when $m=2$). Therefore, we obtain a natural bijection 
\[
\Psi_{T_{m+3}} : K(T_{m+3})\to \mathcal{T}(\varphi_m,\vartheta), \quad J_{i,j,k}\mapsto\mathrm{T}_{i,j,k}. 
\]
Based on these observations, we can now construct the cluster structures from the WKB spectral network $\mathcal W(\varphi_m,\vartheta)$ in a similar manner. We consider $\mathcal{T}(\varphi_m,\vartheta)$ as an extended cluster, where $\mathcal{T}^{\circ}(\varphi_m,\vartheta)$ is the set of cluster variables and its complement is the set of coefficient variables. The associated exchange relations for $\mathbf z(\varphi_m,\vartheta)$ associated with $\mathcal W(\varphi_m,\vartheta)$ include equations corresponding to equations in \Cref{ex3}. The quiver $Q(\varphi_m,\vartheta)$ associated with $\mathcal W(\varphi_m,\vartheta)$ is constructed in such a way the exchange relations \eqref{ex1} match the 3-term exchange relations associated with $\mathcal W(\varphi_m,\vartheta)$. Therefore, we can immediately conclude that $\mathcal A(Q(\varphi_m,\vartheta),\mathcal{T}(\varphi_m,\vartheta))$ is isomorphic to $\mathbb C[Gr(3,m+3)]$ in this case. Based on the analysis above, we can now state the following result. 

\begin{sloppypar}
\begin{theorem}
\label{main_thm1}
Let $\mathcal W(\varphi_m,\vartheta)$ be a BPS-free WKB spectral network, where $\varphi_m$ is a polynomial of degree $m\ge2$ whose zeros are almost on a line in $\mathbb C$. Then, there exists a quiver $Q(\varphi_m,\vartheta)$ associated with $\mathcal W(\varphi_m,\vartheta)$ such that the associated cluster algebra $\mathcal A(Q(\varphi_m,\vartheta),\mathcal{T}(\varphi_m,\vartheta))$ is isomorphic to $\mathbb C[Gr(3,m+3)]$. In particular, $\mathcal A(Q(\varphi_m,\vartheta), \mathcal{T}(\varphi_m,\vartheta))$ is of type $A_2, D_4, E_6, E_8$ when $m = 2,3,4,5$ respectively. Additionally, while changing $\vartheta$ along $\mathbb R$, the variations on $\mathcal{T}(\varphi_m,\vartheta)$ will correspond to a sequence of seed mutations on $\mathcal A(Q(\varphi_m,\vartheta),\mathcal{T}(\varphi_m,\vartheta))$.
\end{theorem}
\end{sloppypar}

\begin{proof}
In the previous discussion, we defined the quiver $Q(\varphi_m,\vartheta)$ for the case when zeros of $\varphi_m(z)$ are on a line. Now we consider the case when zeros of $\varphi_m(z)$ are almost on a line in $\mathbb C$. By definition, there exists another BPS-free WKB spectral network $\mathcal W(\widetilde{\varphi},\theta_1)$ where zeros of $\widetilde{\varphi}(z)$ are on a line in $\mathbb C$ such that $\mathcal W(\varphi_m,\vartheta)$ can be obtained from $\mathcal W(\widetilde{\varphi},\theta_1)$ by crossing certain BPS-ful locus associated with a subset of adjacent pairs of zeros in $\mathrm{Zero}(\widetilde{\varphi})$. We claim that the variation on $\mathcal{T}(\widetilde{\vartheta},\vartheta_1)$ by crossing a BPS-ful locus corresponds to a sequence of seed mutations on $\mathcal A(Q(\widetilde{\varphi},\vartheta_1), \mathcal T(\widetilde{\varphi},\vartheta_1))$. As before, we assume that zeros of $\widetilde{\varphi}(z)$ are on the real line and the labels of marked points of $\mathcal W(\widetilde{\varphi},\vartheta_1)$ are given as $31,21,23,13, \dots$ in the range of $(0,2\pi)$, with the asymptotic direction $l_1$ being between the first $31$ and $21$. For simplicity, we will only consider counterclockwise crossings.

According to the proof process of Proposition \ref{compatible trees on finite Gr(3,n)}, we can write down the variations of pairs of compatible abelianization trees determined by pairs of adjacent zeros $\{z_i,z_{i+1}\}_{i=1}^{m-1}$ by crossing the associated BPS-ful locus counterclockwise as follows:
\begin{sloppypar}
\begin{enumerate}
\item when $m=2n\ge2$ and $i=1$, $\{\mathrm{T}_{n,n+1,n+3}$, $\mathrm{T}_{n+1,n+3,n+4}\}$ changes into $\{\mathrm{T}_{n,n+1,n+3}$, $\mathrm{T}_{n,n+2,n+3}\}$;
\item when $m=2n\ge2$ and $i=m-1$, $\{\mathrm{T}_{1,2,m+2}$, $\mathrm{T}_{2, m+2, m+3}\}$ changes into $\{\mathrm{T}_{1, 2, m+2}$, $\mathrm{T}_{1,m+1,m+2}\}$;
\item when $m$ $=$ $2n\ge2$ and $2\le i=2k\le m-2$, $\{\mathrm{T}_{n+1-k, n+2+k, n+3+k}$, $\mathrm{T}_{n+1-k, n+2-k, n+3+k}\}$ changes into $\{\mathrm{T}_{n+1-k,n+2+k,n+3+k}$, $\mathrm{T}_{n-k,n+1-k,n+2+k}\}$;
\item when $m=2n\ge2$ and $3\le i=2k-1\le m-2$, $\{\mathrm{T}_{n+2-k,n+2+k,n+3+k}$, $\mathrm{T}_{n+1-k,n+2-k,n+2+k}\}$ changes into $\{\mathrm{T}_{n+1-k,n+1+k,n+2+k},\mathrm{T}_{n+1-k,n+2-k,n+2+k}\}$;
\item when $m=2n+1\ge3$ and $i=1$, $\{\mathrm{T}_{n+1,n+2,n+4},\mathrm{T}_{n+1,n+3,n+4}\}$ changes into $\{\mathrm{T}_{n+1,n+3,n+4},\mathrm{T}_{n,n+1,n+3}\}$;
\item when $m=2n+1\ge3$ and $i=m-1$, $\{\mathrm{T}_{1,2,m+2},\mathrm{T}_{2,m+2,m+3}\}$ changes into $\{\mathrm{T}_{1,m+1,m+2},\mathrm{T}_{1,2,m+2}\}$;
\item when $m=2n+1\ge3$ and $2\le i=2k\le m-2$,  $\{\mathrm{T}_{n+2-k,n+3+k,n+4+k}$,    $\mathrm{T}_{n+1-k,n+2-k,n+3+k}\}$ changes into $\{\mathrm{T}_{n+1-k,n+2+k,n+3+k}$, $\mathrm{T}_{n+1-k,n+2-k,n+3+k}\}$;
\item when $m=2n+1\ge3$ and $3\le i=2k-1\le m-2$, $\{\mathrm{T}_{n+2-k,n+2+k,n+3+k}$, $\mathrm{T}_{n+2-k,n+3-k,n+3+k}\}$ changes into $\{\mathrm{T}_{n+2-k,n+2+k,n+3+k},\mathrm{T}_{n+1-k,n+2-k,n+2+k}\}$;
\end{enumerate}
\end{sloppypar}
By observing the exchange relations in \Cref{ex3} associated with $\mathcal A(Q(\widetilde{\varphi},\vartheta_1)$, $\mathcal T(\widetilde{\varphi},\vartheta_1))$, we can obtain the following mutations with respect to each pair of compatible abelianization trees determined by pairs of adjacent zeros in $\mathrm{Zero}(\widetilde{\varphi})$:
\begin{enumerate}
\item when $m=2n\ge2$ and $i=1$,
\begin{equation*}
\begin{split}
\mathrm{T}_{n+1,n+3,n+4}\mathrm{T}_{n,n+2,n+3}=&\mathrm{T}_{n+2,n+3,n+4}\mathrm{T}_{n,n+1,n+3}+\mathrm{T}_{n,n+3,n+4}\mathrm{T}_{n+1,n+2,n+3};
\end{split}
\end{equation*}
\item when $m=2n\ge2$ and $i=m-1$,
\begin{equation*}
\begin{split}
\mathrm{T}_{2,m+2,m+3}\mathrm{T}_{1,m+1,m+2}=&\mathrm{T}_{m+1,m+2,m+3}\mathrm{T}_{1,2,m+2}+\mathrm{T}_{1,m+2,m+3}\mathrm{T}_{2,m+1,m+2};
\end{split}
\end{equation*}
\item when $m=2n\ge2$ and $2\le i=2k\le m-2$,
\begin{equation*}
\begin{split}
\mathrm{T}_{n+1-k,n+2-k,n+3+k}\mathrm{T}_{n-k,n+1-k,n+2+k}=&
\mathrm{T}_{n-k,n+1-k,n+2-k}\mathrm{T}_{n+1-k,n+2+k,n+3+k}\\
&+\mathrm{T}_{n+1-k,n+2-k,n+2+k}\mathrm{T}_{n-k,n+1-k,n+3+k};
\end{split}
\end{equation*}
\item when $m=2n\ge2$ and $3\le i=2k-1\le m-2$,
\begin{equation*}
\begin{split}
\mathrm{T}_{n+2-k,n+2+k,n+3+k}\mathrm{T}_{n+1-k,n+1+k,n+2+k}=&
\mathrm{T}_{n+1+k,n+2+k,n+3+k}\mathrm{T}_{n+1-k,n+2-k,n+2+k}\\
&+\mathrm{T}_{n+1-k,n+2+k,n+3+k}\mathrm{T}_{n+2-k,n+1+k,n+2+k};
\end{split}
\end{equation*}
\item when $m=2n+1\ge3$ and $i=1$,
\begin{equation*}
\begin{split}
\mathrm{T}_{n+1,n+2,n+4}\mathrm{T}_{n,n+1,n+3}=&\mathrm{T}_{n,n+1,n+2}\mathrm{T}_{n+1,n+3,n+4}+\mathrm{T}_{n+1,n+2,n+3}\mathrm{T}_{n,n+1,n+4};
\end{split}
\end{equation*}
\item when $m=2n+1\ge3$ and $i=m-1$,
\begin{equation*}
\begin{split}
\mathrm{T}_{2,m+2,m+3}\mathrm{T}_{1,m+1,m+2}=&\mathrm{T}_{m+1,m+2,m+3}\mathrm{T}_{1,2,m+2}+\mathrm{T}_{1,m+2,m+3}\mathrm{T}_{2,m+1,m+2}.
\end{split}
\end{equation*}
\item when $m=2n+1\ge3$ and $2\le i=2k\le m-2$,
\begin{equation*}
\begin{split}
\mathrm{T}_{n+2-k,n+3+k,n+4+k}\mathrm{T}_{n+1-k,n+2+k,n+3+k}=&
\mathrm{T}_{n+2+k,n+3+k,n+4+k}\mathrm{T}_{n+1-k,n+2-k,n+3+k}\\
&+\mathrm{T}_{n+1-k,n+3+k,n+4+k}\mathrm{T}_{n+2-k,n+2+k,n+3+k}.
\end{split}
\end{equation*}
\item when $m=2n+1\ge3$ and $3\le i=2k-1\le m-2$,
\begin{equation*}
\begin{split}
\mathrm{T}_{n+2-k,n+3-k,n+3+k}\mathrm{T}_{n+1-k,n+2-k,n+2+k}=&
\mathrm{T}_{n+1-k,n+2-k,n+3-k}\mathrm{T}_{n+2-k,n+2+k,n+3+k}\\
&+\mathrm{T}_{n+2-k,n+3-k,n+2+k}\mathrm{T}_{n+1-k,n+2-k,n+3+k}.
\end{split}
\end{equation*}
\end{enumerate}
Note that these exchange relations are independent of the order in which the seed mutations occur when crossing the BPS-ful locus.
Therefore, the claim regarding $\mathcal A(Q(\widetilde{\varphi},\vartheta_1)$, $\mathcal T(\widetilde{\varphi},\vartheta_1))$ is valid. Similarly, we can derive $\mathcal{T}(\varphi_m,\vartheta)$ by a sequence of seed mutations, denoted by $\mu$, on the seed $(Q(\varphi_1,\vartheta_1)$, $\mathcal{T}(\widetilde{\varphi},\vartheta_1)$. We can now define 
\[
Q(\varphi_m,\vartheta):=\mu(Q(\widetilde{\varphi},\vartheta_1)),
\]
which satisfies
\[
\mathcal A(Q(\varphi_m,\vartheta),\mathcal{T}(\varphi_m,\vartheta))=\mathcal A(\mu(Q(\widetilde{\varphi},\vartheta_1),\mathcal{T}(\widetilde{\varphi},\vartheta_1)))\cong\mathbb C[Gr(3,m+3)].
\]
\end{proof}

It is still worth noting that this requirement defines $Q(\varphi_m,\vartheta)$ up to the simultaneous reversal of directions of all edges incident to any subset of connected components of the mutable part of the quiver. To further determine the directions of arrows of the quiver $Q(\varphi_m,\vartheta)$, we make the following conventions based on the first property of the associated Riemann-Hilbert problem \ref{main problem}. Without loss of generality, assume that zeros of $\varphi_m(z)$ are almost on the real line. Furthermore, by varying $\vartheta$, we can further assume that the labels of marked points of $\mathcal{W}(\varphi_m,\vartheta)$ are given as $31,21,23,13, \dots$ in the range of $(0,2\pi)$, with the asymptotic direction $l_1$ being between the first $31$ and $21$.

\begin{convention}
\label{convention}
The direction of the arrow between $\mathrm{T}_{1,2,3}$ and $\mathrm{T}_{1,2,m+2}$ in $Q(\varphi_m,\vartheta)$ is fixed to be from $\mathrm{T}_{1,2,3}$ to $\mathrm{T}_{1,2,m+2}$.
\end{convention}

Theorem \ref{main_thm1} and the proof procedure of \cite[Theorem 3]{S} allow us to derive the following conclusion directly.

\begin{corollary}
\label{Postnikov diagram}
Let $\mathcal W(\varphi_m,\vartheta)$ be a BPS-free WKB spectral network, where $\varphi_m$ is a polynomial of degree $m\ge2$ whose zeros are almost on a line in $\mathbb C$. Then $\mathcal T(\varphi_m,\vartheta)$ can be embedded in a Postnikov diagram $P(\varphi_m,\vartheta)$ of type $(3,m+3)$, i.e., each compatible abelianization tree $\mathrm{T}_{p,q,r}\in\mathcal{T}(\varphi_m,\vartheta)$ corresponds to a 3-subset $(p,q,r)$ in $P(\varphi_m,\vartheta)$. Additionally, while changing $\vartheta$ along $\mathbb R$, the variations on $\mathcal{T}(\varphi_m,\vartheta)$ will correspond to a sequence of geometric exchanges on the Postnikov diagram $P(\varphi_m,\vartheta)$.
\end{corollary}

\begin{example}[Gr(3,5)]
\label{eg:Gr(3,5)}
We continue with Example \ref{ex:trees on Gr(3,5)}. Note that 
\[
\mathcal{T}^{\circ}(\frac{1}{2}(-z^2+1),\vartheta)=\{\mathrm{T}_{2,4,5},\mathrm{T}_{1,2,4}\}.
\]
The associated exchange relations with respect to mutable vertices are described as follows:
\begin{itemize}
\item $r=4,p=2,q=3,s=1$:
\[
\mathrm{T}_{2,4,5}\mathrm{T}_{1,3,4}=\mathrm{T}_{3,4,5}\mathrm{T}_{1,2,4}+\mathrm{T}_{1,4,5}\mathrm{T}_{2,3,4}.
\]
\item $r=2,p=5,q=1,s=4$:
\[
\mathrm{T}_{2,3,5}\mathrm{T}_{1,2,4}=\mathrm{T}_{1,2,3}\mathrm{T}_{2,4,5}+\mathrm{T}_{2,3,4}\mathrm{T}_{1,2,5}.
\]
\end{itemize}
Based on these exchange relations and following the \Cref{convention}, we can obtain the associated quiver $Q(\varphi_2,\vartheta)$ shown in \Cref{associated quiver}. Regarding the triangulation $T_5$ on the pentagon $P_5$ in Figure \ref{triangulation m=2}, as mentioned in the proof procedure of Theorem \ref{main_thm1}, we have
\[
E(T_5)=\{e_{1,2},e_{2,3},e_{3,4},e_{4,5},e_{5,1},e_{1,4},e_{2,4}\},
\]
and the collection of compatible abelianization associated with $T_{m+3}$ is given as
\[
\{\mathrm T_{1,2,3},\mathrm T_{2,3,4}, \mathrm T_{3,4,5}, \mathrm T_{1,4,5}, \mathrm T_{1,2,5}, \mathrm T_{1,2,4}, \mathrm T_{2,4,5}\},
\]
which is equal to $\mathcal T(\frac{1}{2}(-z^2+1),\vartheta)$.
\begin{figure}
\[
\begin{tikzcd}[column sep=small]
&& {\mathrm T_{2,3,4}} \\
{\mathrm T_{3,4,5}} & {\mathrm T_{2,4,5}} && {\mathrm T_{1,2,4}} & {\mathrm T_{1,2,3}} \\
& {\mathrm T_{1,4,5}} && {\mathrm T_{1,2,5}}
\arrow[from=2-4, to=1-3]
\arrow[from=1-3, to=2-2]
\arrow[from=2-5, to=2-4]
\arrow[from=2-2, to=2-4]
\arrow[from=3-2, to=2-2]
\arrow[from=2-2, to=2-1]
\arrow[from=2-4, to=3-4]
\end{tikzcd}
\]
\caption{Associated quiver $Q(\varphi_2,\vartheta)$.}
\label{associated quiver}
\end{figure}
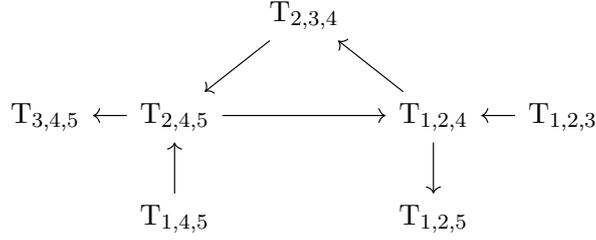
\begin{figure}
\centering
\begin{tikzpicture}
\coordinate[label=above:$1$] (A) at (90:1.5);
\coordinate[label=left:$2$] (B) at (162:1.5);
\coordinate[label=left:$3$] (C) at (234:1.5);
\coordinate[label=right:$4$] (D) at (306:1.5);
\coordinate[label=above:$5$] (E) at (18:1.5);
\draw (A) -- (B) -- (C) -- (D) -- (E) -- cycle;
\draw (A) -- (D);
\draw (B) -- (D);
\end{tikzpicture}
\caption{$P_5$ with a triangulation $T_5$.}
\label{triangulation m=2}
\end{figure}
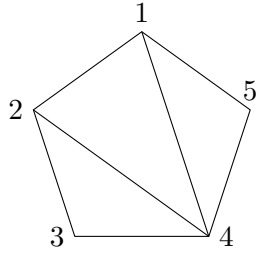

Therefore, we have
\[
\mathcal A(Q(\frac{1}{2}(-z^2+1),\vartheta),\mathcal{T}(\frac{1}{2}(-z^2+1),\vartheta))\cong\mathbb C[Gr(3,5)].
\]
In addition, Figure \ref{Postnikov_diagram} depicts the corresponding Postnikov diagram $P(\frac{1}{2}(-z^2+1),\vartheta)$.
\begin{figure}
\centering
\includegraphics[width=0.3\textwidth]{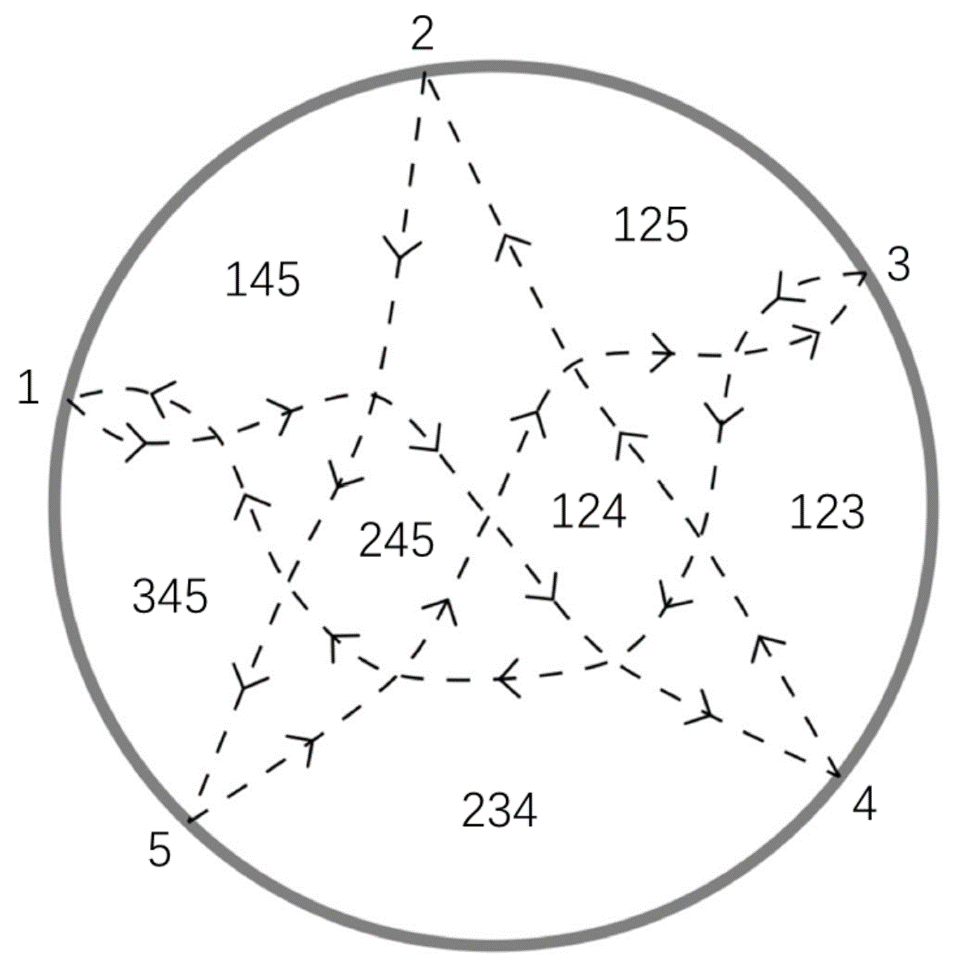}
\caption{The corresponding Postnikov diagram $P(\frac{1}{2}(-z^2+1),\vartheta)$ of type $(3,5)$.}
\label{Postnikov_diagram}
\end{figure}
\end{example}

\subsection{Spectral coordinates from WKB spectral networks}
For a given WKB spectral network $\mathcal W(\varphi,\vartheta)$, each asymptotic direction $l_i$ is now given a nonzero vector $x_i=(x_{i,1},x_{i,2},x_{i,3})^{t}$ in a three dimensional complex vector space $V$. Each compatible abelianizaiton tree has a corresponding $\mathrm{SL}(V)$ invariant in $\mathbb C[Gr(3,m+3)]$ defined below that adheres to the notion of $\mathrm{SL}(V)$ invariants for \emph{tensor diagrams} in \cite{FP}. Before that, a bipartification on each compatible abelianization tree in $\mathcal{T}(\varphi,\vartheta)$ needs to be fixed, which is made possible by Lemma \ref{Bipartification}. Let $\mathrm{T}$ be a compatible abelianization tree with a bipartification, where $a$ out of $a+b$ boundary nodes are white and $c$ out of $c+d$ interior nodes are white. It is clear that $a+3c=b+3d$. We now create a square $(a+3c)\times (b+3d)$ matrix $M(\mathrm T)$ as follows.
\begin{itemize}
\item Each of the white interior ends and the edges incident to the white boundary ends contributes three and one, respectively, to the rows of $M(\mathrm T)$;
\item Each of the black interior ends and the edges incident to the black boundary ends contribute three and one, respectively, to the columns of $M(\mathrm T)$.
\end{itemize}
Using rectangular blocks $M(i,j)$ of sizes $3\times 3$, $3\times 1$, $1\times 3$, and $1\times 1$, the matrix $M(\mathrm T)$ is constructed as follows:
\begin{itemize}
\item $M(i,j)$ is a $3\times3$ identity matrix if $i$ is a white interior end and $j$ is a black interior end incident to $i$;
\item $M(i,j)$ is the column vector $x(v)$ if $i$ is a white interior end and $j$ is an edge incident to $i$ and a black boundary end $v$.
\item $M(i,j)$ is the row vector $x(v)^{t}$ if $j$ is a black interior end and $i$ is an edge incident to $j$ and a white boundary end $v$.
\item otherwise, $M(i,j)$ consists entirely of zeros.
\end{itemize}

\begin{definition}
\label{SL(V) invariants}
The $\mathrm{SL}(V)$ invariant associated to $\mathrm{T}$, denoted by $A_{\mathrm T}$, is defined as $\mathrm{det}(M(\mathrm{T}))$. 
\end{definition}
If there is no confusion, for simplicity, we denote $A_{\mathrm{T}}$ as $A_{\ast}$, where $\ast$ represents the ordered boundary nodes of $\mathrm T$. For example, if $\mathrm T=\mathrm{T}_{p,q,r}$, then $A_{p,q,r}:=A_{\mathrm{T}_{p,q,r}}$.

\begin{example}
Let $\mathcal W(\varphi_m,\vartheta)$ be any BPS-free WKB spectral network where $\varphi_m=(z-z_1)\cdots(z-z_m)$ is a polynomial of degree $m\ge2$ whose zeros are almost on a line in $\mathbb C$. Let $\mathrm T\in\mathcal{T}(\varphi_m,\vartheta)$. According to Proposition \ref{compatible trees on finite Gr(3,n)}, $\mathrm{T}$ can be written as $\mathrm{T}_{p,q,r}$ for three distinct integers, which has a unique trivalent junction and three nodes on the boundary of $\overline{\mathbb C}$. We assign white color to the junction on $\mathrm{T}_{p,q,r}$ to determine the bipartification on $\mathrm{T}$. Then, the associated $\mathrm{SL}(V)$ invariant $A_{p,q,r}$ is 
\[
\mathrm{det}
\begin{bmatrix}
x_{p1}&x_{q1}&x_{r1}\\
x_{p2}&x_{q2}&x_{r2}\\
x_{p3}&x_{q3}&x_{r3}\\
\end{bmatrix}.
\]
\end{example}

Let $\mathcal W(\varphi_m,\vartheta)$ be a BPS-free WKB spectral network, where zeros of $\varphi_m(z)$ are almost on a line in $\mathbb C$. We will now describe the process of obtaining the spectral coordinates $\{X_{\gamma}\}$ with respect to $\Gamma(\varphi_m)$. We begin by defining the spectral coordinates $\{X_{\mathrm{T}}\}$ with respect to $\mathcal{T}^{\circ}(\varphi_m,\vartheta)$. 

\begin{definition}
\label{spectral coordinate1}
The spectral coordinate $X_{\mathrm{T}}$ with respect to a compatible abelianization tree $\mathrm{T}\in\mathcal{T}^{\circ}(\varphi_m,\vartheta)$ is defined as follows:
\[
X_{\mathrm{T}}:=\frac{\displaystyle\prod_{a\in Q_1(\varphi_m,\vartheta),t(a)=\mathrm{T}} A_{s(a)}}{\displaystyle\prod_{b\in Q_1(\varphi_m,\vartheta),s(b)=\mathrm{T}} A_{t(b)}},
\]
where $Q_1(\varphi_m,\vartheta)$ denotes the set of arrows in $Q(\varphi_m,\vartheta)$, and $s(a)$ and $t(a)$ denote the source and target of an arrow $a$, respectively.
\end{definition}

Corollary \ref{Postnikov diagram} yields that $\mathcal T(\varphi_m,\vartheta)$ can be embedded in a Postnikov diagram $P(\varphi_m,\vartheta)$ of type $m+3$. In fact, $\mathrm{T}\in Q_1(\varphi_m,\vartheta)$ can also be obtained from $P(\varphi_m,\vartheta)$. Then, for any mutable vertex $\mathrm{T}\in Q_1(\varphi_m,\vartheta)$, the number of arrows targeting and originating from $\mathrm{T}$ coincide, i.e.,
\[
\sharp\{a\in Q_1(\varphi_m,\vartheta),t(a)=\mathrm{T}\}=\sharp\{b\in Q_1(\varphi_m,\vartheta),s(b)=\mathrm{T}\}.
\]
It is now possible to conisder $X_{\mathrm{T}}$ as a rational function on $(\mathbb{CP}^2)^{m+3}/\mathrm{SL(3,\mathbb C)}$ since it is invariant under $\mathrm{SL}(3,\mathbb C)$ and $(\mathbb C^{\times})^{m+3}$. To assign a generator $\gamma\in\Gamma(\varphi_m)$ to each $X_{\mathrm{T}}$, we need the following setup \cite{N1}.

The lift of each arc $p$ of an abelianization tree $\mathrm T_m\in\mathcal{T}(\varphi_m,\vartheta)$ to a 1-chain $\widetilde p$ on $\Sigma(\varphi_m)$ is canonical: if the sheet $i$ labels $p$, then $\widetilde p$ is the lift of $p$ to a sheet $i$ and $\widetilde p$ has the same orientation as $p$ if $p$ carries representation $V$, otherwise it has the opposite orientation. Consider a formal linear combination of abelianization trees $\Sigma\omega_k\mathrm{T}_k\in\mathbb Z[\mathcal{T}(\varphi_m,\vartheta)]$ with weights $\omega_k\in\mathbb Z$, which we take as a $1$-chain in the singular chain complex $C_{\ast}(\Sigma(\varphi_m))$. By projecting $\Sigma\omega_k\mathrm{T}_k\in\mathbb Z[\mathcal{T}(\varphi_m,\vartheta)]$ to the quotient of the singular chain complex $C_{\ast}(\Sigma(\varphi_m))$ by the subgroup of 0-chains associated to those trivalent junctions in $\{\mathrm{T}_k\}$ pulled back from $\overline{\mathbb C}$, we obtain a class $\left[\Sigma\omega_k\mathrm{T}_k\right]\in\Gamma(\varphi_m)$.

For each spectral coordinate $X_{\mathrm{T}}$ with respect to a compatible abelianization tree $\mathrm{T}\in\mathcal{T}^{\circ}(\varphi_m,\vartheta)$, there is a natural corresponding class
\begin{equation}
\label{spectral coordinate0}
[X_{\mathrm{T}}]:=\left[\sum_{a\in Q_1(\varphi_m,\vartheta),t(a)=\mathrm{T}}s(a)-\sum_{b\in Q_1(\varphi_m,\vartheta),s(b)=\mathrm{T}}t(b)\right] \in\Gamma(\varphi_m).
\end{equation}
The spectral coordinate $X_{\gamma_{\mathrm T}}$ with repect to the class $\gamma_{\mathrm T}$ is naturally defined as $X_{\mathrm T}$. 

\begin{proposition}
\label{generator}
Let $\mathcal W(\varphi_m,\vartheta)$ be a BPS-free WKB spectral network, where $\varphi_m$ is a polynomial of degree $m\ge2$ whose zeros are almost on a line in $\mathbb C$. Then $\{[X_{\mathrm{T}}]: \mathrm{T}\in\mathcal{T}^{\circ}(\varphi_m,\vartheta)\}$ generates a basis in $\Gamma(\varphi_m)$.
\end{proposition}
\begin{proof}
We start with the case when zeros of $\varphi_m(z)$ are on a line in $\mathbb C$. As before, we assume that zeros of $\varphi_m(z)$ are on the real line and the labels of marked points are given as $31,21,23,13, \dots$ in the range of $(0,2\pi)$, with the asymptotic direction $l_1$ lying between the first $31$ and $21$. Then, by the proof process of Theorem \ref{main_thm1}, the seed mutations with respect to each pair of compatible abelianization trees determined by pairs of adjacent zeros have been described in that proof, which enables us to express each spectral coordinate $X_{\mathrm T}$ explicitly in that scenario as follows.
\begin{enumerate}
\item when $m=2n\ge2$ and $i=1$,
\[
X_{\mathrm{T}_{n,n+1,n+3}}=\frac{A_{n,n+1,n+2}A_{n+1,n+3,n+4}}{A_{n,n+1,n+4}A_{n+1,n+2,n+3}},\quad
X_{\mathrm{T}_{n+1,n+3,n+4}}=\frac{A_{n,n+3,n+4}A_{n+1,n+2,n+3}}{A_{n,n+1,n+3}A_{n+2,n+3,n+4}};
\]
\item when $m=2n\ge2$ and $i=m-1$,
\[
X_{\mathrm{T}_{1,2,m+2}}=\frac{A_{1,2,3}A_{2,m+2,m+3}}{A_{2,3,m+2}A_{1,2,m+3}},\quad
X_{\mathrm{T}_{2,m+2,m+3}}=\frac{A_{2,m+1,m+2}A_{1,m+2,m+3}}{A_{1,2,m+2}A_{m+1,m+2,m+3}};
\]
\item when $m=2n\ge2$ and $2\le i=2k\le m-2$,
\[
X_{\mathrm{T}_{n+1-k,n+2+k,n+3+k}}=\frac{A_{n+1-k,n+2-k,n+3+k}A_{n+2+k,n+3+k,n+4+k}}{A_{n+1-k,n+3+k,n+4+k}A_{n+2-k,n+2+k,n+3+k}}, 
\]
\[
X_{\mathrm{T}_{n+1-k,n+2-k,n+3+k}}=\frac{A_{n+1-k,n+2-k,n+2+k}A_{n-k,n+1-k,n+3+k}}{A_{n+1-k,n+2+k,n+3+k}A_{n-k,n+1-k,n+2-k}};
\]
\item when $m=2n\ge2$ and $3\le i=2k-1\le m-2$,
\[
X_{\mathrm{T}_{n+1-k,n+2-k,n+2+k}}=\frac{A_{n+1-k,n+2-k,n+3-k}A_{n+2-k,n+2+k,n+3+k}}{A_{n+2-k,n+3-k,n+2+k}A_{n+1-k,n+2-k,n+3+k}}, 
\]
\[
X_{\mathrm{T}_{n+2-k,n+2+k,n+3+k}}=\frac{A_{n+1+k,n+2+k,n+3+k}A_{n+1-k,n+2-k,n+2+k}}{A_{n+2-k,n+1+k,n+2+k}A_{n+1-k,n+2+k,n+3+k}};
\]
\item when $m=2n+1\ge3$ and $i=1$,
\[
X_{\mathrm{T}_{n+1,n+3,n+4}}=\frac{A_{n+1,n+2,n+4}A_{n+3,n+4,n+5}}{A_{n+1,n+4,n+5}A_{n+2,n+3,n+4}},\quad
X_{\mathrm{T}_{n+1,n+2,n+4}}=\frac{A_{n,n+1,n+4}A_{n+1,n+2,n+3}}{A_{n,n+1,n+2}A_{n+1,n+3,n+4}};
\]
\item when $m=2n+1\ge3$ and $i=m-1$,
\[
X_{\mathrm{T}_{1,2,m+2}}=\frac{A_{1,2,3}A_{2,m+2,m+3}}{A_{1,2,m+3}A_{2,3,m+2}},\quad X_{\mathrm{T}_{2,m+2,m+3}}=\frac{A_{2,m+1,m+2}A_{1,m+2,m+3}}{A_{1,2,m+2}A_{m+1,m+2,m+3}};
\]
\item when $m=2n+1\ge3$ and $2\le i=2k\le m-2$,
\[
X_{\mathrm{T}_{n+1-k,n+2-k,n+3+k}}=\frac{A_{n+1-k,n+2-k,n+3-k}A_{n+2-k,n+3+k,n+4+k}}{A_{n+2-k,n+3-k,n+3+k}A_{n+1-k,n+2-k,n+4+k}}, 
\]
\[
X_{\mathrm{T}_{n+2-k,n+3+k,n+4+k}}=\frac{A_{n+1-k,n+3+k,n+4+k}A_{n+2-k,n+2+k,n+3+k}}{A_{n+1-k,n+2-k,n+3+k}A_{n+2+k,n+3+k,n+4+k}};
\]
\item when $m=2n+1\ge3$ and $3\le i=2k-1\le m-2$,
\[
X_{\mathrm{T}_{n+2-k,n+2+k,n+3+k}}=\frac{A_{n+2-k,n+3+k,n+4+k}A_{n+3-k,n+2+k,n+3+k}}{A_{n+2+k,n+3+k,n+4+k}A_{n+2-k,n+3-k,n+3+k}}, 
\]
\[
X_{\mathrm{T}_{n+2-k,n+3-k,n+3+k}}=\frac{A_{n+1-k,n+3+k,n+4+k}A_{n+2-k,n+2+k,n+3+k}}{A_{n+2+k,n+3+k,n+4+k}A_{n+1-k,n+2-k,n+3+k}};
\]
\end{enumerate}
Therefore, we obtain that for each pair of compatible abelianization trees $\mathrm{T}_1,\mathrm{T}_2$ determined by a pair of adjacent zeros $\{z_i,z_{i+1}\}$, the corresponding homology classes $[X_{\mathrm T_1}], [X_{\mathrm T_2}]$ are linearly independent and belong to the subgroup of $\Gamma(\varphi_m)$ generated by $\gamma_{i,i+1}^{12}, \gamma_{i,i+1}^{23}$. As a result, the set $\{[X_{\mathrm{T}}]:\mathrm{T}\in\mathcal{T}^{\circ}(\varphi_m,\vartheta)\}$ forms a basis in $\Gamma(\varphi_m)$.
Without loss of generality, assume that $\mathcal W(\varphi_m,\vartheta)$ is generic. Then, it is straightforward to derive that after crossing one additional BPS-ful locus associated with $\{z_i,z_{i+1}\}$ which determines compatible abelianization trees $\mathrm{T}_1,\mathrm{T}_2$, there will be an induced isomorphism on $\Gamma(\varphi_m)$, given by the Gauss-Manin connection
\begin{equation}
\begin{split}
\mathrm{GM}: \Gamma(\varphi_m)&\to\Gamma(\varphi_m)\\
[X_{\mathrm{T}}]&\mapsto
\begin{cases}
[X_{\mathrm{T}_1}]+[X_{\mathrm{T}_2}],\quad \mathrm{T}=\mathrm{T}_1;\\
-[X_{\mathrm{T}_2}],\quad \mathrm{T}=\mathrm{T}_2;\\
[X_{\mathrm{T}}],\quad \text{else}.\\
\end{cases}
\end{split}
\end{equation}
Here, we assume that $\mathrm{T}_2$ is the mutated vertex corresponding to this crossing. 
In fact, based on this observation, we can conclude that $\{[X_{\mathrm{T}}]: \mathrm{T}\in\mathcal{T}^{\circ}(\varphi_m,\vartheta)\}$ generates a basis in $\Gamma(\varphi_m)$ when zeros of $\varphi_m(z)$ are almost on a line in $\mathbb C$, since by definition, $\mathcal W(\varphi_m,\vartheta)$ can be obtained from another WKB spectral network $\mathcal W(\widetilde{\varphi},\vartheta_1)$, where zeros of $\vartheta_1$ are on the line, by crossing certain BPS-ful locus associated with a subset of adjacent zero pairs in $\mathrm{Zero}(\widetilde{\varphi})$.
\end{proof}

Based on Proposition \ref{generator}, we can now define the spectral coordinate with respect to any class $\gamma=\sum k_{\mathrm{T}}[X_{\mathrm T}]\in\Gamma(\varphi_m)$ from $\mathcal W(\varphi_m,\vartheta)$ as 
\begin{equation}
\label{spectral coordinate}
X_{\gamma}:=\prod_{\mathrm{T}\in\mathcal{T}^{\circ}(\varphi_m,\vartheta)}X_{\mathrm{T}}^{k_{\mathrm T}},
\end{equation}
which satisfies that $[X_{\gamma}]=\gamma$.

\begin{example}[Gr(3,5)]
\label{jumping example}
We continue with the Example \ref{eg:Gr(3,5)}. The associated spectral coordinates are defined as follows:
\[
X_{\mathrm{T}_{1,2,4}}=X_{\gamma^{21}}=\frac{A_{2,4,5}A_{1,2,3}}{A_{2,3,4}A_{1,2,5}},
\]
\[
X_{\mathrm{T}_{2,4,5}}=X_{\gamma^{13}}=\frac{A_{2,3,4}A_{1,4,5}}{A_{3,4,5}A_{1,2,4}}.
\]
The classes $[X_{\mathrm{T}_{1,2,4}}]=\gamma^{21}$ and $[X_{\mathrm{T}_{2,4,5}}]=\gamma^{13}$ have been illustrated in Figures \ref{class X1}, \ref{class X2}, respectively. If $\mathcal W(\varphi_2,\vartheta)$ crosses a BPS-ful locus by varying $\vartheta$ counterclockwise, we can obtain another pair of compatible abelianization trees
\[
\mathcal{T}^{\circ}(\varphi_2,\vartheta)=\{\mathrm{T}_{1,2,4},\mathrm{T}_{1,3,4}\},
\]
with corresponding spectral coordinates
\[
X_{\mathrm{T}_{1,2,4}}=X_{\gamma^{21}+\gamma^{13}}=\frac{A_{1,2,3}A_{1,4,5}}{A_{1,2,5}A_{1,3,4}},
\]
\[
X_{\mathrm{T}_{1,3,4}}=X_{-\gamma^{13}}=\frac{A_{1,2,4}A_{3,4,5}}{A_{1,4,5}A_{2,3,4}}.
\]
\begin{figure}
\centering
\includegraphics[width=0.3\textwidth]{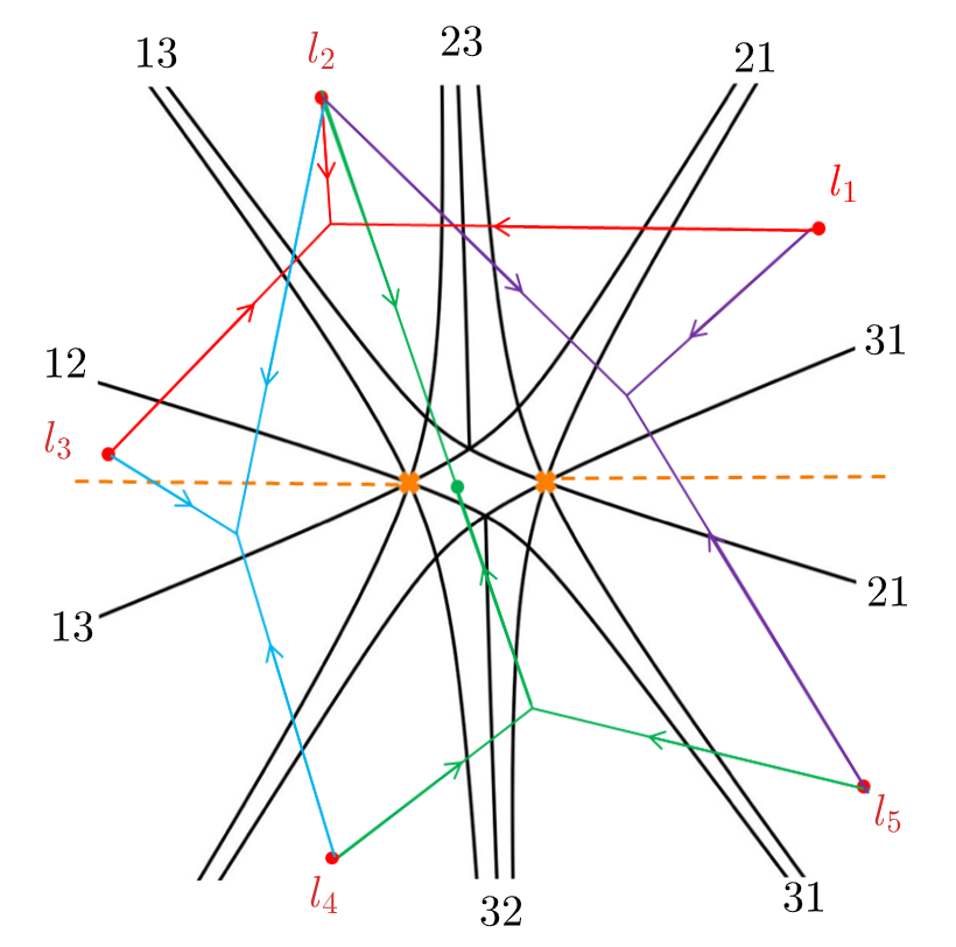}
\caption{The homology class $[X_{\mathrm{T}_{1,2,4}}]$.}
\label{class X1}
\end{figure}
\begin{figure}
\centering
\includegraphics[width=0.3\textwidth]{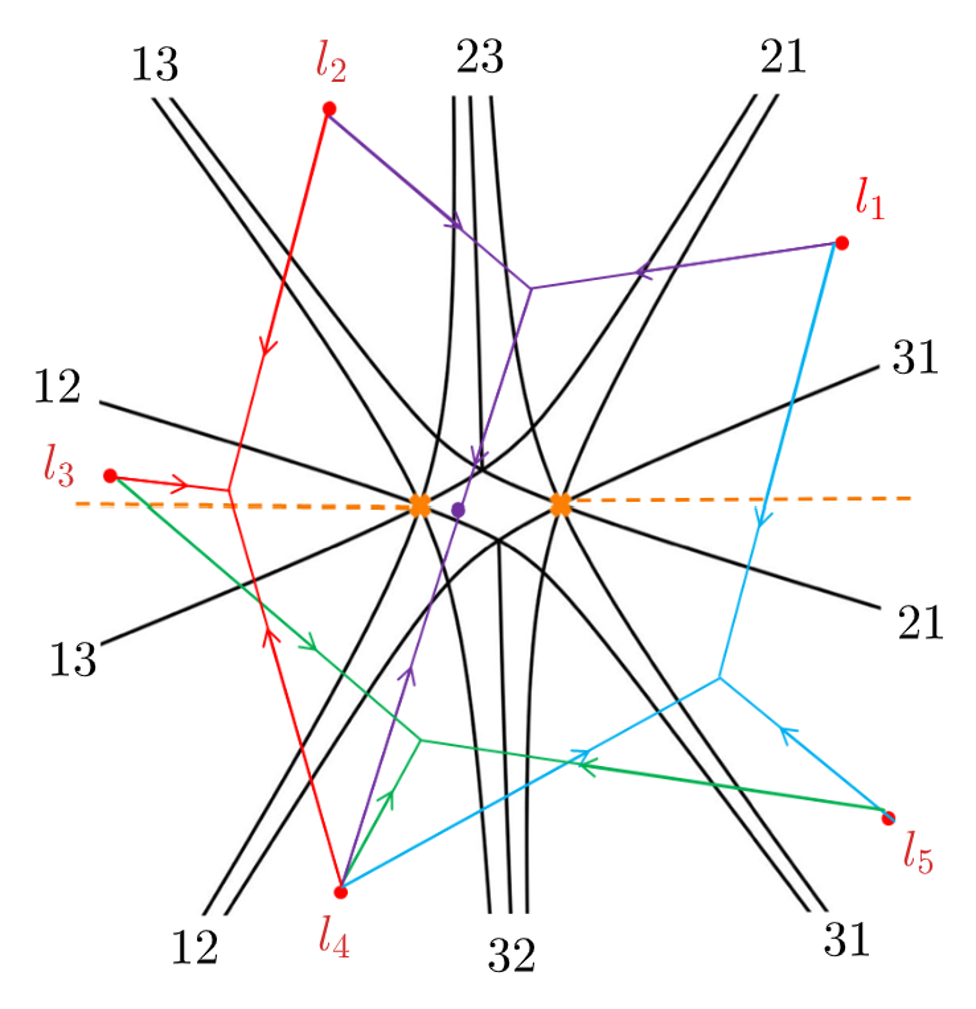}
\caption{The homology class $[X_{\mathrm{T}_{2,4,5}}]$.}
\label{class X2}
\end{figure}
\end{example}

\section{Riemann-Hilbert problems from WKB spectral networks}
\label{Riemann-Hilbert pro}
In this section, the Riemann-Hilbert problems from WKB spectral networks will be the foremost concern. We will review the definition of Riemann-Hilbert problems as outlined in \cite{B1,B2}, and introduce BPS structures associated to WKB spectral networks.

\subsection{BPS structures}
This concept axiomates the Donaldson-Thomas theory when it is applied to a $3$-Calabi-Yau triangulated category equipped with a stability condition.
\begin{definition}\cite[Definition 2.1]{B2}
A \emph{BPS structure} consists of 
\begin{enumerate}
\item a lattice $\Gamma$ known as the \emph{charge lattice} with a skew-symmetric bilinear form
\[
\langle-,-\rangle:\Gamma\times\Gamma\to\mathbb Z
\]
called the \emph{intersection form}.
\item a group homomorphism $Z:\Gamma\to\mathbb C$ referred to as the \emph{central charge}.
\item a collection of rational numbers $\Omega(\gamma) (\gamma\in\Gamma)$ called BPS invariants.
\end{enumerate}
These data must fulfill the following properties:
\begin{enumerate}
\item \emph{Symmetry property}: $\Omega(\gamma)=\Omega(-\gamma)$ for all $\gamma\in\Gamma$.
\item \emph{Support property}: Fix a norm $||\cdot||$ on the finite-dimensional vector space $\Gamma\otimes_{\mathbb Z}\mathbb R$. Then, a constant $C>0$ exists such that $|Z(\gamma)|>C\cdot||\gamma||$ if $\Omega(\gamma)\ne0$.
\end{enumerate}
\end{definition}
If the second requirement can be substituted by the following stronger condition, it is said that a BPS structure is finite:
\begin{enumerate}
\item there are only finitely many class $\gamma\in\Gamma$ such that $\Omega(\gamma)\ne0$.
\end{enumerate}

A finite BPS strcture $(\Gamma,Z,\Omega)$ is said to be non-degenerate if the form $\langle-,-\rangle$ is non-degenerate, and integral if $\Omega(\gamma)\in\mathbb Z\subset\mathbb Q$ for all $\gamma\in\Gamma$.

%------------------------

\subsection{The twisted torus}
Assume that we have a lattice $\Gamma\cong\mathbb Z^n$ with a skew form $\langle-,-\rangle$ as in the definition of a BPS structure. Then there is a associated algebraic torus
\[
\mathbb T_{+}=\mathrm{Hom}_{\mathbb Z}(\Gamma,\mathbb C^{\ast})\cong(\mathbb C^{\ast})^n,
\]
whose co-ordinate ring is denoted by 
\[
\mathbb C[\mathbb T_+]=\mathbb C[\Gamma]\cong\mathbb C[y_1^{\pm1},\dots,y_n^{\pm n}],
\]
where $y_{\gamma}\in\mathbb[\mathbb T_+]$ is the character of $\mathbb T_+$ that correponds to an element $\gamma\in\Gamma$. 

We introduce the concept of a related torsor
\[
\mathbb T_{-}=\{g:\Gamma\to\mathbb C^{\ast}:g(\gamma_1+\gamma_2)=(-1)^{\langle\gamma_1,\gamma_2\rangle}g(\gamma_1)g(\gamma_2)\}
\]
called the \emph{twisted torus}. 

The co-ordinate ring of $\mathbb T_-$ is spanned as a vector space by the functions 
\[
x_{\gamma}:\mathbb T_-\to\mathbb C^{\ast}, \quad x_{\gamma}(g)=g(\gamma)\in\mathbb C^{\ast},
\]
which are referred to as twisted characters. Thus,
\[
\mathbb C[\mathbb T_-]=\bigoplus_{\gamma\in\Gamma}\mathbb C\cdot x_{\gamma},\quad x_{\gamma_1}\cdot x_{\gamma_2}=(-1)^{\langle\gamma_1,\gamma_2\rangle}\cdot x_{\gamma_1+\gamma_2}.
\]

The torus $\mathbb T_+$ acts on the twisted torus $\mathbb T_{-}$ in a natural way, which is defined as
\[
(f\cdot g)(\gamma)=f(\gamma)g(\gamma)\in\mathbb C^{\ast}
\]
for $f\in\mathbb T_+$ and $g\in\mathbb T_-$, and this action is both free and transitive. Therefore, by selecting a base-point $g_0\in\mathbb T_-$, we obtain a bijection
\[
\vartheta_{g_0}:\mathbb T_+\to\mathbb T_-, \quad f\mapsto f\cdot g_0.
\]
It is often useful to choose a base-point in the finite subset 
\[
\{g:\Gamma\to\{\pm1\}:g(\gamma_1+\gamma_2)=(-1)^{\langle\gamma_1,\gamma_2\rangle}g(\gamma_1)\cdot g(\gamma_2)\}\subset\mathbb T_{-},
\]
and the points in this subset are known as quadratic refinements of the form $\langle-,-\rangle$.

In the context of a BPS structure, a class $\gamma\in\Gamma$ is conisdered active if its corresponding BPS invariant $\Omega(\gamma)$ is nonzero. A ray $\mathbb R_{>0}\cdot z\subset\mathbb C^{\ast}$ is called active if it contains a point of the form $Z(\gamma)$ with $\gamma\in\Gamma$ being an active class. For each ray $l=\mathbb R_{>0}\cdot z\subset\mathbb C^{\ast}$ in a finite and integral BPS structure, we define a birational automorphism of the twisted torus $\mathbb T_-$ using the formula
\begin{equation}
\label{birational auto}
\mathbb S(l)^{\ast}(x_{\beta})=x_{\beta}\cdot\prod_{Z(\gamma)\in l}(1-x_{\gamma})^{\Omega(\gamma)\langle\gamma,\beta\rangle},
\end{equation}
where the product is taken over all active classes $\gamma\in\Gamma$ such that $Z(\gamma)\in l$. The finiteness of the BPS structure ensures that this is a finite set.

\subsection{BPS structures associated to WKB spectral networks}
We will introduce the BPS structures associated with a family of WKB spectral networks, as described in \cite{N1}. For each point $a=(a_1,\dots,a_m)\in\mathbb C^m$, there is a corresponding family of WKB spectral networks $\{\mathcal W(\varphi_a=z^m+a_{1}z^{m-1}+\dots+a_{m-1}z+a_m,\vartheta);\vartheta\in\mathbb R\}$. We call a point $a\in\mathbb C^m$ generic if the corresponding family of WKB spectral networks are generic, as described in \Cref{generic}.
\begin{definition}
\label{BPS}
The BPS structure $(\Gamma_a,Z_a,\Omega_a)$ associated to a generic point $a\in\mathbb C^m$ is defined as follows:
\begin{enumerate}
\item the charge lattice is $\Gamma_a=H_1(\Sigma(\varphi_a),\mathbb Z)$ with its intersection form $\langle-,-\rangle$;
\item the central charge $Z_a:\Gamma_a\to\mathbb C$ is the map \eqref{central charge};
\item the BPS invariants $\Omega_a(\gamma):=\Omega(\varphi_a,\gamma)$ are either $0$ or $1$, as discussed in \Cref{BPS_counts}.
\end{enumerate}
\end{definition}

\begin{remark}
\begin{enumerate}
\item The BPS invariants have the symmetry property
\[
\Omega_a(\gamma)=\Omega_a(-\gamma)
\]
This is demonstrated by the fact that any web in $\mathcal W(\varphi_a,\vartheta)$ with charge $\gamma$ has a partner in $\mathcal W(\varphi_a,\vartheta+\pi)$ with charge $-\gamma$, since that $\mathcal W(\varphi_a,\vartheta)$ and $\mathcal W(\varphi_a,\vartheta+\pi)$ differ only by relabeling all trajectories $ij\to ji$.
\item Under the $\mathbb Z/3\mathbb Z$ action on $\Gamma_a$ induced by the cyclic permutation of sheets in $\Sigma(\varphi_a)$, the BPS invariants $\Omega_a(\gamma)$ are invariant. This $\mathbb Z/3\mathbb Z$ symmetry of the BPS invariants is reflected on $\mathcal W(\varphi_a,\vartheta)$ and $\mathcal W(\varphi_a,\vartheta+\frac{2\pi}{3})$, which differ only by cyclic permutation of the sheet labels.
\end{enumerate}
\end{remark}

\subsection{Defining the Riemann-Hilbert problem}
Consider a finite BPS structure $(\Gamma,Z,\Omega)$ and let $r$ be a ray in $\mathbb C^{\ast}$. We define the corresponding half-plane as
\begin{equation}
\label{half plane}
\mathbb H_r=\{\hbar\in\mathbb C^{\ast}:\hbar=z\cdot v\ \text{with}\ z\in r,\,\mathrm{Re}(v)>0\}\subset\mathbb C^{\ast}.
\end{equation}
In our analysis, we will work with meromorphic functions
\[
Y_r:\mathbb H_r\to\mathbb T_+
\]
which can be equivalently expressed as following functions by composing with the twisted characters of $\mathbb T_+$
\[
Y_{r,\gamma}:\mathbb H_r\to\mathbb C^{\ast},\quad Y_{r,\gamma}(t)=y_{\gamma}(Y_r(t)).
\]
The Riemann-Hilbert problem involves an additional choice of the constant term $\xi\in\mathbb T_{-}$.

\begin{problem}\cite[Problem 4.3]{B2}
\label{main problem}
We aim to find a meromorphic funtion $Y_r:\mathbb H_r\to \mathbb T_+$ for each non-active ray $r\in\mathbb C^{\ast}$, subject to the following conditions:
\begin{enumerate}
\item[(RH1)] given two non-active rays $r_{-},r_{+}$ forming the boundary rays of a convex sector $\triangle\subset\mathbb C^{\ast}$, taken in clockwise order, we require that
\[
Y_{r_2}(\hbar)=\mathbb S(\triangle)(Y_{r_1}(\hbar)),
\]
as meromorphic functions of $\hbar\in\mathbb H_{r_-}\cap\mathbb H_{r_+}$, where 
\begin{equation}
\label{ts}
\mathbb S(\triangle):=\prod_{l\in\triangle}\mathbb S(l).
\end{equation}
\item[(RH2)] for each non-active ray $r\subset\mathbb C^{\ast}$ and each class $\gamma\in\Gamma$, we have 
\[
\mathrm{exp}(Z(\gamma)/\hbar)\cdot Y_{r,\gamma}(\hbar)\to\xi(\gamma),
\]
as $\hbar\to0$ in the half-plane $\mathbb H_r$.
\item[(RH3)] for each class $\gamma$ and any non-active ray $r\subset\mathbb C^{\ast}$, there exists $k>0$ such that 
\[
|\hbar|^{-k}<|Y_{r,y}(\hbar)|<|\hbar|^k,
\]
for all $\hbar\in\mathbb H_r$ with $|\hbar|\gg 0$.
\end{enumerate}
\end{problem}

%------------------------------------------------------------------------------------------------------------------------------
%------------------------------------------------------------------------------------------------------------------------------

\subsection{Solving the Riemann-Hilbert problems}
\label{solution}
Solutions to the Riemann-Hilbert problems arising from quadratic differentials have been well studied in various references, including \cite{ A1, A2, A3, A4, BM}. Building on their ideas, we will construct solutions to the Riemann-Hilbert problems from WKB spectral networks $\mathcal W(\varphi_a,\vartheta)$ when zeros of $\varphi_a(z)$ are almost on a line. Our approach primarily relies on the spectral coordinates defined earlier and the asymptotic properties of solutions to certain differential equations discussed in \Cref{App:differential equation}.

\subsubsection{The third order ODE}
\label{sec4.1}
Let $\varphi_a(z)=z^m+a_1z^{m-1}+\cdots+a_{m-1}z+a_m$ be a polynomial in $\mathbb C$. Consider the following differential equation
\begin{equation}
\label{third order ODE}
\hbar^3y'''(z)+\varphi_a(z)y(z)=0,
\end{equation}
which is defined in a neighborhood of infinity, where $\hbar$ is a complex parameter. Let $\hbar$ be a positive real number. We define the Stokes sector $\mathcal S_k (k\in\mathbb Z)$ as 
\[
\mathcal S_k:=\left\{z\in\mathbb C:\,\left|\mathrm{arg}\,z-\frac{2\pi (k-1)}{m+3}\right|<\frac{\pi}{m+3} \right\}.
\]

Based on \Cref{mainthm4}, we can derive the following result.
\begin{proposition}
The differential equation \eqref{third order ODE} admits a subdominant solution $y_1$ in $\mathcal S_1$ whose asymptotic property is given by
\[
y_1(z,a; \hbar)\sim \frac{\hbar}{i\sqrt{3}}z^{r_m}\mathrm{exp}(-\frac{1}{\hbar}\frac{3}{m+3}z^{\frac{m+3}{3}}),\quad |z|\to\infty,
\]
where $r_m$ is a polynomial in $a=(a_1,\dots,a_m)$.
\end{proposition}

Note that \eqref{third order ODE} is invariant under the Symanzik-Sibuya rotation:
\[
z\to\omega_1^{-1}z,\quad a_j\to\omega_1^{-j}a_j,\, j=1,\dots,m,
\]
where $\omega_1:=e^{2\pi i/(m+3)}$. Then the rotated subdominant solutions $y_k$ in $\mathcal S_k$ can then be expressed as
\[
y_k(z,a;\hbar):=y_1(\omega_1^{-k+1}z,\{\omega_1^{(-k+1)j}a_j\};\hbar).
\]

The Wronskian of solutions $y_i,y_j,y_k$ is defined as 
\[
W[y_i,y_j,y_k]:=\mathrm{det}
\begin{bmatrix}
y_{i}&y_{j}&y_{k}\\
y'_{i}&y'_{j}&y'_{k}\\
y''_{i}&y''_{j}&y''_{k}\\
\end{bmatrix},
\]
which is independent of $z$ since 
\[
\frac{dW[y_i,y_j,y_k]}{dz}=0.
\]
The set of solutions $\{y_k,y_{k+1},y_{k+2}\}$ forms a basis for the solutions of \eqref{third order ODE} since the Wronskian of these solutions is nonzero.

\subsubsection{Solutions to the Riemann-Hilbert problem}
Let $S^m\subset\mathbb C^m$ be an open submanifold that contains points $a\in\mathbb C^m$ satisfying
\begin{enumerate}
\item zeros of $\varphi_a(z)$ are almost on a line in $\mathbb C$;
\item $\mathcal W(\varphi_a,\vartheta)$ is generic, for $\vartheta\in\mathbb R$.
\end{enumerate}
It is worth noting that $S^2=\{(a,b)\in\mathbb C^2,a\ne b\}$. For any point $a\in S^m$, there exists a quadratic refinement $g$ belonging to the twisted torus $\mathbb T_{a,-}$, defined by setting 
\begin{equation}
\label{quad refinement}
g(\gamma^{12}_{i,i+1})=g(\gamma^{23}_{i,i+1})=-1,\quad i=1,2,\dots,m+2,
\end{equation}
which is unique with the property that $g(\gamma)=-1$ for every active class $\gamma\in\Gamma_a$. Therefore, for any $\gamma=\sum_{i=1}^{m-1}(k_i\gamma^{12}_{i,i+1}+h_i\gamma^{23}_{i,i+1})\in\Gamma_a$, the associated quadratic refinement is given by 
\begin{equation}
g(\gamma)=g\left(\sum_{i=1}^{m-1}(k_i\gamma^{12}_{i,i+1}+h_i\gamma^{23}_{i,i+1})\right)=\prod_{i}(-1)^{(k_ih_i+k_i+h_i)}.
\end{equation}
Note that the well-definedness of the quadratic refinement $g$ can be checked by some simple calculations. We can then use the quadratic refinement $g$ to identify the twisted torus $\mathbb T_{a,-}$ with the standard torus $\mathbb T_{a}=\mathbb T_{a,+}$. Under this identification, we can view the birational automorphism \eqref{birational auto} as a birational automorphism of $\mathbb T_{a}$, which is defined by 
\begin{equation}
\label{refinement}
\mathbb S(l)^{\ast}(y_{\beta})=y_{\beta}\cdot\prod_{Z(\gamma)\in l}(1+y_{\gamma})^{\Omega(\gamma)\langle\gamma,\beta\rangle}.
\end{equation}
Then, to solve the Riemann-Hilbert problem \ref{main problem}, we need to choose a constant term $\xi\in\mathbb T_a$ and construct meromorphic maps
\[
Y_r:\mathbb H_r\to\mathbb T_a
\]
for all non-active rays $r=e^{i\vartheta_0}\cdot\mathbb R_{>0}\subset\mathbb C^{\ast}$, where $\mathbb H_r$ is the half-plane defined in \eqref{half plane}.

We now consider a generic WKB spectral network $\mathcal W(\varphi_a,\vartheta)$ for some $a\in S^m$ and let $\{l_1,l_2,\dots,l_{m+3}\}$ be the set of chosen asymptotic directions. Consider \eqref{third order ODE} with $\hbar=1$. Based on the discussion in \Cref{sec4.1}, we are able to obtain $m+3$ distinct solutions $y_1,\dots,y_{m+3}$ with the property that for each $k\in\{1,\dots,m+3\}$, $y_k$ is subdominant along the asymptotic direction $l_{k}$. Let $\mathcal V(3,n)$ be the subspace of the Grassmannian manifold $Gr(3,n)$ defined as:
\begin{equation}
\label{Vspace}
\mathcal V(3,n):=\{(v_1,v_2,\dots,v_{n})\in Gr(3,n)\ |\
\mathrm{det}(v_i,v_{i+1},v_{i+2})\ne0, i\in\mathbb Z/n\mathbb Z\}.
\end{equation}
Note that the space $\mathcal V(3,n)$ has a natural action of the group $\mathbb Z/n\mathbb Z$, given by
\[
i\cdot (v_1,\dots,v_{n}):=(v_{1+i},\dots,v_{n+i}),
\]
where $i\in\mathbb Z/n\mathbb Z$. 

As a result, we can asscociate an element 
\begin{equation}
\label{M}
M(\varphi_a,\vartheta):=
\begin{bmatrix}
y_1&y_2&\cdots&y_{m+3}\\
y_1'&y'_2&\cdots&y_{m+3}'\\
y_1''&y''_2&\cdots&y''_{m+3}\\
\end{bmatrix}
\in\mathcal V(3,m+3).
\end{equation}
with the spectral network $\mathcal W(\varphi_a,\vartheta)$. 

By attaching each column vector $v_{i}$ from the point $v\in\mathcal V(3,m+3)$ to the asymptotic direction $l_{i}$, we obtain a meromorphic function
\begin{equation}
\begin{split}
\mathcal X_{a,\vartheta}:\mathcal V(3,m+3)&\to\mathbb T_a\\
v&\mapsto (X_{a,\vartheta}(v):\gamma\mapsto X_{a,\vartheta,\gamma}(v)),
\end{split}
\end{equation}
where $X_{a,\vartheta,\gamma}(v)$ is the spectral coordinate with respect to $\gamma\in\Gamma_a$ for the spectral network $\mathcal W(\varphi_a,\vartheta)$ as defined in \eqref{spectral coordinate}. Note that using the Gauss-Manin connection, the map $X_{a,\vartheta,\gamma}$ makes senses for $\gamma\in\Gamma_a$. Due to the invariance property of the spectral coordinates, we can descend $\mathcal X_{a,\vartheta}$ to the following quotient space
\[
V(3,m+3):=\mathcal V(3,m+3)/\mathrm{PGL(3,\mathbb C)}.
\]

Let $r=e^{i\vartheta_0}\cdot\mathbb R_{>0}\subset\mathbb C$ be a non-active ray corresponding to the BPS structure $(\Gamma_a,Z_a,\Omega_a)$ for the WKB spectral network $\mathcal W(\varphi_a,\vartheta)$. Note that the ray $r$ is assumed to be non-active is equivalent to the claim that $\mathcal W(\lambda^{-3}\varphi_a,0)$ is BPS-free for any $\lambda\in r$.

Based on the previous discussion, we can derive a holomorphic map
\begin{equation}
\label{F}
F_r: \mathbb H_r\to V(3,m+3), \quad \hbar\mapsto \overline{M(\hbar^{-3}\varphi_a,0)}.
\end{equation}

Define 
\begin{equation}
\label{X}
\mathcal X_r:=\mathcal X_{a,\vartheta_0}: V(3,m+3)\to\mathbb T_a.
\end{equation}
Then we compose $F_r$ with $\mathcal X_r$ to obtain the required meromorphic function
\begin{equation}
\label{Y}
Y_r:=\mathcal X_r\circ F_r:\mathbb H_r\to\mathbb T_a.
\end{equation}

We now proceed to investigate the conditions (RH1)-(RH3) of Problem \ref{main problem} on $S^m$.

\subsubsection{Jumping}
\label{sec.4.3}
We begin by considering the jumping condition (RH1). Take a point $a\in S^m$ and let $l=e^{i\vartheta}\cdot\mathbb R\subset\mathbb C^{\ast}$ be an active ray for the corresponding BPS structure $(\Gamma_{a},Z_{a},\Omega_{a})$. Consider non-active rays $r_{-}=e^{i\vartheta_{-}}\cdot\mathbb R_{>0}$ and $r_{+}=e^{i\vartheta_{+}}\cdot\mathbb R_{>0}$ which are small anti-clockwise and clockwise deformations of the ray $l$. 

\begin{proposition}
\label{transformation}
The two systems of meromorphic functions are related by 
\[
Y_{r_+}=\mathbb S(l)\circ Y_{r_-}.
\]
\end{proposition}
\begin{proof}
We begin by proving that two systems of meromorphic functions $\mathcal X_{r_{\pm}}$ are related by 
\[
\mathcal X_{r_+}=\mathbb S(l)\circ \mathcal X_{r_-}.
\]
We consider the case where zeros of $\varphi_a(z)$ are almost on the real line and the labels of marked points are given as $31,21,23,13, \dots$ in the range of $(0,2\pi)$, with the asymptotic direction $l_1$ lying between the first $31$ and $21$. The other cases can be checked in a similar manner. Note that the intersection form on $\Gamma_a$ can be determined by $\langle\gamma_{i,i+1}^{13},\gamma^{21}_{i,i+1}\rangle=1$, $i=1,\dots,m-1$. There are eight cases that we need to consider, determined by a pair of adjacent zeros $\{z_i,z_{i+1}\}$ where the crossing from $\mathcal W(\varphi_a,\vartheta_+)$ to $\mathcal W(\varphi_a,\vartheta_-)$ happens. Base on the proof process of \Cref{generator}, we can have the following calculations:
\begin{enumerate}
\item when $m=2n\ge2$ and $i=1$,
\begin{equation*}
\begin{split}
\mathcal X^{\ast}_{r_-}(y_{\gamma^{21}_{1,2}}(1+y_{\gamma^{13}_{1,2}}))=&\frac{A_{n,n+3,n+4}A_{n,n+1,n+2}}{A_{n,n+2,n+3}A_{n,n+1,n+4}}\left(1+\frac{A_{n,n+1,n+3}A_{n+2,n+3,n+4}}{A_{n,n+3,n+4}A_{n+1,n+2,n+3}}\right)\\
=&\frac{A_{n,n+1,n+2}A_{n+1,n+3,n+4}}{A_{n,n+1,n+4}A_{n+1,n+2,n+3}]}\\
=&\mathcal X^{\ast}_{r_+}(y_{\gamma^{23}_{1,2}})
\end{split}
\end{equation*}
\begin{equation*}
\begin{split}
\mathcal X^{\ast}_{r_-}(y_{\gamma^{13}_{1,2}})&=\frac{A_{n,n+3,n+4}A_{n+1,n+2,n+3}}{A_{n,n+1,n+3}A_{n+2,n+3,n+4}}=\mathcal X^{\ast}_{r_+}(y_{\gamma^{13}_{1,2}})
\end{split}
\end{equation*}
\item when $m=2n\ge2$ and $i=m-1$,
\begin{equation*}
\begin{split}
\mathcal X^{\ast}_{r_-}(y_{\gamma^{21}_{m-1,m}}(1+y_{\gamma^{13}_{m-1,m}}))=&\frac{A_{1,2,3}A_{2,m+1,m+2}A_{1,m+2,m+3}}{A_{1,m+1,m+2}A_{2,3,m+2}A_{1,2,m+3}}\cdot\left(1+\frac{A_{1,2,m+2}A_{m+1,m+2,m+3}}{A_{2,m+1,m+2}A_{1,m+2,m+3}}\right)\\
=&\frac{A_{1,2,3}A_{2,m+2,m+3}}{A_{2,3,m+2}A_{1,2,m+3}}\\
=&\mathcal X^{\ast}_{r_+}(y_{\gamma^{23}_{m-1,m}})
\end{split}
\end{equation*}
\begin{equation*}
\begin{split}
\mathcal X^{\ast}_{r_-}(y_{\gamma^{13}_{m-1,m}})&=\frac{A_{2,m+1,m+2}A_{1,m+2,m+3}}{A_{1,2,m+2}A_{m+1,m+2,m+3}}=\mathcal X^{\ast}_{r_+}(y_{\gamma^{13}_{m-1,m}})
\end{split}
\end{equation*}
\item when $m=2n\ge2$ and $2\le i=2k\le m-2$,
\begin{equation*}
\begin{split}
\mathcal X^{\ast}_{r_-}(y_{\gamma^{21}_{i,i+1}}(1+y_{\gamma^{13}_{i,i+1}}))=&\frac{A_{n+1-k,n+2-k,n+2+k}A_{n-k,n+1-k,n+3+k}A_{n+2+k,n+3+k,n+4+k}}{A_{n-k,n+1-k,n+2+k}A_{n+1-k,n+3+k,n+4+k}A_{n+2-k,2,n+2+k,n+3+k}}\\
&\cdot\left(1+\frac{A_{n+1-k,n+2+k,n+3+k}A_{n-k,n+1-k,n+2-k}}{A_{n+1-k,n+2-k,n+2+k}A_{n-k,n+1-k,n+3+k}}\right)\\
=&\frac{A_{n+1-k,n+2-k,n+3+k}A_{n+2+k,n+3+k,n+4+k}}{A_{n+1-k,n+3+k,n+4+k}A_{n+2-k,n+2+k,n+3+k}}\\
=&\mathcal X^{\ast}_{r_+}(y_{\gamma^{23}_{i,i+1}})
\end{split}
\end{equation*}
\begin{equation*}
\begin{split}
\mathcal X^{\ast}_{r_-}(y_{\gamma^{13}_{i,i+1}})&=\frac{A_{n+1-k,n+2-k,n+2+k}A_{n-k,n+1-k,n+3+k}}{A_{n+1-k,n+2+k,n+3+k}A_{n-k,n+1-k,n+2-k}}=\mathcal X^{\ast}_{r_+}(y_{\gamma^{13}_{i,i+1}})
\end{split}
\end{equation*}
\item when $m=2n\ge2$ and $3\le i=2k-1\le m-2$,
\begin{equation*}
\begin{split}
\mathcal X^{\ast}_{r_-}(y_{\gamma^{21}_{i,i+1}}(1+y_{\gamma^{13}_{i,i+1}}))&=\frac{A_{n+1+k,n+2+k,n+3+k}A_{n+1-k,n+2-k,n+2+k}A_{n+1-k,n+2-k,n+3-k}}{A_{n+1-k,n+1+k,n+2+k}A_{n+2-k,n+3-k,n+2+k}A_{n+1-k,2,n+2-k,n+3+k}}\\
&\cdot\left(1+\frac{A_{n+2-k,n+1+k,n+2+k}A_{n+1-k,n+2+k,n+3+k}}{A_{n+1+k,n+2+k,n+3+k}A_{n+1-k,n+2-k,n+2+k}}\right)\\
&=\frac{A_{n+1-k,n+2-k,n+3-k}A_{n+2-k,n+2+k,n+3+k}}{A_{n+2-k,n+3-k,n+2+k}A_{n+1-k,n+2-k,n+3+k}}\\
&=\mathcal X^{\ast}_{r_+}(y_{\gamma^{23}_{i,i+1}})
\end{split}
\end{equation*}
\begin{equation*}
\begin{split}
\mathcal X^{\ast}_{r_-}(y_{\gamma^{13}_{i,i+1}})&=\frac{A_{n+1+k,n+2+k,n+3+k}A_{n+1-k,n+2-k,n+2+k}}{A_{n+2-k,n+1+k,n+2+k}A_{n+1-k,n+2+k,n+3+k}}=\mathcal X^{\ast}_{r_+}(y_{\gamma^{13}_{i,i+1}})
\end{split}
\end{equation*}
\item when $m=2n+1\ge3$ and $i=1$,
\begin{equation*}
\begin{split}
\mathcal X^{\ast}_{r_-}(y_{\gamma^{21}_{1,2}}(1+y_{\gamma^{13}_{1,2}}))=&\frac{A_{n,n+1,n+4}A_{n+1,n+2,n+3}A_{n+3,n+4,n+5}}{A_{n,n+1,n+3}A_{n+1,n+4,n+5}A_{n+2,n+3,n+4}}\left(1+\frac{A_{n,n+1,n+2}A_{n+1,n+3,n+4}}{A_{n,n+1,n+4}A_{n+1,n+2,n+3}}\right)\\
=&\frac{A_{n+1,n+2,n+4}A_{n+3,n+4,n+5}}{A_{n+1,n+4,n+5}A_{n+2,n+3,n+4}}\\
=&\mathcal X^{\ast}_{r_+}(y_{\gamma^{23}_{1,2}})
\end{split}
\end{equation*}
\begin{equation*}
\begin{split}
\mathcal X^{\ast}_{r_-}(y_{\gamma^{13}_{1,2}})&=\frac{A_{n,n+1,n+4}A_{n+1,n+2,n+3}}{A_{n,n+1,n+2}A_{n+1,n+3,n+4}}=\mathcal X^{\ast}_{r_+}(y_{\gamma^{13}_{1,2}})
\end{split}
\end{equation*}
\item when $m=2n+1\ge3$ and $i=m-1$,
\begin{equation*}
\begin{split}
\mathcal X^{\ast}_{r_-}(y_{\gamma^{21}_{m-1,m}}(1+y_{\gamma^{13}_{m-1,m}}))=&\frac{A_{1,2,3}A_{2,m+1,m+2}A_{1,m+2,m+3}}{A_{1,2,m+3}A_{2,3,m+2}A_{1,m+1,m+2}}\left(1+\frac{A_{1,2,m+2}A_{m+1,m+2,m+3}}{A_{2,m+1,m+2}A_{1,m+2,m+3}}\right)\\
=&\frac{A_{1,2,3}A_{2,m+2,m+3}}{A_{1,2,m+3}A_{2,3,m+2}}\\
=&\mathcal X^{\ast}_{r_+}(y_{\gamma^{23}_{m-1,m}})
\end{split}
\end{equation*}
\begin{equation*}
\begin{split}
\mathcal X^{\ast}_{r_-}(y_{\gamma^{13}_{m-1,m}})&=\frac{A_{2,m+1,m+2}A_{1,m+2,m+3}}{A_{1,2,m+2}A_{m+1,m+2,m+3}}=\mathcal X^{\ast}_{r_+}(y_{\gamma^{13}_{m-1,m}})
\end{split}
\end{equation*}
\item when $m=2n+1\ge3$ and $2\le i=2k\le m-2$,
\begin{equation*}
\begin{split}
\mathcal X^{\ast}_{r_-}(y_{\gamma^{21}_{i,i+1}}(1+y_{\gamma^{13}_{i,i+1}}))&=\frac{A_{n+1-k,n+3+k,n+4+k}A_{n+2-k,n+2+k,n+3+k}A_{n+1-k,n+2-k,n+3-k}}{A_{n+1-k,n+2+k,n+3+k}A_{n+2-k,n+3-k,n+3+k}A_{n+1-k,2,n+2-k,n+4+k}}\\
&\cdot\left(1+\frac{A_{n+1-k,n+2-k,n+3-k}A_{n+2+k,n+3+k,n+4+k}}{A_{n+1-k,n+3+k,n+4+k}A_{n+2-k,n+2+k,n+3+k}}\right)\\
&=\frac{A_{n+1-k,n+2-k,n+3-k}A_{n+2-k,n+3+k,n+4+k}}{A_{n+2-k,n+3-k,n+3+k}A_{n+1-k,n+2-k,n+4+k}}\\
&=\mathcal X^{\ast}_{r_+}(y_{\gamma^{23}_{i,i+1}})
\end{split}
\end{equation*}
\begin{equation*}
\begin{split}
\mathcal X^{\ast}_{r_-}(y_{\gamma^{13}_{i,i+1}})&=\frac{A_{n+1-k,n+3+k,n+4+k}A_{n+1-k,n+2+k,n+3+k}}{A_{n+1-k,n+2-k,n+3+k}A_{n+2+k,n+3+k,n+4+k}}=\mathcal X^{\ast}_{r_+}(y_{\gamma^{13}_{i,i+1}})
\end{split}
\end{equation*}
\item when $m=2n+1\ge3$ and $3\le i=2k-1\le m-2$,
\begin{equation*}
\begin{split}
\mathcal X^{\ast}_{r_-}(y_{\gamma^{21}_{i,i+1}}(1+y_{\gamma^{13}_{i,i+1}}))&=\frac{A_{n+1-k,n+3+k,n+4+k}A_{n+2-k,n+2+k,n+3+k}A_{n+3-k,n+2+k,n+3+k}}{A_{n+1-k,n+2+k,n+3+k}A_{n+2+k,n+3+k,n+4+k}A_{n+2-k,2,n+3-k,n+3+k}}\\
&\cdot\left(1+\frac{A_{n+2+k,n+3+k,n+4+k}A_{n+1-k,n+2-k,n+3+k}}{A_{n+1-k,n+3+k,n+4+k}A_{n+2-k,n+2+k,n+3+k}}\right)\\
&=\frac{A_{n+2-k,n+3+k,n+4+k}A_{n+3-k,n+2+k,n+3+k}}{A_{n+2+k,n+3+k,n+4+k}A_{n+2-k,n+3-k,n+3+k}}\\
&=\mathcal X^{\ast}_{r_+}(y_{\gamma^{23}_{i,i+1}})
\end{split}
\end{equation*}
\begin{equation*}
\begin{split}
\mathcal X^{\ast}_{r_-}(y_{\gamma^{13}_{i,i+1}})&=\frac{A_{n+1-k,n+3+k,n+4+k}A_{n+2-k,n+2+k,n+3+k}}{A_{n+2+k,n+3+k,n+4+k}A_{n+1-k,n+2-k,n+3+k}}=\mathcal X^{\ast}_{r_+}(y_{\gamma^{13}_{i,i+1}})
\end{split}
\end{equation*}
\end{enumerate}
By definition, the corresponding wall-crossing automorphism $\mathbb S(l)$ satisfies
\[
\mathbb S(l)^{\ast}(y_{\gamma^{13}_{i,i+1}})=y_{\gamma^{13}_{i,i+1}},\quad \mathbb S(l)^{\ast}(y_{\gamma^{21}_{i,i+1}})=y_{\gamma^{21}_{i,i+1}}\cdot(1+y_{\gamma^{13}_{i,i+1}}),
\]
and we can conclude that $X_{r_+}^{\ast}=X_{r_-}^{\ast}\circ\mathbb S(l)^{\ast}$ as required.
Therefore, for any $\hbar\in\mathbb H_{r_-}\cap \mathbb H_{r_+}$, we have
\[
\mathcal X_{r_+}(F_{r_+}(\hbar))=\mathbb S(l)\circ\mathcal X_{r_-}(F_{r_-}(\hbar)),
\]
i.e.,
\[
Y_{r_+}=\mathbb S(l)\circ Y_{r_-}.
\]
\end{proof}

In fact, \Cref{transformation} still holds for non-generic WKB spectral networks when zeros of polynomials are almost on a line in $\mathbb C$, similar to the statement mentioned in \cite[Proposition 7.1]{BM}. Let $\mathcal W(\varphi_a,\vartheta)$ be a non-generic WKB spectral network and $l=e^{i\vartheta}\cdot\mathbb R\subset\mathbb C^{\ast}$ be an active ray. We can deform the point $a\in\mathbb C^{m}$ to a nearby generic point $a'$. Under this deformation, the ray splits into finite rays since there are only finite active generators in this case, but for $a'$ close enough to $a$, these rays $l_i$ will be contained in the sector bounded by the non-active rays $r_{\pm}$. Then, the result for the non-generic point $a$ now follows by applying the same result for the generic point $a'$ to each of the rays $l_i$ and take the composition of the automorphisms $\mathbb S_{a'}(l_i)$.
%------------------------------------------------------------------------------------------------------------------------

\subsubsection{Behavior as $\hbar\to0$}
\label{sec.4.4}
To verify condition (RH2) of Problem \ref{main problem}, we need to show that the map 
\[
Y_r:\mathbb H_r\to \mathbb T_{a}\quad \hbar\mapsto\mathcal X_r(F_r(\hbar))
\]
has the correct asymptotics as $\hbar\to0$. Specifically, we want to show that
\[
Y_{r,\gamma}(\hbar)\sim\mathrm{exp}(-Z_a(\gamma)/\hbar)\cdot\xi(\gamma)
\]
for any $\gamma\in\Gamma_a$ as $\hbar\to0$ in the half-plane $\mathbb H_r$ for the constant term $\xi(\gamma)=1$. 

We consider the action of $\mathbb C^{\ast}$ on the space $\{\hbar^{-3}\varphi_a(z):\hbar\in\mathbb C,a\in S^m\}$ which rescales $a=\{a_1,\dots,a_m\}\in S^m$ with weights $(3,6,\dots,3(m-1), 3m)$, $\hbar$ with weight $m+3$, and $z$ with weight 3, i.e., 
\begin{equation*}
\begin{split}
c\cdot \left(\hbar^{-3}\varphi_a(z)\right)&=c\cdot \left(\hbar^{-3}(z^m+a_1z^{m-1}+\cdots+a_{m-1}z+a_m)\right)\\
&=(c^{m+3}\hbar)^{-3}(z^m+c^3a_1(c^3z)^{m-1}+\cdots+c^{3m-3}a_{m-1}(c^3z)+c^{3m}a_m).
\end{split}
\end{equation*}
for some $c\in\mathbb C^{\ast}$. Then, the $\mathbb C^{\ast}$-action on \eqref{sec4.1} can be naturally given by
\[
y'''(c\cdot z)+(c\cdot\left(\hbar^{-3}\varphi_a(z)\right)y(c\cdot z)=0,
\]
for some $c\in\mathbb C^{\ast}$. Note that the differential equation $y'''(z)+\hbar^{-3}\varphi_a(z)y(z)=0$ is unchanged under this $\mathbb C^{\ast}$ rescaling. Therefore, by applying this $\mathbb C^{\ast}$-action, we can assume, for the sake of simplication, that the ray $r$ is the positive real axis and that the WKB spectral network $\mathcal W(\varphi_a,0)$ is, hence, BPS-free. 

In the asymptotic region, the solution $y_k$ of \eqref{third order ODE} discussed in \Cref{sec4.1} is 
\begin{equation}
\label{asymptotic solution}
y_k\simeq\mathrm{exp}\left[-\frac{\delta_k}{\hbar}\int^z_{\infty\,l_{k}}x_{a_k}dz\right],
\end{equation}
where $a_k\in\{1,2,3\}$ is the sheet's labels of $\Sigma(\varphi_a)$, $\delta_k\in\{1,e^{\pm2\pi i/3}\}$ ensures the asymptotic behaviors that $y_k$ is subdominant in the sector containing $l_{k}$, and $\infty\,l_{k}$ represents the infinity along the direction $l_{k}$. When we substitute \eqref{asymptotic solution} into the Wronskian $W[y_i,y_j,y_k]$, one obtains that the leading term is in $1/\hbar$. Since the endpoints of the integration paths of the tree solutions are the same, $W[y_i,y_j,y_k]$ is independent of the final point which allows us to deform the contour associated with $Y_{r,\gamma}(\hbar)$. Based on the discussion in \cite[Section 3.2]{IKS}, the integration path of $Y_{r,\gamma}(\hbar)$ can be naturally identified with the class $\gamma\in\Gamma_a$. At the leading order in $\hbar^{-1}$ of the WKB approximation, we obtain the required asymptotic behavior:
\[
Y_{r,\gamma}(\hbar)\sim\mathrm{exp}(-Z_a(\gamma)/\hbar), \quad\hbar\to0.
\]

Based on the discussion above, we conclude with the following result:
\begin{proposition}
\label{hto0}
For each non-active ray $r\subset\mathbb C^{\ast}$ and each class $\gamma\in\Gamma_a$, where $a\in S^m$, the meromorphic function $Y_{r,\gamma}:\mathbb H_r\to\mathbb C^{\ast}$ satisfies
\[
\mathrm{exp}(Z_a(\gamma)/\hbar)\cdot Y_{r,\gamma}(\hbar)\to1,
\]
as $\hbar\to0$ in the half-plane $\mathbb H_r$.
\end{proposition}

\begin{remark}
According to \cite{IKS}, it is predicted that for any $\gamma\in\Gamma_a$, the logarithm of $\mathrm{log}\,Y_{r,\gamma}(\hbar)$ can be expressed as $\hbar^{-1}\Pi_{\gamma}(\hbar)$. Here, 
$\Pi_{\gamma}(\hbar)=\sum^{\infty}_{n=0}\hbar^n\Pi^{(n)}_{\gamma}$ represents the WKB period of \eqref{third order ODE} with respect to $\gamma$, as defined in \cite{IKS}. It is worth noting that $\Pi^{(0)}_{\gamma}=Z_a(\gamma)$. This prediction has been tested through numerical calculations for certain cases in \cite{IKKS,IKS}.
\end{remark}

%------------------------------------------------------------------------------------------------------------------------
%------------------------------------------------------------------------------------------------------------------------
%------------------------------------------------------------------------------------------------------------------------

\subsubsection{Behavior as $\hbar\to\infty$}
\label{sec.4.5}
The final step is to verify the condition (RH3). In fact we will establish an even stronger result: for a fixed point $a\in S^m$, the point $F_r(\hbar)$ converges to a well-defined limit point, meaning that any three distinct columns from the limit point are linearly independent. To demonstrate this, we will utilize the $\mathbb C^{\ast}$-action on $\hbar^{-3}\varphi_a(z)$ introduced earlier.
\begin{proposition}
\label{htoinfty}
For any point $a\in S^m$, $F_r(\hbar)\in V(3,m+3)$ has a well-defined limit as $\hbar\to\infty$ in a fixed half-plane. This limit is independent of the choice of $a\in S^m$ and is a fixed point of the $\mathbb Z/(m+3)\mathbb Z$-action.
\end{proposition}
\begin{proof}
As we discussed earlier, \eqref{third order ODE} is unchanged under the $\mathbb C^{\ast}$ rescaling. Therefore,
\[
M(\hbar^{-3}\varphi_a(z),0)=M(c\cdot (\hbar^{-3}\varphi_a(z)),0).
\]
Take $c^{m+3}\hbar=1$ and send $\hbar\to\infty$ in a fixed half-plane, it follows that 
\[
c\cdot (\hbar^{-3}\varphi_a(z))\to \varphi_0(z)=z^m.
\]
We claim that $F_r(\hbar)$ with respect to the polynomial $\varphi_0(z)$ converges to a well-defined limit in $V(3,m+3)$. Note that for any integer $0\le k\le m+2$,
\[
y_{0}(0,0;\hbar)=y_{k}(0,0;\hbar),
\]
\[
y'_{0}(0,0;\hbar)=\omega_1^{k}y'_{k}(0,0;\hbar),
\]
\[
y''_{0}(0,0;\hbar)=\omega_1^{2k}y''_{k}(0,0;\hbar).
\]
Since the Wronskian $W[y_i,y_j,y_k]$ is independent of $z$,
\begin{equation}
\begin{gathered}
W[y_i,y_j,y_k]=\mathrm{det}\begin{bmatrix}
y_{i}(0,0;\hbar)&y_{j}(0,0;\hbar)&y_{k}(0,0;\hbar)\\
y'_{i}(0,0;\hbar)&y'_{j}(0,0;\hbar)&y'_{k}(0,0;\hbar)\\
y''_{i}(0,0;\hbar)&y''_{j}(0,0;\hbar)&y''_{k}(0,0;\hbar)\\
\end{bmatrix}\\
=\mathrm{det}\begin{bmatrix}
y_{0}(0,0;\hbar)&y_{0}(0,0;\hbar)&y_{0}(0,0;\hbar)\\
\omega_1^{-i}y'_{0}(0,0;\hbar)&\omega_1^{-j}y'_{0}(0,0;\hbar)&\omega_1^{-k}y'_{0}(0,0;\hbar)\\
\omega_1^{-2i}y''_{m}(0,0;\hbar)&\omega_1^{-2j}y''_{m}(0,0;\hbar)&\omega_1^{-2k}y''_{0}(0,0;\hbar)\\
\end{bmatrix}\\
=y_0(0,0;\hbar)y'_0(0,0;\hbar)y''_0(0,0;\hbar)(\omega_1^{-i}-\omega_1^{-j})(\omega^{-j}-\omega_1^{-k})(\omega_1^{-k}-\omega_1^{-i})
\end{gathered}
\end{equation}
Since 
\[
W[y_1,y_2,y_3]=y_m(0,0;\hbar)y'_m(0,0;\hbar)y''_m(0,0;\hbar)(\omega_1^{-1}-\omega_1^{-2})(\omega_1^{-2}-\omega_1^{-3})(\omega_1^{-2}-\omega_1^{-1})\ne0,
\]
we obtain $W[y_i,y_j,y_k]\ne0$. Therefore, $F_r(\hbar)$ converges to a well-defined limit point in $V(3,m+3)$, which is independent of the choice of $a\in S^m$.

Note that the above $\mathbb C^{\ast}$-action induces an action of the $(m+3)$-th roots of unity $\mu_{m+3}\subset\mathbb C^{\ast}$, which leaves $\hbar$ invariant. In particular, $\omega_1=e^{{2\pi i}/{m+3}}\in\mu_{m+3}$. Since this action rescales $z$ by an element of $\mu_{m+3}$, it follows that this limit is a fixed point of the $\mathbb Z/(m+3)\mathbb Z$-action.
\end{proof}

%----------------------------------------------------------------------------------------------------------------
%----------------------------------------------------------------------------------------------------------------

\section{Polynomial cubic differentials and stability conditions}
\label{cubicstab}
In this section, we will utilize the cluster strutures and spectral coordinates derived from rank 3 WKB spectral networks $\mathcal W(\varphi,\vartheta)$ when zeros of $\varphi(z)$ are almost on a line in $\mathbb C$, as introduced in \Cref{sec3cluster}. By adopting the approach outlined in \cite{BS} and tailoring it to our case, we aim to establish a holomorphic embedding between the spaces of framed cubic differentials, associated with these WKB spectral networks, and stability conditions.

\subsection{The spaces of polynomial cubic differentials}
In this section, we will primarily introduce some spaces related to cubic differentials. We define a \emph{polynomial cubic differential} (cf. \cite{DM}) as a holomorphic differential on $\mathbb C$ in the form of $\varphi(z)dz^3$, where $\varphi(z)$ is a polynomial function. Let $\mathcal C(m)\simeq\mathbb C^{\ast}\times\mathbb C^m$ represent the vector space of polynomial cubic differentials of degree $m$ with nonzero leading coefficient. Unless specified otherwise, we always assume that $m\ge2$. Differentials $\varphi_1(z)dz^3$ and $\varphi_2(z)dz^3$ are considered equivalent if there exists an automorphism $f\in\mathrm{Aut}(\mathbb C)$ such that $f^{\ast}(\varphi_2(z)dz^3)=\varphi_{1}(z)dz^3$. Denote by $\mathcal {MC}(m)$ the space of equivalence classes of polynomial cubic differentials in $\mathcal C(m)$. According to \cite[Proposition 3.1]{DM}, the space $\mathcal{MC}(m)$, as a complex orbifold or a complex algebraic variety, can be described as the quotient of $\mathbb C^{m-1}$ by the action of $\mu_{m+3}=\mathbb Z/(m+3)\mathbb Z$, where the action is given by 
\[
(a_{m-2},a_{m-3},\cdots,a_0)\mapsto(\xi^{m+1}a_{m-2},\xi^ma_{m-3},\dots,\xi^3a_0)
\]
for any $\xi\in\mu_{m+3}$. 

The action of $\mathbb C$ on $\mathcal{MC}(m)$ is defined by the pullback of the standard $\mathbb C^{\ast}$ action, which rescales the cubic differentials, via the map $\mathbb C\to\mathbb C^{\ast}$ defined by $t\mapsto e^{-3\pi it}$.

\begin{definition}
\label{def.5.1}
Let $\Gamma$ be a free abelian group of rank $2m-2$. A $\Gamma$-framing of $\varphi(z)dz^3\in\mathcal{MC}(m)$ is an isomorphism of abelian groups
\[
\theta:\Gamma\to\Gamma(\varphi).
\] 
We denote by $\mathcal{MC}(m)^{\Gamma}$ the space of pairs $(\varphi(z)dz^3,\theta)$ consisting of a cubic differential $\varphi(z)dz^3\in\mathcal{MC}(m)$ and a $\Gamma$-framing $\theta$.
\end{definition}
This provides us with an unbranched cover
\[
\mathcal{MC}(m)^{\Gamma}\to\mathcal{MC}(m).
\]
For any $(\varphi(z)dz^3,\theta)\in\mathcal{MC}(m)^{\Gamma}$, the composition of the framing $\theta$ and the period map $Z_{\varphi}:\Gamma(\varphi)\to\mathbb C$ gives a group homomorphism $Z\circ\theta:\Gamma\to\mathbb C$. Thus, we obtain the map
\[
\pi_1:\mathcal{MC}(m)^{\Gamma}\to\mathrm{Hom}_{\mathbb Z}(\Gamma,\mathbb C),
\]
defined by $\pi_1(\varphi(z)dz^3,\theta):=Z_{\varphi}\circ\theta$. We refer to this map $\pi_1$ the \emph{extended period map}. Based on the observation in \Cref{dependence} and using a similar approach as in \cite[Theorem 4.12]{BS}, we can conclude the following:
\begin{proposition}
\label{local injection}
The space $\mathcal{MC}(m)^{\Gamma}$ is a complex manifold of dimension $m-1$ and the extended period map 
\[
\pi_1: \mathcal{MC}(m)^{\Gamma}\to\mathrm{Hom}_{\mathbb Z}(\Gamma,\mathbb C)
\]
is a local injection of complex manifolds.
\end{proposition}
Indeed, there is a subgroup $\Gamma_1\subset\Gamma$ of rank $m-1$ such that $\pi_1:\mathcal{MC}(m)^{\Gamma}\to\mathrm{Hom}(\Gamma_1,\mathbb C)$ is a local isomorphism. 

We introduce the open subspace $\mathcal{MC}(m)_{1}$ of $\mathcal{MC}(m)$ defined by
\[
\mathcal{MC}(m)_{1}:=\{\varphi(z)dz^3\in\mathcal{MC}(m): \text{zeros of }\, \varphi(z) \text{ are simple and almost on a line}\}.
\] 
Let us assume a fixed point $\phi=(\varphi(z)dz^3,\theta)\in\mathcal{MC}(m)^{\Gamma}$ and denote by $\mathcal{MC}(m)_{\ast}^{\Gamma}$ $(\mathcal{MC}(m)_{\ast})$ the connected component in $\mathcal{MC}(m)^{\Gamma}$ $(\mathcal{MC}(m))$ containing $\phi$ $(\varphi(z)dz^3)$. Note that for any two points $\phi_i=(\varphi_i(z)dz^3,\theta_i)\in\mathcal{MC}(m)_{\ast}^{\Gamma}, i=1,2$, there is continuous path $p:[0,1]\to\mathcal{MC}(m)_{\ast}$ which connects $\phi_1$ and $\phi_2$ such that the following diagram 
\begin{equation}
\xymatrix @C=5mm{
&\Gamma \ar[ld]_{\theta_1} \ar[rd]^{\theta_2}&\\
\Gamma(\varphi_1) \ar[rr]^{\mathrm{GM}_p} &&\Gamma(\varphi_2)\\
}
\end{equation}
commutes. We can now introduce a $\mathbb C$-action on $\mathcal{MC}(m)^{\Gamma}_{\ast}$ that is compatible with the $\mathbb C$-action on the space of stability conditions.
\begin{definition}
\label{def5.3}
For a framed cubic differential $\phi=(\varphi(z)dz^3,\theta)\in\mathcal{MC}(m)^{\Gamma}_{\ast}$, the new framed differential $\phi':=z\cdot\phi=(z\cdot\varphi(z)dz^3,z\cdot\theta)$ for $z\in\mathbb C$ is defined as follows:
\[
z\cdot\varphi(z)dz^3:=e^{-3i\pi z}\cdot\varphi(z)dz^3,\quad z\cdot\theta:=\mathrm{GM}_p\circ\theta,
\]
where $\mathrm{GM}_p$ represents the Gauss-Manin connection along the path $p(t):=e^{-3i\pi zt}\varphi(z)dz^3$. 
\end{definition}
In a similar manner, we can establish the notions of $\mathcal{MC}(m)_{1,\ast}, \mathcal{MC}(m)_{1,\ast}^{\Gamma}$, and the $\mathbb C$-action on $\mathcal{MC}(m)^{\Gamma}_{1,\ast}$. A point $\phi=\varphi(z)dz^3\in\mathcal{MC}(m)_1$ is called BPS-free if the WKB spectral network $\mathcal W(\varphi,0)$ is BPS-free. Equivalently, we have the BPS-free concept for the other spaces as well.

\begin{remark}
In fact, the space $\mathcal{MC}(m)^{\Gamma}$ can be understood as the space of WKB spectral networks $\mathcal W(\varphi,0)$, with $\varphi(z)$ being an arbitrary polynomial of degree $m$ in $\mathbb C$. Similarly, the space $\mathcal{MC}(m)^{\Gamma}_1$ corresponds to the space of WKB spectral networks $\mathcal W(\varphi,0)$, where zeros of $\varphi(z)$ are simple and almost on a line in $\mathbb C$. The latter space $\mathcal{MC}(m)^{\Gamma}_1$ is our main focus in this section.
\end{remark}
%----------------------------------------------------------------------------
%----------------------------------------------------------------------------
%----------------------------------------------------------------------------

\subsection{Categories from cubic differentials}
\label{CY3fcubic}
In this section, we will introduce the $CY_3$ categories from BPS-free framed polynomial cubic differentials in $\mathcal{MC}(m)_1$ by using the quiver with potential given in \cite{CZ}. We recommend \cite{DWZ,Gin,Kel1,Kel2,M2} for some basics of quivers with potentials, as well as the formation of $CY_3$ (Calabi-Yau of dimension 3) categories from quivers with potentials.

Let $\varphi(z)dz^3\in\mathcal{MC}(m)_{1}$ be a BPS-free framed polynomial cubic differential. According to \Cref{main_thm1} and \Cref{Postnikov diagram}, $(\varphi(z)dz^3,\theta)$ will give rise to a quiver $Q(\varphi,0)$, which can be obtained from a Postnikov diagram $P(\varphi,0)$. $Q(\varphi,0)$ is well-defined as we do not specify a particular seed corresponding to $Q(\varphi,0)$. For simplicity, we denote $Q(\varphi,0)$ by $Q(\varphi)$. Note that $Q(\varphi)$ is an iced quiver. We denote the \emph{principal part} of $Q({\varphi})$ by $Q(\varphi)^{pr}$, which is the full subquiver of $Q(\varphi)$ by removing iced vertices and the associated arrows. 

We can now describe $CY_3$ categories associated with quiver with potentials that arise from classes in $\mathcal{MC}(m)_{1}$. Following the constructions in \cite[Definition 2.1]{CZ}, there is a reduced potential $W(\varphi)^{pr}$ associated with $Q(\varphi)^{pr}$. By \cite[Theorem 2.4 and Theorem 2.14]{CZ}, we know that mutations of $(Q(\varphi)^{pr},W(\varphi)^{pr})$ are compatible with the geometric exchanges of $P(\varphi)$, and $(Q(\varphi)^{pr},W(\varphi)^{pr})$ is the unique nondegenerate quiver with potential with underlying quiver $\varphi(Q)$ up to right equivalence and mutation equivalence. $(Q(\varphi)^{pr}, W(\varphi)^{pr})$ is defined to be the quiver with potential associated with $\varphi(z)dz^3\in\mathcal{MC}(m)_1$. According to \cite{Kel2,KY}, we can obtain a $CY_3$ category $\mathrm{pvd}(\Gamma(Q(\varphi)^{pr}, W(\varphi)^{pr}))$, which has a canonical bounded t-structure whose heart 
\[
\mathcal H(Q(\varphi)^{pr},W(\varphi)^{pr})\cong\mathrm{mod}\,\mathcal J(Q(\varphi)^{pr},W(\varphi)^{pr}).
\]
In particular, the simple modules of this heart are in bijection with the vertices of $Q(\varphi)^{pr}$, and the spaces of extensions between them are determined by the arrows in $Q(\varphi)^{pr}$. Let $k$ be a vertex of $Q(\varphi)^{pr}$. Due to \cite{KY}, there exists a pair of $\mathbb C$-linear triangulated equivalences
\[
\Phi_{\pm}:\mathrm{pvd}(\Gamma(\mu_k(Q(\varphi)^{pr}, W(\varphi)^{pr})))\to\mathrm{pvd}(\Gamma(Q(\varphi)^{pr}, W(\varphi)^{pr})).
\]
Moreover, $\Phi_{\pm}$ induce tilts (cf. \cite{KQ2}) in the simple object $S_k\in\mathcal H(Q(\varphi)^{pr},W(\varphi)^{pr})$ in the sense that 
\[
\Phi_{\pm}(\mathcal H(\mu_k(Q(\varphi)^{pr},W(\varphi)^{pr}))=\mu_{S_k}^{\pm}(\mathcal H(Q(\varphi)^{pr},W(\varphi)^{pr})).
\]

%==========================================================
%==========================================================

\subsection{Stability conditions from cubic differentials}
\label{cubictostab}
As an application, we aim to construct a holomorphic map $K$ between th spaces of framed polynomial cubic differentials, associated with these WKB spectral networks, into spaces of stability conditions, adopting the approach of \cite{BS}. 

We refer to \cite{B3} for a comprehensive introduction to stability conditions on triangulated categories. A crucial property of stability conditions is summarized as follows. The stability condition space $\mathrm{Stab}(\mathcal D)$ on a triangulated category $\mathcal D$ has a complex structure and the projection map
\[
\pi:\mathrm{Stab}(\mathcal D)\to\mathrm{Hom}_{\mathbb Z}(K(\mathcal D),\mathbb C)
\]
defined by $(Z,\mathcal P)\mapsto Z$ is a local isomorphism (\cite[Theorem 1.2]{B3}). We also introduce two commuting group actions on $\mathrm{Stab}(\mathcal D)$. The first is the action of $\mathrm{Aut}(\mathcal D)$. For any $\Phi\in\mathrm{Aut}(\mathcal D)$ and $(Z,\mathcal P)\in\mathrm{Stab}(\mathcal D)$, we set $\Phi\cdot(Z,\mathcal P)=(Z\circ\Phi^{-1},\Phi(\mathcal P))$. The second is the $\mathbb C$-action. For any $z\in\mathbb C$ and $(Z,\mathcal P)\in\mathrm{Stab}(\mathcal D)$, we set $z\cdot(Z,\mathcal P)=(Z',\mathcal P')$ where
\[
Z'(E)=e^{-i\pi z}Z(E)
\]
for all $E\in\mathcal D$ and 
\[
\mathcal P'(\phi)=\mathcal P(\phi+\mathrm{Re}(z))
\]
for $\phi\in\mathbb R$.

We shall now fix a base-point $\phi_0=(\varphi_0(z)dz^3,\theta_0)\in\mathcal{MC}(m)^{\Gamma}_{1}$, supposing that $\mathcal W(\varphi_0,0)$ is BPS-free and generic. Denote by $\mathcal{MC}(m)^{\Gamma}_{1,\ast}$ the connected component in $\mathcal{MC}(m)^{\Gamma}_{1}$ containing $\phi_0$. For the sake of convenience, we will represent $\mathcal D$ as $\mathcal D(\phi_0)$ and use $\mathcal H$ to refer to the canonical heart of $\mathcal D$ throughout the upcoming discussion. We define $\mathrm{Stab}_0(\mathcal D)$ as the connected component in $\mathrm{Stab}(\mathcal D)$ containing stability conditions possessing heart $\mathcal H$. Let $\mathcal Aut_0(\mathcal D):=\mathrm{Aut}_{0}(\mathcal D)/\mathrm{Nil}_0(\mathcal D)$, where $\mathrm{Aut}_{0}(\mathcal D)$ is the subgroup of $\mathrm{Aut}(\mathcal D)$ that preserves $\mathrm{Stab}_0(\mathcal D)$ and $\mathrm{Nil}_0(\mathcal D)$ consists of automorphisms in $\mathrm{Aut}_{0}(\mathcal D)$ that act trivially on $\mathrm{Stab}_0(\mathcal D)$. We describe the action of $t\in\mathbb C$ on $\mathrm{Hom}_{\mathbb Z}(\Gamma,\mathbb C)$ as follows: 
\[
Z\mapsto e^{-i\pi t}\cdot Z,
\]
for $Z\in\mathrm{Hom}_{\mathbb Z}(\Gamma,\mathbb C)$. Now we can state the following result. 
\begin{theorem}
There is a holomorphic embedding $K$ between complex manifolds that fits into a commutative diagram
\begin{equation}
\label{diagram6.1}
\xymatrix @C=5mm{
\mathcal{MC}(m)^{\Gamma}_{1,\ast}\ar[rr]^{K\qquad\,\,} \ar[rd]_{\pi_1}&&\mathrm{Stab}_{0}(\mathcal D)/\mathcal Aut_0(\mathcal D) \ar[ld]^{\pi}\\
&\mathrm{Hom}_{\mathbb Z}(\Gamma,\mathbb C)&
}
\end{equation}
and which commutes with the $\mathbb C$-actions on both sides.
\end{theorem}

\begin{proof}
As mentioned in \cite[Proposition 11.3]{BS}, the space $\mathrm{Stab}_0(\mathcal D)/\mathcal Aut_0(\mathcal D)$ possesses a manifold structure due to the fact that the action of the group $\mathcal Aut_0(\mathcal D)$ on the connected component $\mathrm{Stab}_0(\mathcal D)$ is free. Now we establish a homomorphism 
\begin{equation}
\begin{split}
\nu_0:K(\mathcal D)&\to\Gamma(\varphi)\\
S_e&\mapsto [X_e],
\end{split}
\end{equation}
where $S_e$ is the simple object associated with $e\in Q(\varphi_0)^{pr}$, and $[X_e]\in\Gamma(\varphi)$ is the homology class corresponding to the spectral coordinate $X_e$ with respect to $e$ as described in \Cref{spectral coordinate1} and \eqref{spectral coordinate0}. According to \Cref{generator}, we conclude that $\nu_0$ is an isomorphism of abelian groups. In particular, $\nu_0$ takes the Euler form to the intersection form. We define $Z_{0}:=Z_{\varphi_0}\circ\nu_{0}$, where $Z_{\varphi_0}$ is the period map of $\varphi_0(z)$. Then $(\mathcal H, Z_{0})$ forms a unique stability condition, denoted by $\sigma(\phi_0)\in\mathrm{Stab}(\mathcal D)$, corresponding to $\phi_0$. Suppose $\phi$ is another BPS-free framed cubic differential in $\mathcal{MC}(m)^{\Gamma}_{1,\ast}$. Using a similar approach as \cite[Proposition 10.7]{BS}, we can construct an equivalence $\Psi:\mathcal{D}(\phi)\to\mathcal D$ such that $\nu_0\circ\psi=\theta_0\circ\theta^{-1}\circ\nu$, where $\psi$ is the map on the Grothendieck groups induced by $\Psi$. Then, we can derive a holomorphic map 
\begin{equation}
\begin{split}
K_0:B_0&\to\mathrm{Stab}_0(\mathcal D)/\mathcal Aut_0(\mathcal D)\\ 
\phi&\mapsto\Psi(\sigma(\phi)),
\end{split}
\end{equation}
where $B_0$ consists of all BPS-free points in $\mathrm{Stab}_0(\mathcal D)/\mathcal Aut_0(\mathcal D)$, which makes the diagram \eqref{diagram6.1} commute. We can further extend $K_0$ to another holomorphic map $K:\mathcal{MC}_{m,1}^{\Gamma}\to\mathrm{Stab}_0(\mathcal D)/\mathcal Aut_0(\mathcal D)$, which ensures that the diagram \eqref{diagram6.1} remains commutative, using the same argument as \cite[Proposition 11.3]{BS}. As a result of the preceding discussions, we infer that the map $K$ is injective. Moreover, since $\pi\circ K=\pi_1$, and both $\pi$ and $\pi_1$ possess the properties of local ismorphism and local injection, respectively, we obtain that $K$ is also an embedding. 
\end{proof}

Let $\mathcal{MC}(m)^{\Gamma}_{\ast,s}$ be an open subspace of $\mathcal{MC}(m)^{\Gamma}_{\ast}$ consisting of framed differentials $(\varphi(z)dz^3,\theta)$ where zeros of $\varphi(z)$ are distinct. We then propose the following conjecture as a potential avenue for future exploration: 
\begin{conjecture}
The map $K$ can be holomorphically extended to $\mathcal{MC}(m)^{\Gamma}_{\ast,s}$, denoted as $\widetilde{K}$:
\[
\widetilde{K}: \mathcal{MC}(m)^{\Gamma}_{\ast,s}\to\mathrm{Stab}_0(\mathcal D)/\mathcal Aut_0(\mathcal D).
\]
\end{conjecture}

%===================================
%===================================
%===================================

\section{A further example}
\label{further example}
In this section, we consider a specific WKB spectral network $\mathcal W(\varphi=\frac{1}{2}(-z^3+3z^2+2),\vartheta)$ referring to \cite{N1}. The geometric model of $\mathcal W(\varphi,\vartheta_0=0.1)$ is given in \Cref{23}.

Compared to cases where the zeros of polynomials are almost on a line in $\mathbb C$, $\mathcal W(\varphi,\vartheta)$ is in a different chamber, meaning that the BPS counts are different. In this case, the spectral curve $\Sigma(\varphi)$ is a 3-sheeted cover of $\mathbb C$, with branch points of order $3$ over $\mathrm{Zero}(\varphi)$. $\Sigma(\varphi)$ is a torus with 3-holes. Thus, $\Gamma(\varphi)$ is a lattice of rank 4, with a rank 2 intersection pairing. We present a choice of generators in \Cref{23}, where $\langle\gamma_1,\gamma_2\rangle=1$ and $\gamma_3,\gamma_4$ lie in the kernal of the pairing $\langle-,-\rangle$. The collection of compatible abelianization trees from $\mathcal W(\varphi,\vartheta_0)$ is given by 
\[
\mathcal{T}(\varphi,\vartheta_0)=\mathcal{T}^{\circ}(\varphi,\vartheta_0)\cup\{\mathrm{T}_{2,3,4,5,6,1},\mathrm{T}_{1,3,6},\mathrm{T}_{2,3,6},\mathrm{T}_{1,4,6}\},
\]
where $\mathrm{T}_{2,3,4,5,6,1}$ is shown in \Cref{ex:compatible tree m=3}.
\begin{figure}
\centering
\includegraphics[width=0.4\textwidth]{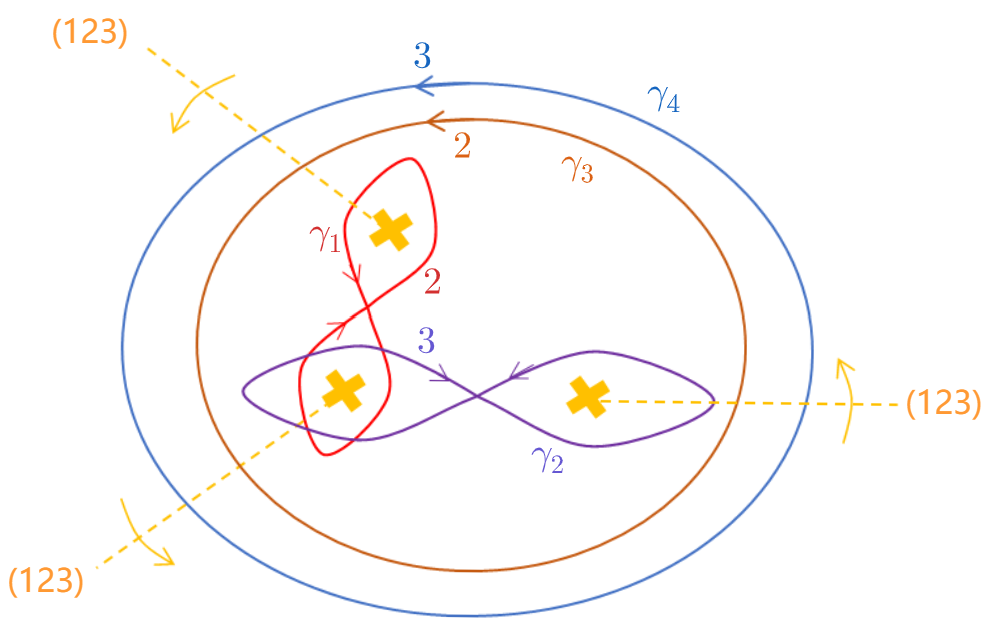}
\caption{Generators $\{\gamma_i\}_{i=1}^4$ for $\Gamma=H_1(\Sigma(\varphi),\mathbb Z)$.}
\label{22}
\end{figure}

\begin{figure}
\centering
\includegraphics[width=0.3\textwidth]{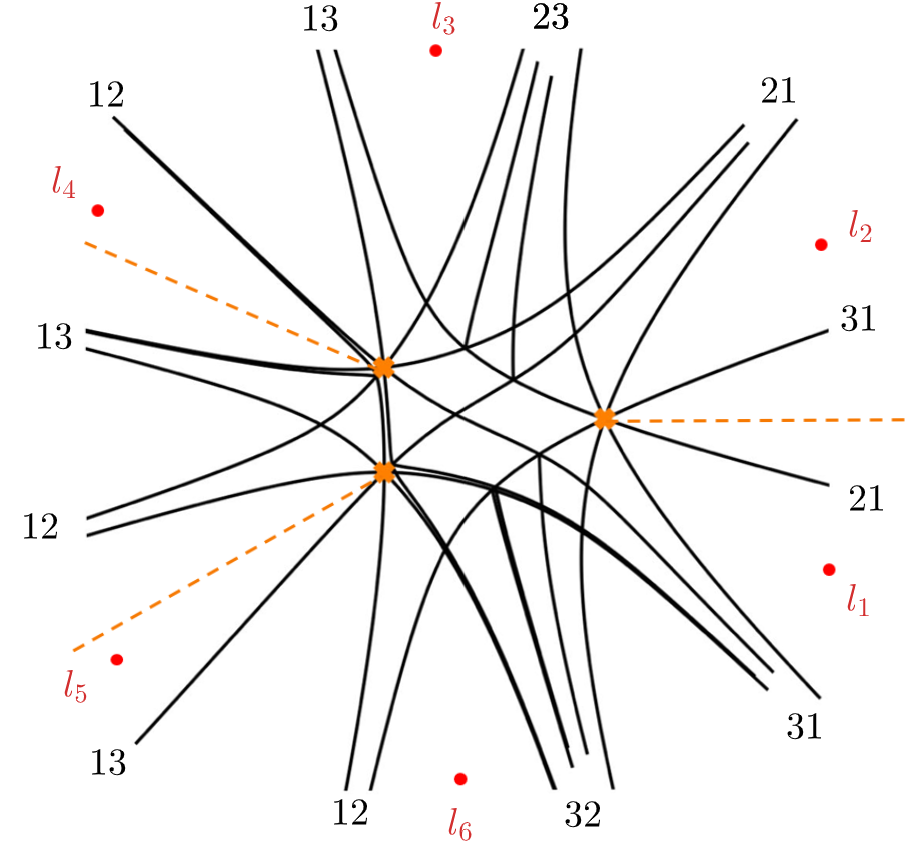}
\caption{The geometric model of $\mathcal W(\varphi,\vartheta_0)$.}
\label{23}
\end{figure}

By examining the spectral network $\mathcal W(\varphi,\vartheta)$ as $\vartheta$ varies, we can obtain the BPS counts $\Omega(\varphi,\gamma)$ as follows. For each unordered pair of distinct zeros $(z,z')$ of $\varphi$, we find $6$ finite BPS webs as shown in the left figure of \Cref{BPS counts}. Additionally, there are also $6$ more finite BPS webs that are three-string junctions involving all three zeros of $\varphi(z)$, as depicted in the right figure of \Cref{BPS counts}. This gives us a total of 24 nonzero BPS counts.

By varying $\vartheta$ from $\vartheta_0$ counterclockwise, we can obtain different collection of compatible abelianization trees corresponding to each BPS-free locus. It can be checked that each collection of compatible abelianization trees $\mathcal{T}(\varphi,\vartheta)$ from a BPS-free locus can provide a seed $(Q(\varphi,\vartheta),\mathcal{T}(\varphi,\vartheta))$ in $\mathbb C[Gr(3,6)]$, and each crossing of $\mathcal W(\varphi,\vartheta)$ corrresponds to a seed mutation. The first six collections of compatible abelianization trees, in order, are listed as follows :
\begin{table}[h!]
\begin{center}
\begin{tabular}{|m{5em}|m{20em}|}
\hline
$\vartheta$&$\mathcal{T}^{\circ}(\varphi,\vartheta)$\\
\hline
$\vartheta_0$&$\mathrm{T}_{2,3,4,5,6,1},\mathrm{T}_{1,3,6},\mathrm{T}_{2,3,6},\mathrm{T}_{1,4,6}$\\
\hline
$\vartheta_1$&$\mathrm{T}_{2,3,4,5,6,1},\mathrm{T}_{2,4,5},\mathrm{T}_{2,3,6},\mathrm{T}_{1,4,6}$\\
\hline
$\vartheta_2$&$\mathrm{T}_{2,4,6},\mathrm{T}_{2,4,5},\mathrm{T}_{2,3,6},\mathrm{T}_{1,4,6}$\\
\hline
$\vartheta_3$&$\mathrm{T}_{2,4,6},\mathrm{T}_{2,4,5},\mathrm{T}_{2,3,6},\mathrm{T}_{2,5,6}$\\
\hline
$\vartheta_4$&$\mathrm{T}_{2,4,6},\mathrm{T}_{3,4,6},\mathrm{T}_{2,3,6},\mathrm{T}_{2,5,6}$\\
\hline
$\vartheta_5$&$\mathrm{T}_{2,4,6},\mathrm{T}_{3,4,6},\mathrm{T}_{1,2,4},\mathrm{T}_{2,5,6}$\\
\hline
\end{tabular}
\end{center}
\caption{The first six collections of compatible abelianization trees $\mathcal{T}^{\circ}(\varphi,\vartheta)$}
\end{table}

The quiver $Q(\varphi,\vartheta_2)$ associated with $\mathcal W(\varphi,\vartheta_2)$ has been shown in \Cref{26}. Note that $\mathcal{T(\varphi,\vartheta_2)}$ is considered as extended clusters, and $\mathcal{T}^{\circ}(\varphi,\vartheta_2)$ is taken as cluster variables and others as coefficient variables. The exchange relations corresponding to mutations at $\mathrm{T}_{1,4,6},\mathrm{T}_{2,4,5},\mathrm{T}_{2,3,6}$ on $Q(\varphi,\vartheta_2)$ are written as follows:
\begin{enumerate}
\item $\mathrm{T}_{1,4,6}\mathrm{T}_{2,5,6}=\mathrm{T}_{1,5,6}\mathrm{T}_{2,4,6}+\mathrm{T}_{1,2,6}\mathrm{T}_{4,5,6}$;
\item $\mathrm{T}_{2,4,5}\mathrm{T}_{3,4,6}=\mathrm{T}_{3,4,5}\mathrm{T}_{2,4,6}+\mathrm{T}_{4,5,6}\mathrm{T}_{2,3,4}$;
\item $\mathrm{T}_{2,3,6}\mathrm{T}_{1,2,4}=\mathrm{T}_{1,2,3}\mathrm{T}_{2,4,6}+\mathrm{T}_{2,3,4}\mathrm{T}_{1,2,6}$.
\end{enumerate}
Actually, the collections of compatible abelianization trees $\mathcal{T}(\varphi,\vartheta_i),\, i=2,3,4,5$ can be embedded into Postnikov diagrams $P(\varphi,\vartheta_i)$ and the mutations described above correspond to geometric exchanges on the Postnikov diagrams. The quiver $Q(\varphi,\vartheta_0)$ associated to $\mathcal W(\varphi,\vartheta_0)$ is shown in \Cref{27}, which satisfies
\[
Q(\varphi,\vartheta_0)=\mu_{\mathrm{T}_{2,4,5}}\mu_{\mathrm T_{2,4,6}}(Q(\varphi,\vartheta_2)).
\]

The spectral coordinates associated with $\mathcal W(\varphi,\vartheta)$ can be defined similarly by equations \eqref{spectral coordinate0} and \eqref{spectral coordinate}. As an example, the spectral coordinates associated with $\mathcal W(\varphi,\vartheta_0)$ are written as follows:
\[
X_{\mathrm{T}_{1,4,6}}=\frac{A_{2,3,4,5,6,1}}{A_{1,5,6}A_{2,3,4}}=X_{\gamma_1},
\]
\[
X_{\mathrm{T}_{2,3,6}}=\frac{A_{2,3,4,5,6,1}}{A_{1,2,3}A_{4,5,6}}=X_{\gamma_1-\gamma_3},
\]
\[
X_{\mathrm{T}_{1,3,6}}=\frac{A_{1,2,6}A_{3,4,5}}{A_{2,3,4,5,6,1}}=X_{-\gamma_1+\gamma_3+\gamma_4},
\]
\[
X_{\mathrm{T}_{2,3,4,5,6,1}}=\frac{A_{1,3,6}A_{4,5,6}A_{2,3,4}}{A_{3,4,5}A_{2,3,6}A_{1,4,6}}=X_{-\gamma_2-\gamma_3-\gamma_4}.
\]
The bipartification on $\mathrm{T}_{2,3,4,5,6,1}$ is defined by assigning black color to those boundary nodes. Then, by performing some calculations, we have
\[
A_{2,3,4,5,6,1}=\mathrm{det}
\begin{bmatrix}
A_{2,3,6}&A_{4,5,6}\\
A_{1,2,3}&A_{1,4,5}\\
\end{bmatrix}.
\]

\begin{remark}
Lastly, it is worth noting that solutions to the Riemann-Hilbert problem arising from $\mathcal{W}(\varphi,\vartheta)$ can be constructed using a similar approach demonstrated in \Cref{solution}. Moreover, $\mathcal{W}(\varphi,\vartheta)$ can also be linked to a stability condition as formulated in \Cref{cubicstab}.
\end{remark}

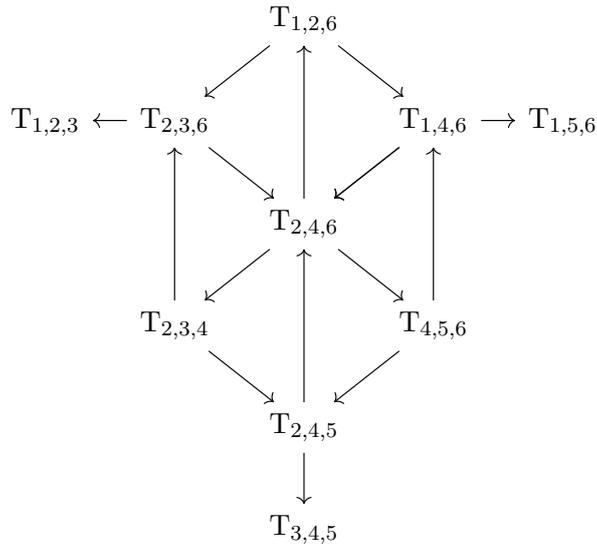
\begin{figure}
\[
\begin{tikzcd}[column sep=small]
&& {\mathrm T_{1,2,6}} \\
{\mathrm T_{1,2,3}} & {\mathrm T_{2,3,6}} && {\mathrm T_{1,4,6}} & {\mathrm T_{1,5,6}}\\
&&\mathrm{T}_{2,4,6}\\
& {\mathrm T_{2,3,4}} && {\mathrm T_{4,5,6}}\\
&&\mathrm{T}_{2,4,5}\\
&&\mathrm{T}_{3,4,5}
\arrow[from=1-3, to=2-4]
\arrow[from=3-3, to=1-3]
\arrow[from=5-3, to=3-3]
\arrow[from=1-3, to=2-2]
\arrow[from=2-4, to=2-5]
\arrow[from=4-2, to=2-2]
\arrow[from=2-2, to=3-3]
\arrow[from=2-4, to=3-3]
\arrow[from=3-3, to=4-4]
\arrow[from=4-2, to=5-3]
\arrow[from=2-4, to=3-3]
\arrow[from=2-2, to=2-1]
\arrow[from=4-4, to=2-4]
\arrow[from=3-3, to=4-2]
\arrow[from=5-3, to=6-3]
\arrow[from=4-4, to=5-3]
\end{tikzcd}
\]
\caption{The quiver $Q(\varphi,\vartheta_2)$ associated to $\mathcal W(\varphi,\vartheta_2)$.}
\label{26}
\end{figure}

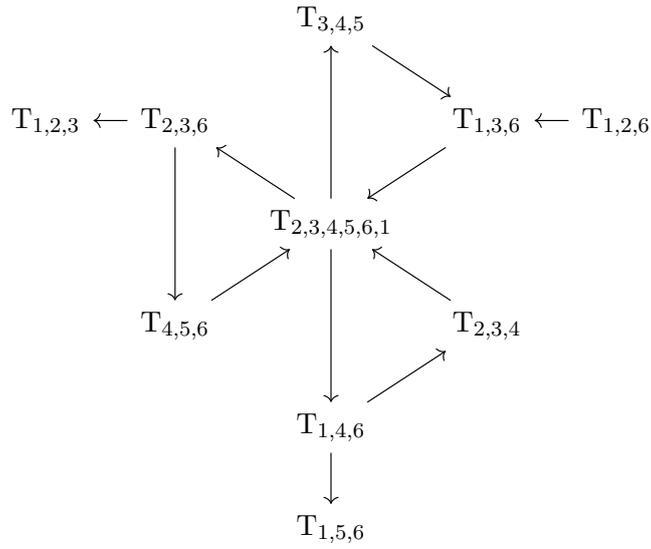
\begin{figure}
\[
\begin{tikzcd}[column sep=small]
&& {\mathrm T_{3,4,5}} \\
{\mathrm T_{1,2,3}} & {\mathrm T_{2,3,6}} && {\mathrm T_{1,3,6}} & {\mathrm T_{1,2,6}}\\
&&\mathrm{T}_{2,3,4,5,6,1}\\
& {\mathrm T_{4,5,6}} && {\mathrm T_{2,3,4}}\\
&&\mathrm{T}_{1,4,6}\\
&&\mathrm{T}_{1,5,6}
\arrow[from=1-3, to=2-4]
\arrow[from=3-3, to=1-3]
\arrow[from=3-3, to=5-3]
\arrow[from=2-5, to=2-4]
\arrow[from=2-2, to=4-2]
\arrow[from=3-3, to=2-2]
\arrow[from=4-4, to=3-3]
\arrow[from=2-4, to=3-3]
\arrow[from=2-2, to=2-1]
\arrow[from=4-2, to=3-3]
\arrow[from=5-3, to=6-3]
\arrow[from=5-3, to=4-4]
\end{tikzcd}
\]
\caption{The quiver $Q(\varphi,\vartheta_0)$ associated to $\mathcal W(\varphi,\vartheta_0)$.}
\label{27}
\end{figure}

%-----------------------------------------------------------------------------------------------------------------------
%-----------------------------------------------------------------------------------------------------------------------
%-----------------------------------------------------------------------------------------------------------------------
%-----------------------------------------------------------------------------------------------------------------------

\appendix
\section{Solutions to certain differential equations}
\label{App:differential equation}
In order to solve the Riemann-Hilbert problem, it is necessary to analyze the asymptotic properties of the solutions to certain differential equations. In the case of quadratic differentials, it is required to analyze the Schr\"{o}dinger equation, which has been well studied in \cite{S1, S2,W}. For our purpose, we need to consider the following differential equation:
\begin{equation}
\label{A.ODE}
y'''(z)-\varphi(z)y(z)=0,
\end{equation}
where $\varphi(z)=z^m+a_{1}z^{m-1}+\dots+a_{m-1}z+a_m$ is a polynomial of degree $m\ge2$. This equation is defined in a neighborhood of infinity. We can also express \eqref{A.ODE} in terms of a new variables $\xi$ such that $z=\xi^3$ and let $w=T^{-1}u$, where
\[
u=
\begin{bmatrix}
y(z)\\
y'(z)\\
y''(z)\\
\end{bmatrix},\,
T=\begin{bmatrix}
1&\omega&1\\
\xi^m&\omega^2\xi^m&\omega^2\xi^m\\
\xi^{2m}&\xi^{2m}&\omega\xi^{2m}
\end{bmatrix},\,\omega=e^{2\pi i/3}.
\]
Then \eqref{A.ODE} becomes 
\begin{equation}
\label{equivalent differential equation}
\frac{dw}{d\xi}=\xi^{m+2}C(\xi)w,
\end{equation}
where
\[
C(\xi)=\begin{bmatrix}
3&0&0\\
0&3\omega&0\\
0&0&3\omega^2\\
\end{bmatrix}+O(\xi^{-1}).
\]

The strategy for obtaining the asymptotic properties of solutions to equation \eqref{A.ODE} follows a similar approach taken in \cite{S1}.

\begin{theorem}
\label{mainthm4}
The differential equation $\eqref{A.ODE}$ admits a solution 
\[
y=y_m(z,a_1,\dots,a_{m})
\]
satisfying the following properties:
\begin{enumerate}
\item $y_m(z,a)$ is an entire function of $(z,a_1,\dots,a_{m})$;
\item $y_m(z,a)$ admits an asymptotic representation
\begin{equation}
\label{solution1}
y_m(z,a)\cong z^{r_m}(1+\sum_{N=1}^{\infty}B_{m,N}z^{-\frac{1}{3}N})\mathrm{exp}(E_m(z,a))
\end{equation}
uniformly on each compact set in $(a_1,\dots,a_m)$-space as $z$ tends to infinity in any closed subsector of the open sector 
\begin{equation}
\label{opensector}
-\frac{7\pi}{m+3}<\mathrm{arg}\,z<\frac{\pi}{m+3}
\end{equation}
where
\begin{equation}
\label{solution2}
E_m(z,a)=\frac{3}{m+3}z^{\frac{1}{3}(m+3)}+\sum^{m+2}_{N=1}A_{m,N}z^{\frac{1}{3}(m+3-N)} 
\end{equation}
and $r_{m},A_{m,N}$ and $B_{m,N}$ are polynomials in $(a_1,\dots,a_{m})$. 
\item the first and second derivatives $y'_m(z,a),y''_m(z,a)$ of $y_m(z,a)$ with respect to $z$ admit asymptotic representations
\begin{equation}
\label{solution'}
y'_m(z,a)\cong z^{\frac{1}{3}m+r_m}(1+\sum_{N=1}^{\infty}C_{m,N}z^{-\frac{1}{3}N})\mathrm{exp}(E_m(z,a))
\end{equation}
\begin{equation}
\label{solution''}
y''_m(z,a)\cong z^{\frac{2}{3}m+r_m}(1+\sum_{N=1}^{\infty}D_{m,N}z^{-\frac{1}{3}N})\mathrm{exp}(E_m(z,a))
\end{equation}
uniformly on each compact set in $(a_1,\dots,a_{m})$-space as $z$ tends to infinity in any closed subsector of the open sector \eqref{opensector}, where $C_{m,N}$ and $D_{m,N}$ are polynomials in $(a_1,\dots,a_{m})$.
\end{enumerate}
\end{theorem}
\begin{proof}
Let $p,q$, and $Y$ be unknown quantities and define
\[
w=
\begin{bmatrix}
1\\p\\q
\end{bmatrix}\mathrm{exp}(\int^{\xi}\eta^{m+2}Y(\eta)d\eta).
\]
Since 
\[
\frac{dw}{d\xi}=\left\{
\begin{bmatrix}
0\\\frac{dp}{d\xi}\\\frac{dq}{d\xi}
\end{bmatrix}+\xi^{m+2}Y(\xi)
\begin{bmatrix}
1\\p\\q
\end{bmatrix}
\right\}\mathrm{exp}(\int^{\xi}\eta^{m+2}Y(\eta)d\eta),
\]
we have 
\[
\begin{bmatrix}
\xi^{m+2}Y(\xi)\\
\frac{dp}{d\xi}+\xi^{m+2}Y(\xi)p\\
\frac{dq}{d\xi}+\xi^{m+2}Y(\xi)q
\end{bmatrix}=
\xi^{m+2}C(\xi)
\begin{bmatrix}
1\\p\\q
\end{bmatrix}.
\]
Representing the matrix $C(\xi)$ using the notation $(c_{ij}(\xi))_{3\times 3}$, we get 
\begin{equation}
\label{equation10}
\left\{
\begin{aligned}
Y(\xi)&=c_{11}(\xi)+c_{12}(\xi)p+c_{13}(\xi)q\\
\frac{dp}{d\xi}&=-\xi^{m+2}Y(\xi)p+\xi^{m+2}(c_{21}(\xi)+c_{22}(\xi)p+c_{23}(\xi)q)\\
\frac{dq}{d\xi}&=-\xi^{m+2}Y(\xi)q+\xi^{m+2}(c_{31}(\xi)+c_{32}(\xi)p+c_{33}(\xi)q)
\end{aligned}
\right.
\end{equation}
Equation \eqref{equation10} can be written as 
\begin{equation}
\label{equation11}
\left\{
\begin{aligned}
Y(\xi)&=c_{11}(\xi)+c_{12}(\xi)p+c_{13}(\xi)q\\
\frac{dp}{d\xi}&=\xi^{m+2}(c_{21}(\xi)+(c_{22}(\xi)-c_{11}(\xi))p-c_{12}(\xi)p^2+c_{23}(\xi)q-c_{13}(\xi)pq)\\
\frac{dq}{d\xi}&=\xi^{m+2}(c_{31}(\xi)+(c_{33}(\xi)-c_{11}(\xi))q-c_{13}(\xi)q^2+c_{32}(\xi)p-c_{13}(\xi)pq)
\end{aligned}
\right.
\end{equation}
To determine $p$ and $q$, we use the following lemma:
\begin{lemma}
\label{lem1}
The differential equation \eqref{equation11} has unique solutions $p(\xi),q(\xi)$ that satisfy the following conditions:
\begin{enumerate}
\item For any given $r\in(0,+\infty)$ and $\delta\in(0,\delta_0)$, where $\delta_0$ is a small positive number, there exists a positive number $N_{r,\delta}$ such that $p(\xi),q(\xi)$ are holomorphic with respect to $(\xi,a_1,\dots,a_{m})$ in the domain 
\begin{equation}
\label{domain0}
\left\{
\begin{aligned}
&|\xi|>N_{r,\delta},\\&|a_0|+\dots+|a_{m-1}|<r,\\
&|\frac{5\pi}{6}+(m+3)\mathrm{arg}\,\xi|<\frac{3\pi}{2}-\delta,\\
&|\frac{7\pi}{6}+(m+3)\mathrm{arg}\,\xi|<\frac{3\pi}{2}-\delta.
\end{aligned}
\right. 
\end{equation}
\item We have 
\begin{equation}
\label{equationpq}
p(\xi)\cong\hat p(\xi):=\sum^{\infty}_{N=1}p_N\xi^{-N},\quad q(\xi)\cong\hat q(\xi):=\sum^{\infty}_{N=1}q_N\xi^{-N}
\end{equation}
uniformly on each compact set in $(a_1,\dots,a_m)$-space as $\xi$ tends to infinity in any closed subsector of the open sector
\begin{equation}
\label{domain2}
-\frac{7\pi}{3(m+3)}<\mathrm{arg}\,\xi<\frac{\pi}{3(m+3)},
\end{equation}
where $p(\xi),q(\xi)$ are polynomials in $(a_1,\dots,a_{m})$.
\end{enumerate}
\end{lemma}
\begin{proof}
To construct a formal solution of \eqref{equation11}, we will express the last two equations in \eqref{equation11} as
\[
\xi^{-m-2}\frac{dp}{d\xi}=c_{21}(\xi)+(c_{22}(\xi)-c_{11}(\xi))p-c_{12}(\xi)p^2+c_{23}(\xi)q-c_{13}(\xi)pq),
\]
\[
\xi^{-m-2}\frac{dq}{d\xi}=c_{31}(\xi)+(c_{33}(\xi)-c_{11}(\xi))q-c_{13}(\xi)q^2+c_{32}(\xi)p-c_{13}(\xi)pq.
\]
By substituting \eqref{equationpq} into \eqref{equation11}, we obtain
\[
3(\omega-1)p_N=\alpha_N(p_1,\dots,p_{N-1},q_1,\dots,q_{N-1};a_1,\dots,a_{m}),
\]
\[
3(\omega^2-1)q_N=\beta_N(p_1,\dots,p_{N-1},q_1,\dots,q_{N-1};a_1,\dots,a_{m}),
\]
for $N\ge1$, where $\alpha_N,\beta_N$ are polynomials in $(p_1,\dots,p_{N-1},q_1,\dots,q_{N-1};a_1,\dots,a_{m})$. This obtains the existence of unique formal solutions, and its coefficients $p_N,q_N$ are polynomials in $(a_1,\dots,a_{m})$.

Define
\[
s=\frac{3(\omega-1)}{m+3}\xi^{m+3}.
\]
In the $s$-plane, let us consider a circle 
\begin{equation}
\label{circle}
|s|=M
\end{equation}
and two points
\begin{equation}
\begin{cases}
s_1=M\mathrm{exp}(i(\pi-\frac{1}{2}\delta))\\
s_2=M\mathrm{exp}(i(\frac{1}{2}\delta-\pi))\\
\end{cases}
\end{equation}
on the circle \eqref{circle}, where $\delta$ is a small positive constant and $M$ is a positive constant. Let $T_1$ and $T_2$ be the tangents to circle \eqref{circle} at $s=s_1$ and $s_2$, respectively. Let $T_1$ intersect the line: $\mathrm{arg}\,s=\frac{3\pi}{2}-\delta$ at $s=s_1'$, and let $T_2$ intersect the line: $\mathrm{arg}\,s=\delta-\frac{3\pi}{2}$ at $s=s_2'$. The Riemann domain $S_{\delta,M}$ is defined to contain the sector:
\[
|\mathrm{arg}\,s|<\frac{3\pi}{2}-\delta,\quad |s|>M'
\]
for a sufficiently large positive constant $M'$ and is bounded by the following curves:
\begin{equation}
\begin{cases}
s=-\tau\mathrm{exp}(i(\frac{3\pi}{2}-\delta))\quad (-\infty<\tau\le-|s_1'|),\\
s=s_1'+\tau\mathrm{exp}(i\mathrm{arg}(s_1-s_1'))\quad (0\le\tau\le|s_1-s_1'|),\\
s=Me^{i\tau} \quad (|\tau|\le\pi-\frac{1}{2}\delta),\\
s=s_2+\tau\mathrm{exp}(i(\delta-\frac{3\pi}{2}))\quad (|s_2'|\le\tau<+\infty).
\end{cases}
\end{equation}
For every point $s$ of $S_{\delta,M}$, we can define a straight line 
\begin{equation}
\label{sline}
\sigma=s+\tau e^{i\vartheta} \quad (0\le\tau<+\infty),
\end{equation}
such that the line $\eqref{sline}$ lies within $S_{\delta,M}$, and 
\[
|\vartheta|\le\frac{1}{2}(\pi-\delta).
\]
In a similar way, if we begin with another variable
\[
s'=\frac{3(\omega^2-1)}{m+3}\xi^{m+3},
\]
there is another Riemann domain $S'_{\delta, M}$ with respect to $s'$ and a straight line 
\begin{equation}
\label{sline2} 
\sigma'=s'+\tau e^{i\vartheta},
\end{equation}
contained in $S'_{\delta,M}$ with $|\vartheta|\le\frac{1}{2}(\pi-\delta)$. We now require another lemma:
\begin{lemma}\cite[Lemma 10.1]{S1}
\label{lem2}
Given positive constants $\delta$ and $\rho$, where $\delta$ is sufficiently small and $\rho$ is arbitrary. Then there exist constants $M_{\delta,\rho}$ and $L_0$ such that 
\[
\int_{\infty}^s|\sigma|^{-\rho}|e^{s\sigma}||\mathrm{d}\sigma|\le L_0|s|^{-\rho}
\]
for $s\in S_{\delta,M_{\delta,\rho}}$, where $L_0$ is independent of $\rho$, and the path of integration is the straight line \eqref{sline}.
\end{lemma}
We can now proceed to complete the proof of Lemma \ref{lem1}. Consider sectors $S_1, S_2$ in the $\xi$-plane and a domain $D_r$ in the $(a_1,\dots,a_{m})$-space defined respectively by 
\[
S_1: |\mathrm{arg}(\omega-1)+(m+3)\mathrm{arg}\xi|\le\frac{3\pi}{2},|\xi|\ge\Omega
\]
\[
S_2:|\mathrm{arg}(\omega^2-1)+(m+3)\mathrm{arg}\xi|\le\frac{3\pi}{2},|\xi|\ge\Omega
\]
and 
\[
D_r:|a_1|+\dots+|a_{m}|<r,
\]
where $\Omega$ is a fixed positive constant. It is known that there exist functions $\hat{p}_r(\xi),\hat{q}_r(\xi)$ which satisfy the following conditions:
\begin{enumerate}
\item $\hat{p}_r(\xi)$ is holomorphic in $(\xi,a_1,\dots,a_{m})$ for $\xi\in S_1$ and $(a_1,\dots,a_{m})\in D_r$;
\item $\hat{q}_r(\xi)$ is holomorphic in $(\xi,a_1,\dots,a_{m})$ for $\xi\in S_2$ and $(a_1,\dots,a_{m})\in D_r$;
\item $\hat p_r(\xi)\cong\hat p(\xi)$, $\frac{d\hat p_r(\xi)}{d\xi}\cong\frac{d\hat p(\xi)}{d\xi}$, as $\xi\to\infty$ in $S_1$ uniformly in $D_r$;
\item $\hat q_r(\xi)\cong\hat q(\xi)$, $\frac{d\hat q_r(\xi)}{d\xi}\cong\frac{d\hat q(\xi)}{d\xi}$, as $\xi\to\infty$ in $S_2$ uniformly in $D_r$.
\end{enumerate}
Set 
\[
p=\bar{p}+\hat{p}_r(\xi),\quad q=\bar q+\hat q_r(\xi).
\]
Then, 
\[
\frac{d\bar{p}}{d\xi}=\xi^{m+2}(\mu_{p,r}(\xi)+\lambda_{p,r}(\xi)\bar p+\nu_{p,r}(\xi)\bar p^2),
\]
\[
\frac{d\bar{q}}{d\xi}=\xi^{m+2}(\mu_{q,r}(\xi)+\lambda_{q,r}(\xi)\bar q+\nu_{q,r}(\xi)\bar q^2),
\]
in the common sector $S:=S_1\cap S_2$, where 
\begin{equation}
\begin{gathered}
\mu_{p,r}(\xi)=c_{21}(\xi)+(c_{22}(\xi)-c_{11}(\xi))\hat p_r-c_{12}(\xi)\hat p_r^2+c_{23}(\xi)q-c_{13}(\xi)\hat p_rq-\xi^{-m-2}\frac{d\hat p_r}{d\xi},\\
\mu_{q,r}(\xi)=c_{31}(\xi)+(c_{33}(\xi)-c_{11}(\xi))\hat q_r-c_{13}(\xi)\hat q_r^2+c_{32}(\xi)p-c_{13}(\xi)p\hat q_r-\xi^{-m-2}\frac{d\hat q_r}{d\xi},\\
\lambda_{p,r}(\xi)=(c_{22}(\xi)-c_{11}(\xi))-2c_{12}(\xi)\hat p_r-c_{13}(\xi)q,\\
\lambda_{q,r}(\xi)=(c_{33}(\xi)-c_{11}(\xi))-2c_{13}(\xi)\hat q_r-c_{13}(\xi)p,\\
\nu_{p,r}(\xi)=-c_{12}(\xi),\\
\nu_{q,r}(\xi)=-c_{13}(\xi).\\
\end{gathered}
\end{equation}
We have 
\begin{equation}
\label{asymp1}
\begin{cases}
\mu_{p,r}\cong0,\\
\mu_{q,r}\cong0,\\
\nu_{p,r}=O(\xi^{-1}),\\
\nu_{q,r}=O(\xi^{-1}),\\
\end{cases}
\quad (\xi\to\infty \text{ in } S \text{ uniformly in } D_r)
\end{equation}
The reason for the asymptotic property of $\mu_{p,r},\mu_{q,r}$ is that $\hat p(\xi),\hat q(\xi)$ are formal solutions of \eqref{equation11}.

We shall now construct solutions $\bar p,\bar q$ such that 
\begin{equation}
\label{asymptotic}
\bar p(\xi)\cong 0, \bar q(\xi)\cong 0,
\end{equation}
as $\xi\to\infty$ in $S^{\circ}$ uniformly in $D_r$, where $S^{\circ}$ denotes the interior of $S$. Consider the following equations
\begin{equation}
\label{equation15}
\begin{split}
\bar p(\xi)=\int^{\xi}_{\infty}\xi^{m+2}(\mu_{p,r}(\eta)+\widetilde{\lambda}_{p,r}(\eta)\bar p(\eta)+\nu_{p,r}(\eta)\bar p(\eta)^2)\mathrm{exp}(3(\omega-1)\frac{\xi^{m+3}-\eta^{m+3}}{m+3})d\eta,
\end{split}
\end{equation}
\begin{equation}
\label{equation16}
\begin{split}
\bar q(\xi)=\int^{\xi}_{\infty}\xi^{m+2}(\mu_{q,r}(\eta)+\widetilde{\lambda}_{q,r}(\eta)\bar q(\eta)+\nu_{q,r}(\eta)\bar q(\eta)^2)\mathrm{exp}(3(\omega^2-1)\frac{\xi^{m+3}-\eta^{m+3}}{m+3})d\eta,
\end{split}
\end{equation}
where
\[
\widetilde{\lambda}_{p,r}(\eta)=\lambda_{p,r}(\eta)-3(\omega-1), \quad \widetilde{\lambda}_{q,r}(\eta)=\lambda_{q,r}(\eta)-3(\omega^2-1).
\]
Let
\[
s=\frac{3(\omega-1)}{m+3}\xi^{m+3},\sigma=\frac{3(\omega-1)}{m+3}\eta^{m+3},
\]
\[
s'=\frac{3(\omega^2-1)}{m+3}\xi^{m+3},\sigma'=\frac{3(\omega^2-1)}{m+3}\eta^{m+3}.
\]
Then equations \eqref{equation15},\eqref{equation16} become 
\begin{equation}
\label{equation12}
\begin{split}
\bar p(\xi)=\frac{1}{3(\omega-1)}\int^{s}_{\infty}(\mu_{p,r}(\eta)+\widetilde{\lambda}_{p,r}(\eta)\bar p(\eta)+\nu_{p,r}(\eta)\bar p(\eta)^2)e^{s-\sigma}d\sigma,
\end{split}
\end{equation}
\begin{equation}
\label{equation13}
\begin{split}
\bar q(\xi)=\frac{1}{3(\omega^2-1)}\int^{s'}_{\infty}(\mu_{q,r}(\eta)+\widetilde{\lambda}_{q,r}(\eta)\bar q(\eta)+\nu_{q,r}(\eta)\bar q(\eta)^2)e^{s'-\sigma'}d\sigma'.
\end{split}
\end{equation}
Consider equations \eqref{equation12}, \eqref{equation13} in the domain
\begin{equation}
\label{domain3}
\widetilde{S_{\delta,M}}:=\{\xi\,|\,s\in S_{\delta,M},s'\in S'_{\delta,M}\},\quad (a_1,\dots,a_{m})\in D_r,
\end{equation}
where constants $r$ and $\delta$ are fixed, and the constant $M$ will be specified later. The paths of integration are given by the straight lines \eqref{sline}, \eqref{sline2} respectively.

Assume that $|w_1(\xi)|,|w_2(\xi)|\le|\xi|^{-1}$ in \eqref{domain3}, and set
\[
v_p(\xi)=\frac{1}{3(\omega-1)}\int^{s}_{\infty}(\mu_{p,r}(\eta)+{\widetilde{\lambda}}_{p,r,w_2}(\eta)w_1(\eta)+\nu_{p,r}(\eta)w_1(\eta)^2)e^{s-\sigma}d\sigma,
\]
\[
v_q(\xi)=\frac{1}{3(\omega^2-1)}\int^{s'}_{\infty}(\mu_{q,r}(\eta)+{\widetilde{\lambda}}_{q,r,w_1}(\eta)w_2(\eta)+\nu_{q,r}(\eta)w_2(\eta)^2)e^{s'-\sigma'}d\sigma',
\]
where 
\[
{\widetilde{\lambda}}_{p,r,w_2}(\xi)=(c_{22}(\xi)-c_{11}(\xi))-2c_{12}(\xi)\hat p_r-c_{13}(\xi)w_2(\xi), 
\]
\[
{\widetilde{\lambda}}_{q,r,w_1}(\xi)=(c_{33}(\xi)-c_{11}(\xi))-2c_{13}(\xi)\hat q_r-c_{13}(\xi)w_1(\xi).
\]
\begin{sloppypar}
By observing the asymptotic properties of $\mu_{p,r}(\xi),\mu_{q,r}(\xi),{\widetilde{\lambda}}_{p,r,w_2}(\eta),{\widetilde{\lambda}}_{q,r,w_2}(\eta),\nu_{p,r}(\xi), \nu_{q,r}(\xi)$, we have
\[
|\mu_{p,r}(\xi)|\le L|\xi|^{-2}, |{\widetilde{\lambda}}_{p,r,w_2}(\eta)|\le L|\xi|^{-1}, |\nu_{p,r}(\xi)|\le L|\xi|^{-1},
\]
\[
|\mu_{q,r}(\xi)|\le L|\xi|^{-2}, |{\widetilde{\lambda}}_{q,r,w_1}(\eta)|\le L|\xi|^{-1}, |\nu_{q,r}(\xi)|\le L|\xi|^{-1},
\]
\end{sloppypar}
in \eqref{domain3}, for some positive constant $L$. Then 
\[
|v_p(\xi)|\le\frac{C_1}{M'_1}\int^s_{\infty}|\eta|^{-1}|e^{s-\sigma}||d\sigma|,
\]
\[
|v_p(\xi)|\le\frac{C_2}{M_2'}\int^s_{\infty}|\eta|^{-1}|e^{s'-\sigma'}||d\sigma'|,
\]
where
\[
C_1=\frac{L}{3|\omega-1|}(2+1/M'_1),\quad M'_1=(\frac{m+3}{3|\omega-1|}M)^{1/(m+3)},
\]
\[
C_1=\frac{L}{3|\omega^2-1|}(2+1/M'_2),\quad M'_2=(\frac{m+3}{3|\omega^2-1|}M)^{1/(m+3)}.
\]
By Lemma \ref{lem2}, we obtain $|v_p(\xi)|\le|\xi|^{-1},|v_q(\xi)|\le|\xi|^{-1}$ in \eqref{domain3} if 
\[
M>M_{\delta,1/(m+3)},\quad C_1L_0\le M'_1,\quad C_2L_0\le M'_2.
\]
Assume that $|w_1^1(\xi)|,|w_1^2(\xi)|,|w_2^1(\xi)|,|w_2^2(\xi)|\le |\xi|^{-1}$, and let
\[
v_p(\xi)=\frac{1}{3(\omega-1)}\int^{s}_{\infty}({\widetilde{\lambda}}_{p,r,w_1^2}(\eta)w_1^1(\eta)-\widetilde{\lambda}_{p,r,w_2^2}(\eta)w_2^1(\eta))+\nu_{p,r}(\eta)(w_1^1(\eta)^2-w_2^1(\eta)^2))e^{s-\sigma}d\sigma,
\]
\[
v_q(\xi)=\frac{1}{3(\omega^2-1)}\int^{s'}_{\infty}({\widetilde{\lambda}}_{q,r,w_1^1}(\eta)w_1^2(\eta)-\widetilde{\lambda}_{q,r,w_2^1}(\eta)w_2^2(\eta))+\nu_{q,r}(\eta)(w_1^2(\eta)^2-w_2^2(\xi)^2))e^{s'-\sigma'}d\sigma'.
\]
Then, we have 
\[
|v_p(\xi)|\le\frac{L_0L(1+3/M')}{3|\omega-1|M'}||w_1^1(\xi)-w_2^1(\xi)||+\frac{L_0L}{3|\omega-1|M'^2}||w_1^2(\xi)-w_2^2(\xi)||,
\]
\[
|v_q(\xi)|\le\frac{L_0L(1+3/M')}{3|\omega^2-1|M'}||w_1^2(\xi)-w_2^2(\xi)||+\frac{L_0L}{3|\omega^2-1|M'^2}||w_1^1(\xi)-w_2^1(\xi)||
\]
in \eqref{domain3}, where $||\cdot||$ denotes $\mathrm{sup}|\cdot|$ in \eqref{domain3}.

Thus we can obtain the existence of solutions $\bar p(\xi),\bar q(\xi)$ such that $|\bar p(\xi)|\le|\xi|^{-1}, |\bar q(\xi)|\le|\xi|^{-1}$ in \eqref{domain3}. To prove that $\bar p(\xi), \bar q(\xi)$ satisfy the asymptotic condition \eqref{asymptotic}, we consider functions $w_1(\xi),w_2(\xi)$ such that $|w_1(\xi)|, |w_2(\xi)|\le K_{N}|\xi|^{-N}$ in \eqref{domain3}, and let 
\[
v_p(\xi)=\frac{1}{3(\omega-1)}\int^{s}_{\infty}(\mu_{p,r}(\eta)+{\widetilde{\lambda}}_{p,r,w_2}(\eta)w_1(\eta)+\nu_{p,r}(\eta)w_1(\eta)^2)e^{s-\sigma}d\sigma,
\]
\[
v_q(\xi)=\frac{1}{3(\omega^2-1)}\int^{s'}_{\infty}(\mu_{q,r}(\eta)+{\widetilde{\lambda}}_{q,r,w_1}(\eta)w_2(\eta)+\nu_{q,r}(\eta)w_2(\eta)^2)e^{s'-\sigma'}d\sigma'.
\]
From the asymptotic property \eqref{asymp1} of $\mu_{p,r}(\xi),\mu_{q,r}(\xi)$, we have $|\mu_{p,r}(\xi)|,|\mu_{q,r}(\xi)|\le L_N|\xi|^{-N-1}$ in \eqref{domain3}, where $L_N$ is a positive constant. Then, we have 
\[
|v_p(r)|\le\frac{L_N+LK_N+LK_N^2}{3|\omega-1|}\int^s_{\infty}|\eta|^{-N-1}|e^{s-\sigma}||d\sigma|,
\]
\[
|v_q(r)|\le\frac{L_N+LK_N+LK_N^2}{3|\omega^2-1|}\int^{s'}_{\infty}|\eta|^{-N-1}|e^{s'-\sigma'}||d\sigma'|.
\]
Hence, we can obtain $|v_p(\xi)|\le K'|\xi|^{-N-1}, |v_q(\xi)|\le K''|\xi|^{-N-1}$ in \eqref{domain3} by Lemma \ref{lem2}.
Set 
\begin{equation}
\label{solution3}
p(\xi)=\bar p(\xi)+\hat p_r(\xi),\quad q(\xi)=\bar q(\xi)+\hat q_r(\xi),
\end{equation}
which are solutions of \eqref{equation11}. If solutions \eqref{solution3} are independent of $r$ and $\delta$, then the proof of Lemma \ref{lem1} will be completed. Let $p_1(\xi)$ and $q_1(\xi)$ be solutions of \eqref{equation11} that satisfy all the requirements in the statement of Lemma \ref{lem1} in the domain 
\begin{equation}
\label{domain4}
\widetilde{S_{\delta_1,M_1}}, (a_1,\dots,a_m)\in D_{r_1},
\end{equation}
and let $p_2(\xi),q_2(\xi)$ be such solutions in the domain 
\begin{equation}
\label{domain5}
\widetilde{S_{\delta_2,M_2}}, (a_1,\dots,a_m)\in D_{r_2}.
\end{equation}
If we choose suitable values for $\delta,M$ and $r$, the domain 
\begin{equation}
\label{domain6}
\widetilde{S_{\delta,M}}, (a_1,\dots, a_m)\in D_r
\end{equation}
will be contained in \eqref{domain4} and \eqref{domain5}. Set 
\[
u(\xi)=p_1(\xi)-p_2(\xi), \quad v(\xi)=q_1(\xi)-q_2(\xi).
\]
Then
\[
\frac{du(\xi)}{d\xi}=\xi^{m+2}(J_1(\xi)u(\xi)+K_1(\xi)v(\xi)),
\]
\[
\frac{dv(\xi)}{d\xi}=\xi^{m+2}(J_2(\xi)u(\xi)+K_2(\xi)v(\xi)),
\]
where 
\[
J_1(\xi)=c_{22}(\xi)-c_{11}(\xi)-c_{12}(\xi)(p_1(\xi)+p_2(\xi))-c_{13}(\xi)q_1(\xi),
\]
\[
J_2(\xi)=c_{32}(\xi)-c_{13}(\xi)q_1(\xi),
\]
\[
K_1(\xi)=c_{23}(\xi)-c_{13}(\xi)p_2(\xi),
\]
\[
K_2(\xi)=c_{33}(\xi)-c_{11}(\xi)-c_{13}(\xi)(q_1(\xi)+q_2(\xi))-c_{13}(\xi)p_2(\xi).
\]
We shall analyze $u(\xi),v(\xi)$ in domain \eqref{domain6}. Since $p_1(\xi)$ $(q_1(\xi))$ and $p_2(\xi)$ $(q_2(\xi))$ satisfy the same asymptotic condition in domain \eqref{domain6}, we have
\begin{equation}
\label{asyp2}
u(\xi)\cong0,\,v(\xi)\cong0, \quad (\text{ as } \xi\to\infty \text{ in } \widetilde{S_{\delta,M}} \text{ uniformly in } D_r).
\end{equation}
We also have 
\[
J_1(\xi)=\omega-1+O(\xi^{-1}), \quad J_2(\xi)=O(\xi^{-1}),
\]
\[
K_1(\xi)=O(\xi^{-1}), \quad K_2(\xi)=\omega^2-1+O(\xi^{-1}),
\]
as $\xi\to\infty$ in $\widetilde{S_{\delta,M}}$ uniformly in $D_r$.
Based on these observations, we can get that $u(\xi)\equiv0, v(\xi)\equiv0$ in \eqref{domain6}, which completes the proof of Lemma \ref{lem1}.
\end{proof} 
Now we can set 
\[
Y(\xi)=c_{11}(\xi)+c_{12}(\xi)p+c_{13}(\xi)q.
\]
Then, in domain \eqref{domain0}, $Y(\xi)$ is a holomorphic function, and
\[
Y(\xi)\cong 3+\sum^{\infty}_{N=1}Y_N\xi^{-N}
\]
holds uniformly on each compact set in $(a_1,\dots,a_{m})$-space as $\xi$ tends to infinity in any closed subsector of \eqref{domain2}, where $Y_N$ are polynomials in $(a_1,\dots,a_{m})$. Let 
\[
\widetilde{Y}(\xi)=Y(\xi)-(3+\sum_{N=1}^{m+3}Y_N\xi^{-N})
\]
and 
\[
E(\xi)=\xi^{Y_{m+3}}\mathrm{exp}(\frac{3}{m+3}\xi^{m+3}+\sum^{m+2}_{N=1}\frac{Y_N}{m+3-N}\xi^{m+3-N}).
\]
Then set 
\[
w=\begin{bmatrix}
1\\p(\xi)\\q(\xi)
\end{bmatrix}E(\xi)\mathrm{exp}(\int^{\xi}_{\infty}\eta^{m+2}\tilde{Y}(\eta)d\eta)
\]
and 
\begin{equation}
\label{equation14}
u(z)=
\begin{bmatrix}
1&\omega&1\\
\xi^m&\omega^2\xi^m&\omega^2\xi^m\\
\xi^{2m}&\xi^{2m}&\omega\xi^{2m}
\end{bmatrix}
\begin{bmatrix}
1\\p(\xi)\\q(\xi)
\end{bmatrix}E(\xi)\mathrm{exp}(\int^{\xi}_{\infty}\eta^{m+2}\widetilde{Y}(\eta)d\eta).
\end{equation}
Let $\Psi(z)$ be the $3$-by-$3$ matrix such that 
\[
\Psi'(z)=A(z)\Psi(z), \Psi(0)=\mathrm{Id}.
\]
Note that $\Psi(z)$ and $\Psi(z)^{-1}$ are entire in $(z,a_1,\dots,a_{m})$, and 
\[
u(z)=\Psi(z)\Psi(z_0)^{-1}u(z_0).
\]
Let $(a_1^0,\dots,a_{m}^0)$ be any fixed, and let $V$ be a small neighborhood of $(a_1^0,\dots,a_{m}^0)$ in the $(a_1,\dots,a_{m})$-space. Choose $z_0$ so that $(z_0^{\frac{1}{3}},a_1,\dots,a_{m})$ is in \eqref{domain0} for every $(a_1,\dots,a_{m})\in V$, which means that $u(z_0)$ is holomorphic in $V$. Hence $u(z)$ is entire in $(z,a_1,\dots,a_{m})$.
From \eqref{equation14} and the definition of $E(\xi)$, we have
\begin{equation}
\begin{cases}
r_m=\frac{1}{3}Y_{m+3}\\
E_m(z,a)=\frac{3}{m+3}z^{\frac{1}{3}(m+3)}+\sum^{m+2}_{N=1}\frac{Y_N}{m+3-N}z^{\frac13(m+3-N)}.
\end{cases}
\end{equation}
%------------------------------------------------------------------------------ %-------------------------------------------------------------------------------- 
%------------------------------------------------------------------------------
%--------------------------------------------------------------------------------
Set $F=z^{-r_m}\mathrm{exp}(-E_m(z,a))y_m(z,a)$. Then
\[
F'=(-r_mz^{-1}-E'_m(z,a)+y'_m(z,a)/y_m(x,a))F,
\]
and hence
\[
(F/y_m(z,a))y'_m(z,a)=F'+r_mz^{-1}F+E'_m(z,a)F.
\]
Note that \eqref{solution1} means
\[
F\cong 1+\sum^{\infty}_{N=1}B_{m,N}z^{-\frac13N}.
\]
Therefore, we have 
\[
F'\cong\sum^{\infty}_{N=1}(-\frac13N)B_{m,N}z^{\frac13N-1}.
\]
Hence
\[
(F/y_m(z,a))y'_m(z,a)\cong z^{\frac13m}(1+\sum^{\infty}_{N=1}\tilde{C}_{m,N}z^{-\frac13N}),
\]
where $\tilde C_{m,N}$ are polynomials in $(a_1,\dots,a_m)$. Furthermore, 
\[
F/y_m(z,a)=z^{-r_m}\mathrm{exp}(-E_m(z,a)).
\]
Thus we obtain \eqref{solution'}. Similarly, \eqref{solution''} can also be obtained. As a result, Theorem \ref{mainthm4} has been fully demonstrated.
\end{proof}

\end{document}